\documentclass{article}
\usepackage{hyperref}
\hypersetup{backref,colorlinks=true,citecolor=cyan}
\usepackage[english]{babel}
\usepackage{mathtools}
\usepackage{amsmath,amssymb,graphicx,textcomp,enumerate,bbm,xcolor,latexsym,theorem}
\usepackage{mathrsfs}
\usepackage[ruled,algosection,noline]{algorithm2e}
\definecolor{darkgreen}{rgb}{0 255 0}

\newcommand{\assign}{:=}
\newcommand{\cdummy}{\cdot}

\newcommand{\nocomma}{}
\newcommand{\noplus}{}

\newcommand{\tmem}[1]{{\em #1\/}}
\newcommand{\tmmathbf}[1]{\mathbf{#1}}
\renewcommand{\mathbbm}[1]{\mathbb{#1}}
\newcommand{\tmop}[1]{\ensuremath{\operatorname{#1}}}
\newcommand{\tmstrong}[1]{\textbf{#1}}
\newcommand{\tmtextbf}[1]{{\bfseries{#1}}}
\newcommand{\tmtextit}[1]{{\itshape{#1}}}

\newcommand{\udots}{{\mathinner{\mskip1mu\raise1pt\vbox{\kern7pt\hbox{.}}\mskip2mu\raise4pt\hbox{.}\mskip2mu\raise7pt\hbox{.}\mskip1mu}}}
\newenvironment{enumeratealpha}{\begin{enumerate}[a{\textup{)}}] }{\end{enumerate}}
\newenvironment{enumeratenumeric}{\begin{enumerate}[1.] }{\end{enumerate}}
\newenvironment{itemizedot}{\begin{itemize} }{\end{itemize}}
\newenvironment{itemizeminus}{\begin{itemize} }{\end{itemize}}
\newenvironment{proof}{\noindent\textbf{Proof.\ }}{\hspace*{\fill}$\Box$\medskip}

\newtheorem{theorem}{Theorem}[section]
\newtheorem{corollary}[theorem]{Corollary}
\newtheorem{definition}[theorem]{Definition}
{\theorembodyfont{\rmfamily}\newtheorem{example}[theorem]{Example}}
\newtheorem{lemma}[theorem]{Lemma}
\newtheorem{proposition}[theorem]{Proposition}
{\theorembodyfont{\rmfamily}\newtheorem{remark}[theorem]{Remark}}
\def\bdb{\mathbf{b}}  
\def\bde{\mathbf{e}}  
\def\bdf{\mathbf{f}}  
\def\bdl{\mathbf{l}} 
\def\bdu{\mathbf{u}}  
\def\bdv{\mathbf{v}}  
\def\bdx{\mathbf{x}} 
\def\bdy{\mathbf{y}} 
\def\bdz{\mathbf{z}} 

\newcommand{\fsupp}{\ell^{0}}

\usepackage{dutchcal}
\def\PolExp{\mathcal{PolExp}}
\def\PolyLog{\mathcal{PolyLog}}

\begin{document}

\title{Polynomial-exponential decomposition from moments}

\author{Bernard Mourrain\\
  Universit\'e C\^ote d'Azur, Inria, {\sc aromath}, France\\
\texttt{bernard.mourrain@inria.fr}}

\date{\today}

\maketitle

\begin{abstract}
  We analyze the decomposition problem of multivariate
  polynomial-exponential functions from their truncated series and
  present new algorithms to compute their decomposition.
  
  Using the duality between polynomials and formal power series, we
  first show how the elements in the dual of an Artinian algebra
  correspond to polynomial-exponential functions. They are also the
  solutions of systems of partial differential equations with constant
  coefficients.  We relate their representation to the inverse system
  of the isolated points of the characteristic variety.
  
  Using the properties of Hankel operators, we establish a
  correspondence between poly\-nomial-exponential series and Artinian
  Gorenstein algebras. We generalize Kronecker theorem to the
  multivariate case, by showing that the symbol of a Hankel operator
  of finite rank is a polynomial-exponential series and by connecting
  the rank of the Hankel operator with the decomposition of the
  symbol.
  
  A generalization of Prony's approach to multivariate decomposition
  problems is presented, exploiting eigenvector methods for solving
  polynomial equations. We show how to compute the frequencies and
  weights of a minimal polynomial-exponential decomposition, using the
  first coefficients of the series. A key ingredient of the approach
  is the flat extension criteria, which leads to a multivariate
  generalization of a rank condition for a Carath{\'e}odory-Fej{\'e}r
  decomposition of multivariate Hankel matrices. A new algorithm is
  given to compute a basis of the Artinian Gorenstein algebra, based
  on a Gram-Schmidt orthogonalization process and to decompose
  polynomial-exponential series.
  
  A general framework for the applications of this approach is
  described and illustrated in different problems. We provide
  Kronecker-type theorems for convolution operators, showing that a
  convolution operator (or a cross-correlation operator) is of finite
  rank, if and only if, its symbol is a polynomial-exponential
  function, and we relate its rank to the decomposition of its
  symbol. We also present Kronecker-type theorems for the
  reconstruction of measures as weighted sums of Dirac measures from
  moments and for the decomposition of polynomial-exponential
  functions from values. Finally, we describe an application of this
  method for the sparse interpolation of polylog functions from
  values.
\end{abstract}
\medskip \medskip
\noindent{}{\bf AMS classification:} 14Q20,  68W30, 47B35, 15B05
\vfill\newpage
{\tableofcontents}

\section{Introduction}

Sensing is a classical technique, which is nowadays heavily used in many
applications to transform continuous signals or functions into discrete data.
In other contexts, in order to analyze physical phenomena or the evolution of
our environments, large sequences of measurements can also be
produced from sensors, cameras or scanners to discretize the problem.

An important challenge is then to recover the underlying structure of the
observed phenomena, signal or function. This means to extract
from the data, a structured or sparse representation of the function, which is
easier to manipulate, to transmit or to analyze. Recovering this underlying
structure can boil down to compute an explicit representation of the function
in a given functional space. Usually, a ``good'' numerical approximation of
the function as a linear combination of a set of basis functions is
sufficient. The choice of the basis functions is very important from this
perspective. It can lead to a representation, with many non-zero coefficients
or a sparse representation with few coefficients, if the basis is well-suited.
To illustrate this point, consider a linear function over a bounded interval
of $\mathbbm{R}$. It has a sparse representation in the monomial basis since
it is represented by two coefficients. But its description as Fourier series
involves an infinite sequence of (decreasing) Fourier coefficients.

This raises the questions of how to determine a good functional space, in
which the functions we consider have a sparse representation, and how to
compute such a decomposition, using a small (if not minimal) amount of
information or measurements.

In the following, we consider a special reconstruction problem, which will
allow us to answer these two questions in several other contexts.
The functional space, in which we are going to compute sparse
representations is the space of polynomial-exponential functions.
The data that we use corresponds to the Taylor coefficients of these functions.
Hereafter, we call them {\em moments}.
They  are for instance the Fourier coefficients of a signal or the values of
a function sampled on a regular grid. It can also be High Order Statistical
moments or cumulants, used in Signal Processing to perform blind
identification {\cite{LiCo:2014}}. Many other examples of reconstruction from
moments can be found in Image Analysis, Computer Vision, Statistics \dots
The reconstruction problem consists in computing a polynomial-exponential
representation of the series from the (truncated) sequence of its moments.
We will see that this problem allows us to recover sparse
representations in several contexts.

With the multi-index notation: $\forall \alpha = (\alpha_1, \ldots, \alpha_n)
\in \mathbbm{N}^n, \forall \tmmathbf{u} \in \mathbbm{C}^n$,
$\alpha ! = \prod_{i = 1}^n \alpha_i !$, $\tmmathbf{u}^{\alpha} = \prod_{i =
1}^n u_i^{\alpha} $, $\tmmathbf{e}_{\xi} (\tmmathbf{y}) = \sum_{\alpha \in
\mathbbm{N}^n} \frac{1}{\alpha !} \xi^{\alpha} \tmmathbf{y}^{\alpha} =
e^{\langle \xi, \tmmathbf{{{y}}} \rangle} = e^{\xi_1 y_{1} + \cdots +
\xi_n y_n}$ for $\xi = (\xi_1, \ldots, \xi_n) \in \mathbbm{C}^n$, and for
$\mathbbm{C} [[{\bdy}]] =\mathbbm{C} [[y_1, \ldots, y_n]]$ the ring of
formal power series in $y_1, \ldots, y_n$, this decomposition problem
can be stated as follows.

\tmtextbf{Polynomial-exponential decomposition from moments:} {\tmem{Given
coefficients $\sigma_{\alpha}$ for $\alpha \in \tmmathbf{a} \subset
\mathbbm{N}^n$ of the series }}
\[ \sigma (\tmmathbf{y}) = \sum_{\alpha \in \mathbbm{N}^n} \sigma_{\alpha} 
   \frac{\tmmathbf{y}^{\alpha}}{\alpha !} \in \mathbbm{C} [[{\bdy}]],
\]
{\tmem{recover $r$ points $\mathbf{\xi}_1, \ldots, \mathbf{\xi}_r \in
\mathbbm{C}^n$ and $r$ polynomial coefficients $\omega_i (\tmmathbf{y}) \in
\mathbbm{C} [\tmmathbf{y}]$ such that}}
\begin{equation}
  \sigma (\tmmathbf{y}) = \sum_{i = 1}^r \hspace{0.25em} \omega_i
  (\tmmathbf{y}) \tmmathbf{e}_{{\xi_i}} (\tmmathbf{y}).
  \label{eq:decompseries}
\end{equation}
A function of the form (\ref{eq:decompseries}) is called a
{\tmem{polynomial-exponential}} function. We aim at recovering the minimal
number $r$ of terms in the decomposition (\ref{eq:decompseries}). Since only
the coefficients $\sigma_{\alpha}$ for $\alpha \in \tmmathbf{a}$ are known,
computing the decomposition (\ref{eq:decompseries}) means that the
coefficients of $\tmmathbf{y}^{\alpha}$ are the same in the series on both
sides of the equality, for $\alpha \in \tmmathbf{a}$.

\subsection{Prony's method in one variable\label{sec:pronyunivar}}

One of the first work in this area is probably due to Gaspard-Clair-Fran{\c
c}ois-Marie Riche de Prony, mathematician and engineer of the {\'E}cole
Nationale des Ponts et Chauss{\'e}es. He was working on Hydraulics. To analyze
the expansion of various gases, he proposed in {\cite{Prony1795}}
a method to fit a sum of exponentials at equally spaced data points in order
to extend the model at intermediate points. More precisely, he was studying
the following problem:

Given a function $h \in C^{\infty} (\mathbbm{R})$ of the form
\begin{equation}
  x \in \mathbbm{R} \mapsto h (x) = \sum_{i = 1}^r \hspace{0.25em} \omega_i
  \hspace{0.25em} e^{f_i x} \in \mathbbm{C} \label{eq:sumexpuni}
\end{equation}

where \ $f_1, \ldots, f_r \in \mathbbm{C}$ are pairwise distinct, $\omega_i
\in \mathbb{C} \setminus \{0\}$, the problem consists in recovering
\begin{itemize}
  \item the distinct \tmtextit{frequencies} $f_1, \ldots, f_r \in \mathbbm{C}$,
  
  \item the {\tmem{coefficients}} $\omega_i \in \mathbb{C} \setminus \{0\}$,
\end{itemize}
Figure \ref{fig:freq} shows an example of such a signal, which is the superposition of several
``oscillations'' with different frequencies.

\begin{figure}[h]
\begin{center}
  \resizebox{12cm}{!}{\includegraphics{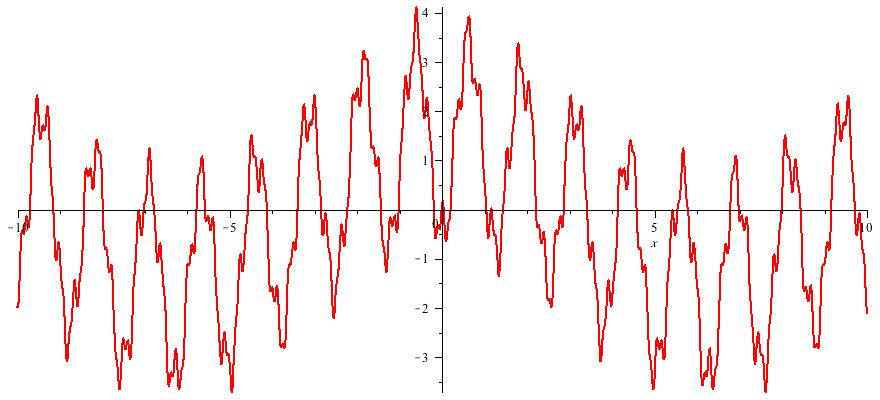}}
\end{center}
 \caption{\label{fig:freq}A superposition of oscillations with different frequencies}
\end{figure}

The approach proposed by G. de Prony can be reformulated into a truncated
series reconstruction problem. By choosing an arithmetic progression of points
in $\mathbbm{R}$, for instance the integers $\mathbbm{N}$, we
can associate to $h$, the generating series:
\[ \sigma_h (y) = \sum_{a \in \mathbbm{N}} h (a)  \frac{y^a}{a!} \in
   \mathbbm{C} [[y]], \]
where $\mathbbm{C} [[y]]$ is the ring of formal power series in the variable
$y$. If $h$ is of the form (\ref{eq:sumexpuni}), then
\begin{equation}
  \sigma_h (y) = \sum_{i = 1}^r \hspace{0.25em} \sum_{a \in \mathbbm{N}}
  \omega_i \xi_i^a \frac{y^a}{a!} = \sum_{i = 1}^r \omega_i e^{\xi_i y} 
  \label{eq:decompsigma}
\end{equation}
where $\xi_i = e^{f_i}$. Prony's method consists in reconstructing
the decomposition (\ref{eq:decompsigma}) from a small number of coefficients
$h (a)$ for $a = 0, \ldots, 2 r - 1$. It performs as
follows:
\begin{itemize}
  \item From the values $h (a)$ for $a \in [0, \ldots, 2 r - 1]$ ,
  compute the polynomial
  \[
  p (x) = \prod_{i = 1}^r (x - \xi_i) = x^r - \sum_{j =  0}^{r - 1} p_j x^j, \]
  which roots are $\xi_i = e^{f_i}$, $i = 1, \ldots, r$ as follows. Since it
  satisfies the recurrence relations
  \[ \forall j \in [0, \ldots, r - 1], \hspace{1cm} \sum^{r - 1}_{i = 0}
     \sigma_{j + i} p_i - \sigma_{j + r} = -\sum_{i = 1}^r w_i \xi_i^j p
     (\xi_i) = 0, \]
  it is the unique solution of the system:    
  \begin{equation}
    \left(\begin{array}{ccccc}
      \sigma_0 & \sigma_1 & \ldots &  & \sigma_{r - 1}\\
      \sigma_1 &  &  & \udots & \\
      \vdots &  & \udots &  & \vdots\\
      & \udots &  &  & \\
      \sigma_{r - 1} &  & \ldots &  & \sigma_{2 r - 2}
    \end{array}\right) \left(\begin{array}{c}
      p_0\\
      p_1\\
      \vdots\\
      \vdots\\
      p_{r - 1}
    \end{array}\right) = \left(\begin{array}{c}
      \sigma_r\\
      \sigma_{r + 1}\\
      \vdots\\
      \vdots\\
      \sigma_{2 r - 1}
    \end{array}\right). \label{syst1}
  \end{equation}
  \item Compute the roots $\xi_1, \ldots, \xi_r$ of the polynomial $p
  (x)$.
  
  \item To determine the weight coefficients $w_1, \ldots,
  w_r$, solve the following linear (Vandermonde) system:

  \[ \left(\begin{array}{cccc}
       1 & 1 & \ldots & 1\\
       \xi_1 & \xi_2 & \ldots & \xi_r\\
       \vdots & \vdots &  & \vdots\\
       \xi_1^{r - 1} & \xi_2^{r - 1} & \ldots & \xi_r^{r - 1}
     \end{array}\right) \left(\begin{array}{c}
       w_1\\
       w_2\\
       \vdots\\
       w_r
     \end{array}\right) = \left(\begin{array}{c}
       h_0\\
       h_1\\
       \vdots\\
       h_{r - 1}
     \end{array}\right). \]
\end{itemize}
This approach can be improved by computing the roots $\xi_1, \ldots,
\xi_r$, directly as the generalized eigenvalues of a pencil of Hankel
matrices. Namely, Equation (\ref{syst1}) implies that
\begin{equation}
  \label{linsyst2} \overbrace{\left(\begin{array}{ccccc}
    \sigma_0 & \sigma_1 & \ldots &  & \sigma_{r - 1}\\
    \sigma_1 &  &  & \udots & \\
    \vdots &  & \udots &  & \vdots\\
    & \udots &  &  & \\
    \sigma_{r - 1} &  & \ldots &  & \sigma_{2 r - 2}
  \end{array}\right)}^{H_0}  \overbrace{\left(\begin{array}{ccccc}
    0 &  &  &  & p_0\\
    1 & \ddots &  &  & p_1\\
    & \ddots & \ddots &  & \vdots\\
    &  & \ddots & 0 & \vdots\\
    &  &  & 1 & p_{r - 1}
  \end{array}\right)}^{M_p} = \overbrace{\left(\begin{array}{ccccc}
    \sigma_1 & \sigma_2 & \ldots &  & \sigma_r\\
    \sigma_2 &  &  & \udots & \\
    \vdots &  & \udots &  & \vdots\\
    & \udots &  &  & \\
    \sigma_r &  & \ldots &  & \sigma_{2 r - 1}
  \end{array}\right)}^{H_1},
\end{equation}
so that the generalized eigenvalues of the pencil $(H_1, H_0)$ are the
eigenvalues of the companion matrix $M_p$ of $p (x)$, that is, its the roots
$\xi_1, \ldots, \xi_r$. This variant of Prony's method is also called
the {\tmem{pencil method}} in the literature.

For numerical improvement purposes, one can also chose an arithmetic
progression \ $\frac{a}{T}$ \ and $a \in [0, \ldots, 2 r - 1]$, with
$T \in \mathbbm{R}^+$ of the same order of magnitude as the frequencies $f_i$.
The roots of the polynomial $p$ are then $\xi_i = e^{\frac{f_i}{T}}$.

\subsection{Related work}

The approximation of functions by a linear combination of exponential
functions appears in many contexts. It is the basis of Fourier analysis, where
infinite series are involved in the decomposition of functions. The
frequencies of these infinite sums of exponentials belong to an infinite grid
on the imaginary axis in the complex plane.

An important problem is to represent or approximate a function by a sparse or
finite sum of exponential terms, removing the constraints on the frequencies.
This problem has a long history and many applications, in particular in signal
processing {\cite{golub_separable_2003}}, {\cite{pereyra_exponential_2010}}.

Many works have been developed in the one dimensional case, which refers to
the well-known problem of {\tmem{parameter estimation for exponential sums}}.
A first family of methods can be classified as Prony-type methods. To take
into account the problem of noisy data, the recurrence relation is computed by
minimization techniques {\cite{pereyra_exponential_2010}}[chap. 1]. Another type of
methods is called Pencil-matrix {\cite{pereyra_exponential_2010}}[chap. 1]. Instead of
computing a recurrence relation, the generalized eigenvalues of a pencil of
Hankel matrices are computed.
The survey paper {\cite{golub_separable_2003}} describes some of these minimization
techniques implementing a variable projection algorithm and their applications
in various domains, including antenna analysis with so-called MUSIC
{\cite{SwindlehurstKailath92}} or ESPRIT {\cite{RoyKailath90}} methods.
In {\cite{BeyMon05}}, another approach based on conjugate-eigenvalue
computation and low rank Hankel matrix approximation is proposed. The
extension of this method in {\cite{PotTas10}}, called Approximate Prony
Method, is using controlled perturbations.
The problem of accurate reconstruction of sums of univariate polynomial-exponential
functions associated to confluent Prony systems has been investigated in \cite{batenkov_accuracy_2013}, where polynomial derivatives of Dirac measures correspond to polynomial weights. 

The reconstruction problem has also been studied in the multi-dimensional case
{\cite{And10}}, {\cite{potts_parameter_2013}},
{\cite{kunis_multivariate_2016}}.
These methods are applicable for problems where the degree of the moments is high enough to
recover the multivariate solutions from some projections in one dimension.
Even more recently, techniques for solving polynomial equations, which rely on the computation of $H$-bases, have been exploited in this context {\cite{sauer_pronys_2016-1}}.

The theory builds on the properties of Hankel matrices of finite
rank, starting with a result due to {\cite{kronecker_zur_1880}} in the one
variable case. This result states that there is an explicit correlation
between polynomial-exponential series and Hankel operators of finite rank. The
literature on Hankel operators is huge and mainly focus on one variable (see
e.g. {\cite{peller_excursion_1998}}). Kronecker's result has been extended to
several variables for multi-index sequences {\cite{power_finite_1982}},
{\cite{barachart_realisation_1985}}, {\cite{andersson_kronecker_2015}}, for
non-commutative variables {\cite{fliess_series_1970}}, for integral
cross-correlation operators {\cite{andersson_general_2015}}. In some cases as in
{\cite{andersson_general_2015}}, methods have been proposed to compute the
rank in terms of  the polynomial-exponential decomposition.

Hankel matrices are central in the theory of  Pad\'e approximants for functions
of one variable. Here also a large literature exists for univariate
Pad\'e approximants: see e.g. \cite{baker_pade_1996} for approximation properties,
\cite{beckermann_numerical_2007} for numerical stability problems,
\cite{beckermann_uniform_1994}, \cite{gathen_modern_2013} for algorithmic aspects.
The extension to multivariate functions is much less developed
{\cite{power_finite_1982}}, {\cite{cuyt_how_1999}}.

This type of approaches is also used in sparse interpolation of black box
polynomials. In the methods developed in {\cite{BenOrTiwari:1988}},
{\cite{Zippel:1979}}, further improved in {\cite{GiLaWe09}}, the sparse
polynomial is evaluated at points of the form $(\omega_1^k, \ldots,
\omega_n^k)$ where $\omega_i$ are prime numbers or primitive roots of unity of
co-prime order. The sparse decomposition of the black box polynomial is
computed from its values by a univariate Prony-like method.

Hankel matrices and their kernels also play an important role in error
correcting codes. Reed-Solomon codes, obtained by evaluation of a polynomial
at a set of points and convolution by a given polynomial, can be decoded from
their syndrome sequence by computing the error locator polynomial
{\cite{macwilliams_theory_1977}}[chap. 9]. This is a linear recurrence
relation between the syndrome coefficients, which corresponds to a non-zero
element in the kernel of a Hankel matrix. Berlekamp
{\cite{berlekamp_nonbinary_1968}} and Massey
{\cite{massey_shift-register_1969}} proposed an efficient algorithm to compute
such polynomials. Sakata extended the approach to compute Gr{\"o}bner bases of
polynomials in the kernel of a multivariate Hankel matrix
{\cite{sakata_finding_1988}}. The computation of multivariate linear
recurrence relations have been further investigated, e.g. in
{\cite{fitzpatrick_finding_1990}} and more recently in
{\cite{berthomieu_linear_2015}}.

Computing polynomials in the kernel of Hankel matrices and their roots is also
the basis of the method proposed by J.J. Sylvester
{\cite{sylvester_essay_1851}} to decompose binary forms. This approach has
been extended recently to the decomposition of multivariate symmetric and
multi-symmetric tensors in {\cite{BCMT:2009:laa}}, {\cite{bernardi_general_2013}}.

A completely different approach, known as {\tmem{compressive
sensing}}, has been developed over the last decades to compute sparse
decompositions of functions (see e.g. {\cite{CandesRombergTao06}}). In this
approach, a (large) dictionary of functions is chosen and a sparse combination
with few non-zero coefficients is computed from some observations. This boils
to find a sparse solution $X$ of an underdetermined linear system $Y = A X$.
Such a solution, which minimizes the $\ell^0$ ``norm'' can be computed by $\ell^1$
minimization, under some hypothesis.

For the sparse reconstruction problem from a discrete set of frequencies, it
is shown in {\cite{CandesRombergTao06}} that the $\ell^1$ minimization provides a
solution, for enough Fourier coefficients (at least $4 r$) chosen at random.
As shown in {\cite{plonka_prony_2014}}, this problem can also be solved by
a Prony-like approach, using only $2 r$ Fourier coefficients.

\subsection{Contributions}

In this paper, we analyze the problem of sparse decomposition of series from
an algebraic point of view and propose new methods to compute such
decompositions.

We exploit the duality between polynomials and formal power series. The
formalism is strongly connected to the inverse systems introduced by F.S
Macaulay \cite{macaulay_algebraic_1916}: evaluations at points correspond to
exponential functions and the multiplication to derivation.
This duality between polynomial equations and partial differential equations
has been investigated previously, for instance in \cite{riquier_les_1910},
\cite{grobner_uber_1937}, 
\cite{malgrange_existence_1956}, 
\cite{emsalem_geometrie_1978}, 
\cite{hormander_introduction_1990}
\cite{pedersen_basis_1999}, 
\cite{oberst_constructive_2001}, 
\cite{hakopian_partial_2004}. 
We give here an explicit description of the elements in the dual of an Artinian algebra (Theorem \ref{thm:artindual}), in terms of polynomial-exponential functions associated
to  the inverse system of the roots of the characteristic variety.
This gives a new and complete characterization of the
solutions of partial differential equations with constant coefficients for
zero-dimensional partial differential systems (Theorem \ref{thm:edp}).

The sparse decomposition problem is tackled by studying the Hankel
operators associated to the generating series. This approach has
been exploited in many contexts: signal processing
(see e.g. \cite{pereyra_exponential_2010}),
functional analysis (see e.g. \cite{peller_excursion_1998}),
but also in tensor decomposition problems \cite{BCMT:2009:laa}, \cite{bernardi_general_2013},
or polynomial optimization \cite{lasserre_moment_2013}. Our algebraic formalism
allows us to establish a correspondence between polynomial-exponential
series and Artinian Gorenstein algebras using these Hankel operators.

A fundamental result on univariate Hankel operators of finite rank is
Kronecker's theorem \cite{kronecker_zur_1880}.  Several works have
already proposed multivariate generalization of Kronecker's theorem, including methods to
compute the rank of the Hankel operator from its symbol
\cite{power_finite_1982}, \cite{gu_finite_1999},
\cite{olshevsky_2d-extension_2010}, \cite{andersson_general_2015},
\cite{andersson_kronecker_2015}.
We prove a new multivariate generalization of Kronecker's theorem (Theorem
\ref{thm:gorenstein}), showing that the symbol associated to a Hankel
operator of finite rank is a polynomial-exponential series and
describing its rank in terms of the polynomial weights.  More
precisely, we show that, as in the univariate case, the rank of the
Hankel operator is simply the sum of the dimensions of the vector spaces
spanned by all the derivatives of the polynomial weights of its
symbol, represented as a polynomial-exponential series.

Exploiting classical eigenvector methods for solving polynomial
equations, we show how to compute a polynomial-exponential
decomposition, using the first coefficients of the generating series
(Algorithms \ref{algo:genprony1} and \ref{algo:genprony2}).
In particular, we show how to recover the weights in the decomposition
from the eigenspaces, for simple roots and multiple roots. Compared to
methods used in \cite{kunis_multivariate_2016} or
\cite{sauer_pronys_2016-1}, we need not solve polynomial equations
deduced from elements in the kernel of the Hankel operator. We 
directly apply linear algebra operations on some truncated Hankel
matrices to recover the decomposition. This allow us to use moments of lower degree
compared for instance to methods proposed in 
\cite{And10},
\cite{potts_parameter_2013},
\cite{kunis_multivariate_2016}.

This approach also leads to a multivariate generalization of
Carath{\'e}odory-Fej{\'e}r decomposition
\cite{caratheodory_uber_1911}, which has been recently generalized to
positive semidefinite multi-level block Toeplitz matrices in
\cite{yang_vandermonde_2016} and \cite{andersson_kronecker_2015}.
These results apply for simple roots or constant weights.
Propositions \ref{dec:vdm:simple} and \ref{dec:vdm:multiple} provide
such a decomposition of Hankel matrices of a polynomial-exponential symbol
in terms of the weights and generalized Vandermonde matrices,
for simple and multiple roots.

To have a complete method for decomposing a polynomial-exponential
series from its first moments using the proposed Algorithms
\ref{algo:genprony1} and \ref{algo:genprony2}, one needs to determine
a basis of the Artinian Gorenstein algebra. A new algorithm is
described to compute such a basis (Algorithm \ref{algo:orthobasis}),
which applies a Gram-Schmidt orthogonalization process and computes
pairwise orthogonal polynomial bases.  The method is connected to
algorithms to compute linear recurrence relations of multi-index
series, such as Berlekamp-Massey-Sakata algorithm
\cite{sakata_finding_1988}, or \cite{fitzpatrick_finding_1990},
\cite{berthomieu_linear_2015}. It proceeds inductively by projecting
orthogonally new elements on the space spanned by the previous
orthogonal polynomials, instead of computing the discrepancy of the
new elements. Thus, it can compute more general basis of the ideal of
recurrence relations, such as Border Bases \cite{m-99-nf},
\cite{MoTr:2005:issac}.

A key ingredient of the approach is the flat extension criteria.  In
Theorem \ref{thm:flatext}, we provide such a criteria based on rank
conditions, for the existence of a polynomial-exponential series
extension. It generalizes results from \cite{curto_solution_1996},
\cite{laurent_generalized_2009}.

A general framework for the application of this approach based on
generating series is described, extending the construction of
{\cite{peter_generalized_2013}} to the multivariate setting. We
illustrate it in different problems, showing that several results in
analysis are consequences of the algebraic multivariate Kronecker
theorem (Theorem \ref{thm:gorenstein}).  In particular, we provide
Kronecker-type theorems (Theorems \ref{thm:conv1}, \ref{thm:conv2} and
\ref{thm:discreteconv}) for convolution operators (or
cross-correlation operators), considered in
{\cite{andersson_general_2015}}, {\cite{andersson_kronecker_2015}}.
Theorem \ref{thm:gorenstein} implies that the rank of a convolution
(or correlation) operator with a polynomial-exponential symbol is the
sum of the dimensions of the space spanned by all the derivatives of
the polynomial weights of the symbol. By Lemma \ref{lm:mu=rank}, this
gives a simple description of the output of the method proposed in
{\cite{andersson_general_2015}} to compute the rank of convolution
operators.  We also deduce Kronecker-type theorems for the
reconstruction of measures as weighted sums of Dirac measures from
moments and the decomposition of polynomial-exponential functions from
values.
Finally, we describe a new approach for the sparse interpolation of
polylog functions from values, with Kronecker-type results on their
decomposition. Compared to previous approaches such as
\cite{BenOrTiwari:1988}, \cite{Zippel:1979}, \cite{GiLaWe09}, we don't
project the problem in one dimension and recover sparse polylog
terms from the multiplicity structure of the roots.

\section{Duality and Hankel operators}

In this section, we consider polynomials and series with coefficients in a
field $\mathbbm{K}$ of characteristic $0$. In the applications, we are going
to take $\mathbbm{K}=\mathbbm{C}$ or $\mathbbm{K}=\mathbbm{R}$.

We are going to use the following notation: $\mathbbm{K} [x_1, \ldots, x_n]
=\mathbbm{K} [\tmmathbf{x}] = R$ is the ring of polynomials in the variables
$x_1, \ldots, x_n$ with coefficients in the field $\mathbbm{K}$, $\mathbbm{K}
[[y_1, \ldots, y_n]] =\mathbbm{K} [[\tmmathbf{y}]]$ is the ring of formal
power series in the variables $y_1, \ldots, y_n$ with coefficients in
$\mathbbm{K}$. For a set $B \subset \mathbbm{K} [\tmmathbf{x}]$, $B^+ =
\cup_{i = 1}^n x_i B \cup B$, $\partial B = B^+ \setminus B$. For $\alpha,
\beta \in \mathbbm{N}^n$, we say that $\alpha \ll \beta$ if $\alpha_i
\leqslant \beta_i$ for $i = 1, \ldots, n$.

\subsection{Duality}

In this section, we describe the natural isomorphism between the ring of
formal power series and the dual of $R =\mathbbm{K} [x_1, \ldots,
x_n]$. It is given by the following pairing:
\begin{eqnarray*}
  \mathbbm{K} [[y_1, \ldots, y_n]] \times \mathbbm{K} [x_1, \ldots,  x_n] & \rightarrow & \mathbbm{K}\\
  (\tmmathbf{y}^{\alpha}, \tmmathbf{x}^{\beta}) & \mapsto
  & \left\langle \tmmathbf{y}^{\alpha} |  
  \tmmathbf{x}^{\beta} \right\rangle = \left\{ \begin{array}{ll}
    \alpha ! & \tmop{if} \alpha = \beta\\
    0 & \tmop{otherwise}.
  \end{array} \right.
\end{eqnarray*}
Namely, if $\Lambda \in \tmop{Hom}_{\mathbbm{K}} (\mathbbm{K} [\tmmathbf{x}]
\nocomma, \mathbbm{K}) = R^{\ast}$ is an element of the dual of $\mathbbm{K}
[\tmmathbf{x}]$, it can be represented by the series:
\begin{equation}
  \Lambda (\tmmathbf{y}) = \sum_{\alpha \in \mathbbm{N}^n} \Lambda
  (\tmmathbf{x}^{\alpha})  \frac{\tmmathbf{y}^{\alpha}}{\alpha !} \in
  \mathbbm{K} [[y_1, \ldots, y_n]], \label{eq:dualfps}
\end{equation}
so that we have $\langle \Lambda (\tmmathbf{y}) | 
 \tmmathbf{x}^{\alpha} \rangle = \Lambda (\tmmathbf{x}^{\alpha})$.
The map $\Lambda \in R^{\ast} \mapsto \sum_{\alpha \in \mathbbm{N}^n} \Lambda
(\tmmathbf{x}^{\alpha})  \frac{\tmmathbf{y}^{\alpha}}{\alpha !} \in
\mathbbm{K} [[\tmmathbf{y}]]$ is an isomorphism and any series \ $\sigma
(\tmmathbf{y}) = \sum_{\alpha \in \mathbbm{N}^n} \sigma_{\alpha} 
\frac{\tmmathbf{y}^{\alpha}}{\alpha !} \in \mathbbm{K} [[\tmmathbf{y}]]$ can
be interpreted as a linear form
\[ p = \sum_{\alpha \in A \subset \mathbbm{N}^n} p_{\alpha}
   \tmmathbf{x}^{\alpha} \in \mathbbm{K} [\tmmathbf{x}] \mapsto \langle \sigma
   \mid p \rangle = \sum_{\alpha \in A \subset \mathbbm{N}^n} p_{\alpha}
   \sigma_{\alpha}. \]
Any linear form $\sigma \in R^{\ast}$ is uniquely defined by its
coefficients $\sigma_{\alpha} = \langle \sigma \mid
\tmmathbf{x}^{\alpha} \rangle$ for $\alpha \in \mathbbm{N}^n$, which are
called the {\tmem{moments}} of $\sigma$.

From now on, we identify the dual $\tmop{Hom}_{\mathbbm{K}} (\mathbbm{K}
[\tmmathbf{x}], \mathbbm{K})$ with $\mathbbm{K} [[\tmmathbf{y}]]$.
Using this identification, the dual basis of the monomial basis
$(\tmmathbf{x}^{\alpha})_{\alpha \in \mathbbm{N}^n}$ is $\left(
\frac{\tmmathbf{y}^{\alpha}}{\alpha !} \right)_{\alpha \in \mathbbm{N}^n}$.

If $\mathbbm{K}$ is a subfield of a field $\mathbbm{L}$, we have the
embedding $\mathbbm{K} [[\tmmathbf{y}]] \hookrightarrow \mathbbm{L}
[[\tmmathbf{y}]]$, which allows to identify an element of $\mathbbm{K}
[\tmmathbf{x}]^{\ast}$ with an element of $\mathbbm{L} [\tmmathbf{x}]^{\ast}$.

The truncation of an element $\sigma (\tmmathbf{y}) = \sum_{\alpha \in
\mathbbm{N}^n} \sigma_{\alpha}  \frac{\tmmathbf{y}^{\alpha}}{\alpha !} \in
\mathbbm{K} [[\tmmathbf{y}]]$ in degree $d$ \ is $\sum_{| \alpha | \leqslant
d} \sigma_{\alpha}  \frac{\tmmathbf{y}^{\alpha}}{\alpha !}$. It is denoted
$\sigma (\tmmathbf{y}) + ((\tmmathbf{y}))^{d + 1}$, that is, the class of
$\sigma$ modulo the ideal \ $(y_1, \ldots, y_n)^{d + 1} \subset \mathbbm{K}
[[\tmmathbf{y}]]$.

For an ideal $I \subset R =\mathbbm{K} [\tmmathbf{x}]$, we denote by $I^{\bot}
\subset \mathbbm{K} [[\tmmathbf{y}]]$ the space of linear forms $\sigma \in
\mathbbm{K} [[\tmmathbf{y}]]$, such that $\forall p \in I$, $\langle \sigma
\mid p \rangle = 0$. Similarly, for a vector space $D \subset \mathbbm{K}
[[\tmmathbf{y}]]$, we denoted by $D^{\bot} \subset \mathbbm{K} [\tmmathbf{x}]$
the space of polynomials $p \in \mathbbm{K} [\tmmathbf{x}]$, such that
$\forall \sigma \in D$, $\langle \sigma \mid p \rangle = 0$. If $D$
is closed for the $(\tmmathbf{y})$-adic topology, then $D^{\perp \perp} = D$
and if $I \subset \mathbbm{K} [\tmmathbf{x}]$ is an ideal $I^{\perp}$, then
$I^{\perp \perp} = I$.

The dual space $\tmop{Hom} (\mathbbm{K} [\tmmathbf{x}], \mathbbm{K})
\equiv \mathbbm{K} [[\tmmathbf{y}]]$ has a natural structure of $\mathbbm{K}
[\tmmathbf{x}]$-module, defined as follows: $\forall \sigma (\tmmathbf{y}
) \in \mathbbm{K} [[\tmmathbf{y}]], \forall p
(\tmmathbf{x}), q (\tmmathbf{x}) \in \mathbbm{K} [\tmmathbf{x}]
\nocomma$,
\begin{eqnarray*}
  \langle p (\tmmathbf{x}) \star \sigma (\tmmathbf{y} )
  \mid q (\tmmathbf{x}) \rangle & = & \langle \sigma (\tmmathbf{y}
  ) \mid p (\tmmathbf{x}) q (\tmmathbf{x}) \rangle
.
\end{eqnarray*}
We check that $\forall \sigma \in \mathbbm{K} [[\tmmathbf{y}]]
\nocomma, \forall p, q \in \mathbbm{K} [\tmmathbf{x}]$, $(p
q) \star \sigma = p \star (q \star \sigma)$. See e.g.
{\cite{emsalem_geometrie_1978}}, {\cite{mourrain_isolated_1996}} for more
details.

For $p = \sum_{\beta} p_{\beta} \tmmathbf{x}^{\beta} \in \mathbbm{K}
[\tmmathbf{x}]$ and $\sigma = \sum_{\alpha \in \mathbbm{N}^n} \sigma_{\alpha} 
\frac{\tmmathbf{y}^{\alpha}}{\alpha !} \in \mathbbm{K} [[\tmmathbf{y}]]$, the
series expansion of $p \star \sigma$ is $p \star \sigma = \sum_{\alpha \in
\mathbbm{N}^n} \rho_{\alpha}  \frac{\tmmathbf{y}^{\alpha}}{\alpha !} \in
\mathbbm{K} [[\tmmathbf{y}]]$ with $\forall \alpha \in \mathbbm{N}^n,$
\[
\rho_{\alpha} = \sum_{\beta} p_{\beta} \sigma_{\alpha + \beta}.
\]
Identifying $\mathbbm{K} [\tmmathbf{x}]$ with the set $\fsupp
(\mathbbm{K}^{\mathbbm{N}^n})$ of sequences $p =
(p_{\alpha})_{\alpha \in \mathbbm{N}^n}$ of finite support (i.e. a finite number of non-zero terms), and
$\mathbbm{K} [[\tmmathbf{y}]]$ with the set of sequences $\sigma =
(\sigma_{\alpha})_{\alpha \in \mathbbm{N}^n}$, $p \star \sigma$ is the
cross-correlation sequence of $p$ and $\sigma$.

For any $\sigma (\tmmathbf{y} ) \in \mathbbm{K}
[[\tmmathbf{y}]]$, the inner product associated to $\sigma (\tmmathbf{y})$ on
$\mathbbm{K} [\tmmathbf{x}]$ is defined as follows:
\begin{eqnarray*}
  \mathbbm{K} [\tmmathbf{x}] \times \mathbbm{K} [\tmmathbf{x}] & \rightarrow &
  \mathbbm{K}\\
  (p (\tmmathbf{x}) , q (\tmmathbf{x})) & \mapsto & \langle
  p (\tmmathbf{x}), q (\tmmathbf{x}) \rangle_{\sigma}
  \assign \langle \sigma (\tmmathbf{y}) |   p
  (\tmmathbf{x}) q (\tmmathbf{x}) \rangle.
\end{eqnarray*}

\subsubsection{Polynomial-Exponential series}

Among the elements of $\tmop{Hom} (\mathbbm{K} [\tmmathbf{x}],
\mathbbm{K}) \equiv \mathbbm{K} [[\tmmathbf{y}]]$, we have the evaluations at
points of $\mathbbm{K}^n$:

\begin{definition}
  The evaluation at a point $\xi = (\xi_1, \ldots, \xi_n) \in \mathbbm{K}^n$
  is:
  \begin{eqnarray*}
    \tmmathbf{e}_{{\xi}} : \mathbbm{K} [x_1, \ldots, x_n] &
    \rightarrow & \mathbbm{K}\\
    p (\tmmathbf{x}) & \mapsto & p (\xi)
  \end{eqnarray*}
  It corresponds to the series:
  \begin{eqnarray*}
    \tmmathbf{e}_{{\xi}} (\tmmathbf{y}) & = & \sum_{\alpha \in \mathbbm{N}^n}
    \xi^{\alpha}  \frac{\tmmathbf{y}^{\alpha}}{\alpha !} = \hspace{0.25em}
    e^{\xi_1 y_1 + \cdots + \xi_n y_n} = e^{\langle \xi, \mathbf{y} \rangle}.
  \end{eqnarray*}
\end{definition}

Using this formalism, the series $\sigma (\tmmathbf{y}) = \sum_{i = 1}^r
\omega_i \tmmathbf{e}_{\xi_i} (\tmmathbf{y})$ with $\omega_i \in \mathbbm{K}$
can be interpreted as a linear combination of evaluations at the points
$\xi_i$ with coefficients $\omega_i$, for $i = 1, \ldots, r$.
These series belong to the more general family of polynomial-exponential
series, that we define now.

\begin{definition}
  Let $\PolExp(y_1, \ldots, y_n) = \left\{ \sigma = \sum_{i
  = 1}^r \omega_i (\tmmathbf{y}) \tmmathbf{e}_{\xi_i} (\tmmathbf{y}) \in
  \mathbbm{K} [[\tmmathbf{y}]] \mid \xi_i \in \mathbbm{K}^n, \omega_i
  (\tmmathbf{y})\in \mathbbm{K} [\tmmathbf{y}] \right\}$ be the set of
  polynomial-exponential series. The polynomials $\omega_i (\tmmathbf{y})$ are called the weights of $\sigma$
and $\xi_{i}$ the frequencies.
\end{definition}

Notice that the product of $\tmmathbf{y}^{\alpha} 
\tmmathbf{e}_{{\xi}} (\tmmathbf{y})$ with a monomial $\tmmathbf{x}^{\beta}
\in \mathbbm{C} [x_1, \ldots, x_n]$ is given by
\begin{eqnarray}
  \left\langle \tmmathbf{y}^{\alpha} \tmmathbf{e}_{{\xi}} (\tmmathbf{y})
  |   \tmmathbf{x}^{\beta} \right\rangle & = &  \frac{\beta
  !}{(\beta - \alpha) !} \xi^{^{\beta - \alpha}} = \partial_{x_1}^{\alpha_1} \cdots \partial_{x_n}^{\alpha_n} 
  (\tmmathbf{x}^{\beta}) (\xi) \tmop{if} \alpha_i \leqslant \beta_i \tmop{for}
  i = 1, \ldots, n \label{eq:der} \\
  & = & 0 \tmop{otherwise}. \nonumber
\end{eqnarray}
Therefore an element $\sigma = \sum_{i = 1}^r \hspace{0.25em} \omega_i
(\tmmathbf{y}) \tmmathbf{e}_{{\xi_i}} (\tmmathbf{y})$ of
$\PolExp(\tmmathbf{y})$ can also be seen as a sum of
{\tmem{polynomial differential operators}} $\omega_i (\tmmathbf{\partial})$ ``at'' the
points $\xi_i$, that we call {\tmem{infinitesimal operators}}: $\forall p \in
\mathbbm{K} [\tmmathbf{x}], \langle \sigma |   p \rangle =
\sum_{i = 1}^r \hspace{0.25em} \omega_i (\tmmathbf{\partial}) (p) (\xi)$.

\subsubsection{Differential operators} An interesting property of the
outer product defined on $\mathbbm{K} [[\tmmathbf{y}]]$ is that polynomials
act as differentials on the series:

\begin{lemma}
  \label{lem:derivative}$\forall p \in \mathbbm{K} [\tmmathbf{x}], \forall
  \sigma \in \mathbbm{K} [[\tmmathbf{y}]]$, $p (\tmmathbf{x}) \star \sigma
  (\tmmathbf{y}) = p (\tmmathbf{\partial}) (\sigma)$.
\end{lemma}

\begin{proof}
  We first prove the relation for $p = x_i$ ($i \in [1, n]$) and
  $\sigma =\tmmathbf{y}^{\alpha}$ ($\alpha \in \mathbbm{N}^n$). Let $e_i = (0,
  \ldots, 0, 1, 0, \ldots, 0)$ be the exponent vector of $x_i$. $\forall \beta
  \in \mathbbm{N}^n$ and $\forall i \in [1, n]$, we have
  \begin{eqnarray*}
    \left\langle x_i \star \tmmathbf{y}^{\alpha}  | \tmmathbf{x}^{\beta} \right\rangle & = & \left\langle \tmmathbf{y}^{\alpha} |   x_i \tmmathbf{x}^{\beta} \right\rangle =
    \hspace{1em} \alpha ! \hspace{1em} \tmop{if} \hspace{1em} \alpha = \beta +
    e_i \hspace{1em} \tmop{and} \hspace{1em} 0 \hspace{1em} \tmop{otherwise}\\
    & = & \alpha_i  \langle \tmmathbf{y}^{\alpha - e_i} |
     \tmmathbf{x}^{\beta} \rangle.
  \end{eqnarray*}
  with the convention that $\tmmathbf{y}^{\alpha - e_i} = 0$ if $\alpha_i =
  0$. This shows that $x_i \star \tmmathbf{y}^{\alpha} = \alpha_i
  \tmmathbf{y}^{\alpha - e_i} = \partial_{y_i} (\tmmathbf{y}^{\alpha})$. By
  transitivity and bilinearity of the product $\star$, we deduce
  that $\forall p \in \mathbbm{K} [\tmmathbf{x}], \forall \sigma \in
  \mathbbm{K} [[\tmmathbf{y}]]$, $p (\tmmathbf{x}) \star \sigma (\tmmathbf{y}) = p (\tmmathbf{\partial}) (\sigma)$.
\end{proof}

  \label{rem:edp}This property can be useful to analyze the solution of
  partial differential equations. Let $p_1 (\tmmathbf{\partial}), \ldots, p_s (\tmmathbf{\partial}) \in
  \mathbbm{K} [\partial_1, \ldots, \partial_n]=\mathbbm{K}[\tmmathbf{\partial}]$ be a set of
  partial differential polynomials with constant coefficients $\in \mathbbm{K}$. The set of solutions
  $\sigma \in \mathbbm{K} [[{\bdy}]]$ of the system
  \[ p_1 (\tmmathbf{\partial}) (\sigma) = 0, \ldots, p_s (\tmmathbf{\partial}) (\sigma) = 0 \]
  is in correspondence with the elements $\sigma \in (p_1, \ldots,
  p_s)^{\bot}$, which satisfy $p_i \star \sigma = 0$ for $i = 1, \ldots, s$
  (see Theorem \ref{thm:edp}). 
  The variety $\mathcal{V}(p_{1},\ldots, p_{n})\subset \overline{\mathbbm{K}}^{n}$ is called the {\em  characteristic variety} and
  $I=(p_{1},\ldots,p_{n})$ the {\em  characteristic ideal} of the system of partial differential equations.
\begin{lemma}
  \label{lem:diffpolexp}$\forall p \in \mathbbm{K} [\tmmathbf{x}], \forall
  \omega \in \mathbbm{K} [[\tmmathbf{y}]]$, $\xi \in \mathbbm{K}^n, p
  (\tmmathbf{x}) \star (\omega (\tmmathbf{y}) \tmmathbf{e}_{\xi}
  (\tmmathbf{y})) = p (\xi_1 + \partial_{y_1}, \ldots, \xi_n + \partial_{y_n})
  (\omega (\tmmathbf{y})) \tmmathbf{e}_{\xi} (\tmmathbf{y})$.
\end{lemma}

\begin{proof}
  By the previous lemma, $x_i \star (\omega (\tmmathbf{y}) \tmmathbf{e}_{\xi}
  (\tmmathbf{y})) = \partial_{y_i} (\omega) (\tmmathbf{y}) \tmmathbf{e}_{\xi}
  (\tmmathbf{y}) + \xi_i \, \omega (\tmmathbf{y}) \tmmathbf{e}_{\xi}
  (\tmmathbf{y}) = ( \xi_i + \partial_{y_i}$)$(\omega
  (\tmmathbf{y})) \tmmathbf{e}_{\xi} (\tmmathbf{y})$ for $i = 1, \ldots, n$.
  We deduce that the relation is true for any polynomial $p \in \mathbbm{K}
  [\tmmathbf{x}]$ by repeated multiplications by the variables and linear
  combination.
\end{proof}

\begin{definition}
  For a subset $D \subset \mathbbm{K} [[\tmmathbf{y}]]$, the {\tmem{inverse
  system}} generated by $D$ is the vector space spanned by the elements $p
  (\tmmathbf{x}) \star \delta (\tmmathbf{y})$ for $\delta (\tmmathbf{y}) \in
  D$, $p (\tmmathbf{x}) \in \mathbbm{K} [\tmmathbf{x}]$, that is, by the
  elements in $D$ and all their derivatives. \
  
  For $\omega \in \mathbbm{K} [\tmmathbf{y}]$, we denote by ${\mu}
  (\omega)$ the dimension of the inverse system of $\omega$, generated by
  $\omega$ and all its derivatives $\tmmathbf{\partial}^{\alpha}
  (\omega)$, $\alpha \in \mathbbm{N}^n$.
\end{definition}
A simple way to compute this dimension is given in the next lemma:
\begin{lemma}\label{lm:mu=rank}
  For $\omega \in \mathbbm{K} [\tmmathbf{y}]$, $\mu(\omega)$ is the rank of the matrix $\Theta=(\theta_{\alpha,\beta})_{\alpha\in A,\beta\in B}$
  where
$$ 
\omega(\tmmathbf{y} + \tmmathbf{t}) = \sum_{\alpha\in A\subset \mathbb{N}^{n}, \beta\in B\subset \mathbb{N}^{n}} \theta_{\alpha,\beta}\, \tmmathbf{y}^{\alpha} \tmmathbf{t}^{\beta}
$$
for some finite subsets $A,B$ of $\mathbb{N}^{n}$.
\end{lemma}
\begin{proof} The Taylor expansion of $\omega(\tmmathbf{y} + \tmmathbf{t}) $ at $\tmmathbf{y}$ yields
$$ 
\omega(\tmmathbf{y} + \tmmathbf{t}) = \sum_{\beta \in B\subset \mathbb{N}^{n}} \tmmathbf{t}^{\beta}
  \frac{1}{\beta!}
    \tmmathbf{\partial}^{\beta}(\omega)(\tmmathbf{y}) = \sum_{\beta \in B\subset \mathbb{N}^{n}}   \tmmathbf{t}^{\beta} \sum_{\alpha\in A\subset \mathbb{N}} \theta_{\alpha,\beta}\, \tmmathbf{y}^{\alpha}.
$$
This shows that the rank of  the matrix $\Theta=(\theta_{\alpha,\beta})_{\alpha\in A,\beta\in B}$, that is, the rank of the vector space spanned by $\sum_{\alpha\in A\subset \mathbb{N}} \theta_{\alpha,\beta}\, \tmmathbf{y}^{\alpha}$ for $\beta\in B$ is the rank $\mu(\omega)$ of the vector space spanned by all the derivatives $\tmmathbf{\partial}^{\beta}(\omega)(\tmmathbf{y})$ of $\omega$.
\end{proof}
 
\begin{lemma}
  \label{lem:indep}The series $\tmmathbf{y}^{\alpha_{i, j}}
  \tmmathbf{e}_{{\xi_i}} (\tmmathbf{y})$ for $i = 1, \ldots, r$ and $j = 1,
  \ldots {\mu}_i$ with $\alpha_{i, 1}, \ldots, \alpha_{i, {\mu}_i} \in
  \mathbbm{N}^n$ and $\xi_i \in \mathbbm{K}^n$ pairwise distinct are linearly
  independent.
\end{lemma}

\begin{proof}
  Suppose that there exist $w_{i, j} \in \mathbbm{K}$ such that $\sigma
  (\tmmathbf{y}) = \sum_{i = 1}^r \sum_{j = 1}^{{\mu}_i} \omega_{i, j}
  \tmmathbf{y}^{\alpha_{i, j}} \tmmathbf{e}_{{\xi_i}} (\tmmathbf{y}) = 0$ and let
  $\omega_i  (\tmmathbf{y}) =\sum_{j = 1}^{{\mu}_i} \omega_{i, j}
  \tmmathbf{y}^{\alpha_{i, j}}$. Then $\forall p \in \mathbbm{K} [\tmmathbf{x}], p \star
  \sigma = 0 = \sum_{i = 1}^r p (\xi_i + \tmmathbf{\partial}) (\omega_i)
  \tmmathbf{e}_{{\xi_i}} (\tmmathbf{y})$. If the weights $\omega_i
  (\tmmathbf{y}) \in \mathbbm{K}$ are of degree $0$, by choosing for $p$ an
  interpolation polynomial at one of the distinct points $\xi_i$, we deduce
  that $\omega_i = 0$ for $i = 1, \ldots, r$. If the weights $\omega_i 
  (\tmmathbf{y}) \in \mathbbm{K}$ are degree $\geqslant 1$, by choosing $p = l
  (\tmmathbf{x}) - l (\xi_i) \in \mathbbm{K} [\tmmathbf{x}]$ for a separating
  polynomial $l$ of degree $1$ ($l (\xi_i) \neq l (\xi_j)$ if $i \neq j$), we
  can reduce to a case where at least one of the non-zero weights has one
  degree less. By induction on the degree, we deduce that $\omega_i 
  (\tmmathbf{y}) = 0$ for $i = 1, \ldots, r$. This proves the linear
  independency of the series $\tmmathbf{y}^{\alpha_{i, j}}
  \tmmathbf{e}_{{\xi_i}} (\tmmathbf{y})$ for any $\alpha_{i, 1}, \ldots,
  \alpha_{i, {\mu}_i} \in \mathbbm{N}^n$ and $\xi_i \in \mathbbm{K}^n$
  pairwise distinct. 
\end{proof}

\subsubsection{Z-transform and positive characteristic}\label{sec:ztransform}
In the identification (\ref{eq:dualfps}) of $\tmop{Hom}_{\mathbbm{K}}
(\mathbbm{K} [\tmmathbf{x}], \mathbbm{K})$ with the ring of power
series in the variables $\tmmathbf{y}$, we can replace
$\frac{\tmmathbf{y}^{\alpha}}{\alpha !}$ \ by $\tmmathbf{z}^{\alpha}$ where
$\tmmathbf{z}= (z_1, \ldots, z_n)$ is a set of new variables, so that
$\tmop{Hom}_{\mathbbm{K}} (\mathbbm{K} [\tmmathbf{x}], \mathbbm{K})$
is identified with $\mathbbm{K} [[z_1, \ldots, z_n]] =\mathbbm{K}
[[\tmmathbf{z}]]$. Any $\Lambda \in \tmop{Hom}_{\mathbbm{K}} (\mathbbm{K}
[\tmmathbf{x}], \mathbbm{K}) = R^{\ast}$ can be represented by the
series:
\begin{equation}
  \Lambda (\tmmathbf{y}) = \sum_{\alpha \in \mathbbm{N}^n} \Lambda
  (\tmmathbf{x}^{\alpha})  {\tmmathbf{z}}^{\alpha} \in \mathbbm{K}
  [[{\tmmathbf{z}}]] \label{eq:dualiny}
\end{equation}
where $(\tmmathbf{z}^{\alpha})_{\alpha \in \mathbbm{N}^n}$ denotes the basis
dual to the monomial basis $({\bdx}^{\alpha})_{\alpha \in
\mathbbm{N}^n}$ of $\mathbbm{K} [{\bdx}]$. This corresponds to the
{\tmem{Z-transform}} of the sequence $(\Lambda
(\tmmathbf{x}^{\alpha}))_{\alpha \in \mathbbm{N}^n}$
{\cite{encyclopedia_mathematics_2016}} or to the embedding in the ring of
{\tmem{divided powers}} (${\bdz^{\alpha}} =
\frac{\tmmathbf{y}^{\alpha}}{\alpha !}$)
{\cite{eisunbud_commutative_1994}}[Sec. A 2.4],
{\cite{iarrobino_power_1999}}[Appendix A]. It allows to extend the duality
properties to any field $\mathbbm{K}$, which is not of characteristic $0$.

The inverse transformation of the series $\sum_{\alpha \in \mathbbm{N}^n}
\sigma_{\alpha}  {\bdz}^{\alpha} \in \mathbbm{K} [[\tmmathbf{z}]]$ into
the series $\sum_{\alpha \in \mathbbm{N}^n} \sigma_{\alpha} 
\frac{\tmmathbf{y}^{\alpha}}{\alpha !} \in \mathbbm{K} [[\tmmathbf{y}]]$ is
known as the {\tmem{Borel transform}} {\cite{encyclopedia_mathematics_2016}}.

For $\alpha, \beta \in \mathbbm{N}^n$, we have$\begin{array}{lll}
  {\bdx}^{\alpha} \star {\bdz}^{\beta} & = &  \left\{
  \begin{array}{ll}
    {\bdz}^{\beta - \alpha} & \tmop{if} \beta - \alpha \in
    \mathbbm{N}^n\\
    0 & \tmop{otherwise}
  \end{array} \right.
\end{array}$so that $z_i$ plays the role of the inverse of $x_i$. This
explains the terminology of inverse system, introduced in
{\cite{macaulay_algebraic_1916}}. With this formalism, the variables $x_1,
\ldots, x_n$ act on the series in $\mathbbm{K} [[\tmmathbf{z}]]$
as shift operators:
\[ x_i \star \left( \sum_{\alpha \in \mathbbm{N}^n} \sigma_{\alpha} 
   {\tmmathbf{z}}^{\alpha} \right) = \sum_{\alpha \in \mathbbm{N}^n}
   \sigma_{\alpha + e_i}  {\tmmathbf{z}}^{\alpha} \]
where $e_1, \ldots, e_n$ is the canonical basis of $\mathbbm{N}^n$. Therefore,
for any $p_1, \ldots, p_n \in \mathbbm{K} [\tmmathbf{x}]$, the system of
equations
\[ p_1 \star \sigma = 0, \ldots, p_n \star \sigma = 0 \]
corresponds to a system of {\tmem{difference equations}} on $\sigma \in
\mathbbm{K} [[\tmmathbf{z}]]$.

In this setting, the evaluation $\tmmathbf{e}_{\xi}$ at a point $\xi \in
\mathbbm{K}^n$ is represented in $\mathbbm{K} [[\tmmathbf{z}]]$ by the
rational fraction $\frac{1}{\prod_{j = 1}^n (1 - \xi_j z_j)}$. The series
$\tmmathbf{y}^{\beta} \tmmathbf{e}_{\xi} \in \mathbbm{K} [[\tmmathbf{y}]]$
corresponds to the series of $\mathbbm{K} [[\tmmathbf{z}]]$
\[ \sum_{\alpha \in \mathbbm{N}^n} \frac{(\alpha + \beta) !}{\alpha !}
   \xi^{\alpha} \tmmathbf{z}^{\alpha + \beta} = \beta !\tmmathbf{z}^{\beta}
   \sum_{\alpha \in \mathbbm{N}^n}  \binom{\alpha + \beta}{\beta} \xi^{\alpha}
   \tmmathbf{z}^{\alpha} = \frac{\beta !\,\tmmathbf{z}^{\beta}
   \hspace{0.25em}}{\prod_{j = 1}^n (1 - \xi_j z_j)^{1 + \beta_j}}. \]
The reconstruction of truncated series consists then in finding points
{\tmem{$\mathbf{\xi}_1, \ldots, \mathbf{\xi}_{r'} \in \mathbbm{K}^n$}} and
finite sets $A_i$ of coefficients $\omega_{i, \alpha} \in \mathbbm{K}$ for $i
= 1, \ldots, r'$ and $\alpha \in A_i $ such that
\begin{equation}
  \sum_{\alpha \in \mathbbm{N}^n} \sigma_{\alpha}  {\bdz}^{\alpha} =
  \sum_{i = 1}^{r'} \sum_{\alpha \in A_i} \frac{\omega_{i, \alpha}
  \tmmathbf{z}^{\alpha} \hspace{0.25em}}{\prod_{i = 1}^n (1 - \xi_{i, j}
  z_j)^{1 + \alpha_j}} = \prod_{i = 1}^n \bar{z}_j \sum_{i = 1}^{r'}
  \sum_{\alpha \in A_i} \frac{\omega_{i, \alpha}}{\prod_{i = 1}^n (\bar{z}_j -
  \xi_{i, j})^{1 + \alpha_j}} \label{eq:decompseriesiny}
\end{equation}
where $\bar{z}_j = z_j^{- 1}$.

In the univariate case, this reduces to computing polynomials $\omega
\nocomma (z), \delta (z) = \prod_{i = 1}^{r'} (1 - \xi_i z)^{{\mu}_i} \in
\mathbbm{K} [z]$ with $\deg (\omega) <\deg (\delta) = \sum_i {\mu}_i = r$ such that
\[ \sum_{k \in \mathbbm{N}} \sigma_k z^k = \frac{w (z)}{\delta (z)}. \]
The decomposition can thus be computed from the {\tmem{Pad{\'e} approximant}}
of order $(r - 1, r)$ of the sequence $(\sigma_k)_{k \in \mathbbm{N}}$ (see
e.g. {\cite{gathen_modern_2013}}[chap. 5]).

Unfortunately, this representation in terms of Pad{\'e} approximant does not
extend so nicely to the multivariate case. The series $\sigma = \sum_{\alpha
\in \mathbbm{N}^n} \sigma_{\alpha}  {\bdz}^{\alpha} $ with a
decomposition of the form (\ref{eq:decompseriesiny}) correspond to the series
$\sum_{\alpha \in \mathbbm{N}^n} \sigma_{\alpha}  {\bdz}^{- \alpha}$,
which is rational of the form $\frac{\tmmathbf{z}^{\tmmathbf{1}} p(\tmmathbf{z})}{\prod q_i (z_i)}$ with
a splitable denominator where $\deg (q_i) \geqslant 1$ are univariate polynomials
(see e.g. {\cite{power_finite_1982}}, {\cite{barachart_realisation_1985}}). Though
Pad{\'e} approximants could be computed in this case by ``separating'' the
variables (or by relaxing the constraints on the Pad{\'e} approximants
{\cite{cuyt_how_1999}}), the rational fraction
$\frac{\tmmathbf{z}^{\tmmathbf{1}} p (\tmmathbf{z})}{\prod q_i (z_i)}$ is
mixing the coordinates of the points \ {\tmem{$\mathbf{\xi}_1, \ldots,
\mathbf{\xi}_{r'} \in \mathbbm{K}^n$}} and the weights $\omega_{i, \alpha}$.

As the duality between multiplication and differential operators is less
natural in $\mathbbm{K} [[\tmmathbf{z}]]$, we will use hereafter the
identification (\ref{eq:dualfps}) of $R^{*}$ with  $\mathbbm{K} [[\tmmathbf{y}]]$,
when $\mathbbm{K}$ is of characteristic $0$.

\subsection{Hankel operators}

The product $\star$ allows us to define a
Hankel operator as a multiplication operator by a dual element $\in
\mathbbm{K} [[\tmmathbf{y}]]$: 

\begin{definition}
  The Hankel operator associated to an element $\sigma (\tmmathbf{y} 
  ) \in \mathbbm{C} [[\tmmathbf{y}]]$ is
  \begin{eqnarray*}
    H_{\sigma} : \mathbbm{K} [\tmmathbf{x}] & \rightarrow &
    \mathbbm{K} [[\tmmathbf{y}]]\\
    p (\tmmathbf{x}) & \mapsto & p (\tmmathbf{x}) \star \sigma
    (\tmmathbf{y} ).
  \end{eqnarray*}
  Its kernel is denoted $I_{\sigma}=\ker H_{\sigma}$.
  The series $\sigma = \sum_{\alpha \in \mathbbm{N}^n} \sigma_{\alpha} 
  \frac{\tmmathbf{y}^{\alpha}}{\alpha !} = H_{\sigma} (1) \in \mathbbm{K}
  [[\tmmathbf{y}]]$ is called the {\tmem{symbol}} of $H_{\sigma}$.
\end{definition}

\begin{definition}
  The rank of an element $\sigma \in \mathbbm{K} [[\tmmathbf{y}]]$ is 
the rank of the Hankel operator   $H_{\sigma} = r$. 
\end{definition}

\begin{definition}
  The variety $\mathcal{V}_{\overline{\mathbbm{K}}}(I_{\sigma})$ is called the {\em characteristic variety} of $\sigma$. 
\end{definition}

The Hankel operator can also be interpreted as an operator on sequences:
\begin{eqnarray*}
  H_{\sigma} : \fsupp (\mathbbm{K}^{\mathbbm{N}^n}) & \rightarrow &
  \mathbbm{K}^{\mathbbm{N}^n}\\
  p = (p_{\beta})_{\beta \in B \subset \mathbbm{N}^n} & \mapsto & p \star
  \sigma = \left( \sum_{\beta \in B} p_{\beta} \sigma_{\alpha + \beta}
  \right)_{\alpha \in \mathbbm{N}^n}
\end{eqnarray*}
where $\fsupp (\mathbbm{K}^{\mathbbm{N}^n})$ is the set of sequences $\in
\mathbbm{K}^{\mathbbm{N}^n}$ with a finite support. This
definition applies for a field $\mathbbm{K}$ of any characteristic. The
operator $H_{\sigma}$ can also be interpreted, via the Z-transform of the
sequence $p \star \sigma$ (see Section \ref{sec:ztransform}), as the Hankel
operator
\begin{eqnarray*}
  H_{\sigma} : \mathbbm{K} [\tmmathbf{x}] & \rightarrow &
  \mathbbm{K} [[\tmmathbf{z}]]\\
  p = \sum_{\beta \in B} p_{\beta} \tmmathbf{x}^{\beta} & \mapsto & p \star
  \sigma = \sum_{\alpha \in \mathbbm{N}^n}  \left( \sum_{\beta \in B}
  p_{\beta} \sigma_{\alpha + \beta} \right) \tmmathbf{z}^{\alpha}.
\end{eqnarray*}

As $\forall p, q \in \mathbbm{K} [\tmmathbf{x}]$, \ $p q \star \sigma = p
\star (q \star \sigma)$, we easily check that $I_{\sigma} = \ker H_{\sigma}$
is an ideal of $\mathbbm{K} [\tmmathbf{x}]$ and that
$\mathcal{A}_{\sigma} =\mathbbm{K} [\tmmathbf{x}] / I_{\sigma}$ is an
algebra.

Since \ $\forall p (\tmmathbf{x}) , q (\tmmathbf{x}) \in
\mathbbm{K} [\tmmathbf{x}]$, $\langle p (\tmmathbf{x}) + I_{\sigma}, 
q (\tmmathbf{x}) + I_{\sigma} \rangle_{\sigma} = \langle p
(\tmmathbf{x}), q (\tmmathbf{x}) \rangle_{\sigma}$, we see
that \ $\langle \cdot , \cdot \rangle_{\sigma}$ induces an inner
product on $\mathcal{A}_{\sigma}$.

Given a sequence $\sigma = (\sigma_{\alpha})_{\alpha \in \mathbbm{N}^n} \in
\mathbbm{K}^{\mathbbm{N}^n}$, the kernel of $H_{\sigma}$ is the set of
polynomials $p = \sum_{\beta \in B} p_{\beta} \tmmathbf{x}^{\beta}$ such that
$\sum_{\beta \in B} p_{\beta} \sigma_{\alpha + \beta}=0$ for all $\alpha \in
\mathbbm{N}^n$. This kernel is also called the set of {\tmem{linear recurrence
relations}} of the sequence $\sigma = (\sigma_{\alpha})_{\alpha \in
\mathbbm{N}^n}$.

\begin{example}
  If $\sigma =\tmmathbf{e}_{\xi}\in \mathbbm{K}[[\tmmathbf{y}]]$ is the evaluation at a point $\xi \in
  \mathbbm{K}^n$, then $H_{\tmmathbf{e}_{\xi}} : p \in \mathbbm{K}
  [\tmmathbf{x}] \mapsto p (\xi) \tmmathbf{e}_{\xi} (\tmmathbf{y}) \in
  \mathbbm{K} [[\tmmathbf{y}]]$. We easily check that $\tmop{rank}
  H_{\tmmathbf{e}_{\xi}} = 1$ since the image of $H_{\tmmathbf{e}_{\xi}}$ is
  spanned by $\tmmathbf{e}_{\xi} (\tmmathbf{y})$ and that
  $I_{\tmmathbf{e}_{\xi}} = (x_1 - \xi_1, \ldots, x_n - \xi_n)$.
  
  If $\sigma = \sum_{i = 1}^r \hspace{0.25em} \omega_i (\tmmathbf{y})
  \tmmathbf{e}_{{\xi_i}} (\tmmathbf{y})$ then, by Lemma \ref{lem:derivative},
  the kernel $I_{\sigma}$ is the set of polynomials $p \in \mathbbm{K}
  [\tmmathbf{x}]$ such that $\forall q \in \mathbbm{K} [\tmmathbf{x}]$, $p$ is
  a solution of the following partial differential equation:
  \[ \sum_{i = 1}^r \hspace{0.25em} \omega_i (\tmmathbf{\partial}) (p q) (\xi_i) = 0. \]
\end{example}

\begin{remark}
  The matrix of the operator $H_{\sigma}$ in the basis
  $(\tmmathbf{x}^{\alpha})_{\alpha \in \mathbbm{N}^n}$ and its dual basis
  $\left( \frac{\tmmathbf{y}^{\alpha}}{\alpha !}  \right)_{\alpha \in
  \mathbbm{N}^n}$ is
\begin{eqnarray*}{}
  [H_{\sigma}] & = & (\sigma_{\alpha + \beta})_{\alpha, \beta \in
    \mathbbm{N}^n} = (\langle \sigma |  \tmmathbf{x}^{\alpha + \beta} \rangle)_{\alpha, \beta \in
    \mathbbm{N}^n}.
  \end{eqnarray*}
\end{remark}

In the case $n = 1$, the coefficients of $[H_{\sigma}]$ depends only on the sum
of the indices indexing the rows and columns, which explains why it is are
called a {\tmem{Hankel}} operator.

\subsubsection{Truncated Hankel operators}

In the sparse reconstruction problem, we are dealing with
truncated series with known coefficients $\sigma_{\alpha}$ for $\alpha$
in a subset $\tmmathbf{a}$ of $\mathbbm{N}^n$. This leads to the definition
of truncated Hankel operators.

\begin{definition}
  For two vector spaces $V, V' \subset \mathbbm{K} [\tmmathbf{x}]$ and \
  $\sigma \in \langle V \cdummy V' \rangle^{\ast} = \langle v \cdummy v' \mid
  v \in V, v' \in V' \rangle^{\ast} \subset \mathbbm{K}
  [[\tmmathbf{y}]]$, we denote by $H_{\sigma}^{V, V'}$ the following map:
  \begin{eqnarray*}
    H^{V, V'}_{\sigma} :V & \rightarrow & V'^{\ast} = \hom_{\mathbbm{K}}
    (V', \mathbbm{K})\\
    p (\tmmathbf{x}) & \mapsto & p (\tmmathbf{x}) \star \sigma
    (\tmmathbf{y} )_{| V' }.
  \end{eqnarray*}
  It is called the {\tmem{truncated Hankel operator}} on $(V, V')$.
\end{definition}

When $V' = V$, the truncated Hankel operator is also denoted $H_{\sigma}^V$.
When $V$ (resp. $V'$) is the vector space of polynomials of degree $\leqslant
d \in \mathbbm{N}$ (resp. $\leqslant d' \in \mathbbm{N}$), the truncated
operator is denoted $H_{\sigma}^{d, d'}$. If $B = \{ b_1, \ldots, b_r
\}$ (resp. $B' = \{ b_1', \ldots, b_r' \}$) is a basis of $V$ (resp.
$V'$), then the matrix of the operator $H_{\sigma}^{V, V'}$ in $B$ and the
dual basis of $B'$ is
\[ [H_{\sigma}^{B, B'}] = (\langle \sigma |   b_j b_i'  \rangle)_{1
   \leqslant i, j \leqslant r}. \]
If $B=\tmmathbf{x}^{\tmmathbf{b}}$ and $B'=\tmmathbf{x}^{\tmmathbf{b}'}$ are monomial sets, we obtain the so-called {\tmem{truncated
moment matrix}} of $\sigma$:
\begin{eqnarray*}
  {}[H_{\sigma}^{B, B'}] & = & (\sigma_{\beta + \beta'})_{\beta' \in \tmmathbf{b}', \beta
  \in \tmmathbf{b}}.
\end{eqnarray*}
When $n = 1$, this matrix is a classical Hankel matrix, which entries depend
only on the sum of the indices of the rows and columns. When $n \geqslant 2$,
we have a similar family of structured matrices whose rows and columns are
indexed by exponents in $\mathbbm{N}^n$ and whose entries depends on the sum
of the row and column indices. These structured matrices called quasi-Hankel
matrices have been studied for instance in {\cite{mp-jcomplexity-2000}}.

\subsection{Artinian algebra}

In this section, we recall the properties of Artinian algebras. Let $I \subset
\mathbbm{K} [\tmmathbf{x}]$ be an ideal and let \ $\mathcal{A}= \mathbbm{K}
[\tmmathbf{x}] / I$ be the associated quotient algebra. 

\begin{definition}
  The quotient algebra $\mathcal{A}$ is Artinian if $\dim_{\mathbbm{K}}
  (\mathcal{A}) < \infty.$
\end{definition}

Notice that if $\mathbbm{K}$ is a subfield of a field $\mathbbm{L}$, we denote
by $\mathcal{A}_{\mathbbm{L}} =\mathbbm{L} [\tmmathbf{x}] / I_{\mathbbm{L}}
=\mathcal{A} \otimes \mathbbm{L}$ where $I_{\mathbbm{L}} = I \otimes
\mathbbm{L}$ is the ideal of $\mathbbm{L} [\tmmathbf{x}]$ generated by the
elements in $I$. As the dimension does not change by extension of the scalars,
we have $\dim_{\mathbbm{K}} (\mathbbm{K} [\tmmathbf{x}] / I) =
\dim_{\mathbbm{L}} (\mathbbm{L} [\tmmathbf{x}] / I_{\mathbbm{L}}) =
\dim_{\mathbbm{L}} (\mathcal{A}_{\mathbbm{L}})$. In particular, $\mathcal{A}$
is Artinian if and only if $\mathcal{A}_{\bar{\mathbbm{K}}}$ is Artinian,
where $\bar{\mathbbm{K}}$ is the algebraic closure.
For the sake of simplicity, we are going to assume hereafter that
{\tmem{$\mathbbm{K}$ is algebraically closed}}. \

A classical result \ states that the quotient algebra $\mathcal{A}=\mathbbm{K} [\tmmathbf{x}] / I$ is finite dimensional, i.e.
Artinian , if and only if, $\mathcal{V}_{\bar{\mathbbm{K}}} (I)$ is finite, that is, $I$
defines a finite number of (isolated) points in $\bar{\mathbbm{K}}^n$ (see
e.g. {\cite{cox_ideals_2015}}[Theorem 6] or {\cite{ElkMou07}}[Theorem 4.3]).
Moreover, we have the following structure theorem (see
{\cite{ElkMou07}}[Theorem 4.9]):

\begin{theorem}
  \label{thm:artindec}Let $\mathcal{A}=\mathbbm{K} [\tmmathbf{x}] / I$ be an
  Artinian algebra of dimension $r$ defined by an ideal $I$. Then we have a
  decomposition into a direct sum of subalgebras
  
  \begin{equation}
    \mathcal{A}=\mathcal{A}_{\xi_1} \oplus \cdots \oplus
    \mathcal{A}_{\xi_{r'}} \label{eq:decsubalg}
  \end{equation}
  where
  \begin{itemize}
    \item $\mathcal{V} (I) = \{ \xi_1, \ldots, \xi_{r'} \} \subset
    \bar{\mathbbm{K}}^n$ with $r' \leqslant r$.
    
    \item $I = Q_1 \cap \cdots \cap Q_{r'}$ is a minimal primary decomposition
    of $I$ where $Q_i$ is $\tmmathbf{m}_{\xi_i}$-primary with
    $\tmmathbf{m}_{\xi_i} = (x_1 - \xi_{i, 1}, \ldots, x_n - \xi_{i, n})$.
    
    \item  $\mathcal{A}_{\xi_i} \equiv \mathbbm{K} [\tmmathbf{x}] / Q_i$ and
    $\mathcal{A}_{\xi_i} \cdot \mathcal{A}_{\xi_j} \equiv 0$ if $i \neq j$.
  \end{itemize}
\end{theorem}

We check that $\mathcal{A}$ localized at $\tmmathbf{m}_{\xi_i}$ is the local algebra
$\mathcal{A}_{\xi_i}$. The dimension of $\mathcal{A}_{\xi_i}$ is the
{\tmem{multiplicity}} of the point $\xi_i$ in $\mathcal{V} (I)$.

The projection of $1$ on the sub-algebras $\mathcal{A}_{\xi_i}$ as
\[ 1 \equiv {\bdu}_{\xi_1} + \cdots +\tmmathbf{u}_{\xi_{r'}} \]
with ${\bdu}_{\xi_i} \in \mathcal{A}_{\xi_i}$ yields the so-called
{\tmem{idempotents}} $\tmmathbf{u}_{\xi_i}$ associated to the roots $\xi_i$.
By construction, they satisfy the following relations in $\mathcal{A}$, which
characterize them:
\begin{itemize}
  \item $1 \equiv {\bdu}_{\xi_1} + \cdots +\tmmathbf{u}_{\xi_{r'}}$,
  
  \item ${\bdu}_{\xi_i}^2 \equiv {\bdu}_{\xi_i}$ for $i = 1,
  \ldots, r'$,
  
  \item $ {\bdu}_{\xi_i} {\bdu}_{\xi_j} \equiv 0$ for $1
  \leqslant i, j \leqslant r'$ and $i \neq j$.
\end{itemize}
The dual $\mathcal{A}^{\ast} = \tmop{Hom}_{\mathbbm{K}} (\mathcal{A},
\mathbbm{K})$ of $\mathcal{A}=\mathbbm{K} [\tmmathbf{x}] / I$ is naturally
identified with the subspace
\[ I^{\bot} = \{ \Lambda \in \mathbbm{K} [\tmmathbf{x}]^{\ast} =\mathbbm{K}
   [[\tmmathbf{y}]] \mid \forall p \in I, \Lambda (p) = 0 \}  \]
of  $\mathbbm{K} [\tmmathbf{x}]^{\ast} =\mathbbm{K} [[\tmmathbf{y}]]$.
As $I$ is stable by multiplication by the variables $x_i$, the
orthogonal subspace $I^{\bot} =\mathcal{A}^{\ast}$ is stable by the derivations
$\frac{d}{\tmop{dy}_i}$: $\forall \Lambda\in I^{\perp}\subset \mathbbm{K}[[\tmmathbf{y}]]$, $\forall\, i=1\ldots n,$ $\frac{d}{\tmop{dy}_i}(\Lambda)\in I^{\perp}$. In the case of a primary ideal, the orthogonal subspace has a
simple form {\cite{macaulay_algebraic_1916}}, {\cite{emsalem_geometrie_1978}},
{\cite{mourrain_isolated_1996}}:

\begin{proposition}
  \label{prop:primaryortho}Let $Q$ be a primary ideal for the maximal ideal
  $\mathfrak{m}_{\xi}$ of the point $\xi \in \mathbbm{K}^n$ and let
  $\mathcal{A}_{\xi} =\mathbbm{K} [\tmmathbf{x}] / Q$. Then
  \[ Q^{\bot} =\mathcal{A}_{\xi}^{\ast} = D_{\xi} (Q) \cdummy
     \tmmathbf{e}_{\xi} (\tmmathbf{y}), \]
  where $D_{\xi} (Q) \subset \mathbbm{K} [\tmmathbf{y}]$ is the set of
  polynomials $\omega (\tmmathbf{y}) \in \mathbbm{K} [\tmmathbf{y}]$ such that
  $\forall q \in Q$, $\omega (\tmmathbf{\partial}) (q) (\xi) = 0$.
\end{proposition}

The vector space $D_{\xi} (Q) \subset \mathbbm{K} [\tmmathbf{y}]$ is called
the {\tmem{inverse system}} of $Q$. As $Q$ is an ideal, $Q^{\bot}
= D_{\xi} (Q) \cdummy \tmmathbf{e}_{\xi} (\tmmathbf{y})$ is stable by the
derivations $\frac{d}{\tmop{dy}_i}$, and so is $D_{\xi} (Q)$.

If $I = Q_1 \cap \cdots \cap Q_{r'}$ is a minimal primary decomposition of
the zero-dimensional ideal $I \subset \mathbbm{K} [\tmmathbf{x}]$ with $Q_i$
$\tmmathbf{m}_{\xi_i}$-primary, then the decomposition (\ref{eq:decsubalg})
implies that
\[ \mathcal{A}^{\ast} = I^{\bot} = Q_1^{\bot} \oplus \cdots \oplus
   Q_{r'}^{\bot} =\mathcal{A}_{\xi_1}^{\ast} \oplus \cdots \oplus
   \mathcal{A}_{\xi_{r'}}^{\ast}\]
since $Q_{i}^{\bot} \cap Q_j^{\bot} = (Q_i + Q_j)^{\bot} = (1)^{\bot} = \{
0 \}$ for $i \neq j$. 
The elements of $\mathcal{A}_{\xi_i}^{\ast}$ are the elements $\Lambda\in \mathcal{A}^{\ast}=I^{\bot}$ such that
$\forall a_{j}\in \mathcal{A}_{j}$ $j\neq i$,  $\lambda(a_{j})=0$.
Any element $\Lambda \in \mathcal{A}^{\ast}$ decomposes
as
\begin{equation}
  \Lambda = {\bdu}_{\xi_1} \star \Lambda + \cdots + {\bdu}_{\xi_{r'}} \star \Lambda. \label{eq:dualdec}
\end{equation}
As we have ${\bdu}_{\xi_i} \star \Lambda
(\mathcal{A}_{\xi_j}) = \Lambda ({\bdu}_{\xi_i}
\mathcal{A}_{\xi_j}) = 0$ for $i \neq j$, we deduce that ${\bdu}_{\xi_i} \star \Lambda \in \mathcal{A}_{\xi_i}^{\ast} = Q_i^{\bot}$

By Proposition \ref{prop:primaryortho}, $Q_i^{\perp} = D_i \tmmathbf{e}_{\xi}
(\tmmathbf{y})$ where $D_i = D_{\xi_i} (Q_i)$ is the inverse system of $Q_i$.
A basis of $D_i$ can be computed (efficiently) from $\xi_i$ and the
generators of the ideal $I$ (see e.g. {\cite{mourrain_isolated_1996}} for more
details).

From the decomposition (\ref{eq:dualdec}) and Proposition
\ref{prop:primaryortho}, we deduce the following result:

\begin{theorem}\label{thm:artindual}
Assume that $\mathbbm{K}$ is algebraically closed.
Let $\mathcal{A}$ be an Artinian algebra of dimension
  $r$ with $\mathcal{V} (I) = \{ \xi_1, \ldots, \xi_{r'} \} \subset
  \mathbbm{K}^n$. Let $D_i = D_{\xi_i} (I) \subset \mathbbm{K} [\tmmathbf{y}]$
  be the vector space of differential polynomials $\omega (\tmmathbf{y}) \in
  \mathbbm{K} [\tmmathbf{y}]$ such that $\forall p \in I, \omega (\tmmathbf{\partial}) (p) (\xi_i) = 0$. Then $D_i$ is stable by
  the derivations $\frac{d}{\tmop{dy}_i}$, $i = 1, \ldots, n$. It is of
  dimension ${\mu}_i$ with $\sum_{i = 1}^{r'} {\mu}_i = r$. Any
  element $\Lambda$ of $\mathcal{A}^{\ast}$ has a unique decomposition of the
  form
  \begin{equation}
    \Lambda (\tmmathbf{y}) = \sum_{i = 1}^{r'} \omega_i (\tmmathbf{y})
    \tmmathbf{e}_{\xi_i} (\tmmathbf{y}), \label{eq:polyexp}
  \end{equation}
  with $\omega_i (\tmmathbf{y}) \in D_i \subset \mathbbm{K}
  [\tmmathbf{y}]$, which is uniquely determined by values $\langle \Lambda | 
   b_i \rangle$ for a basis $B = \{ b_1, \ldots, b_r \}$ of
  $\mathcal{A}$. Moreover, any element of this form is in
  $\mathcal{A}^{\ast}$.
\end{theorem}

\begin{proof}
  For any polynomial $\omega (\tmmathbf{y}) \in \mathbbm{K} [\tmmathbf{y}]$,
  such that $\forall \xi \in \mathcal{V} (I)$, $\forall p \in I,
  \omega (\tmmathbf{\partial}) (p) (\xi) = 0$, the element $\omega (\tmmathbf{y})
  \tmmathbf{e}_{\xi} (\tmmathbf{y})$ is in $I^{\bot}$. Thus an element of the
  form (\ref{eq:polyexp}) is in $I^{\perp} =\mathcal{A}^{\ast}$.
  
  Let us prove that any element $\Lambda \in \mathcal{A}^{\ast}$ is of the
  form (\ref{eq:polyexp}). By the relation (\ref{eq:dualdec}), $\Lambda$
  decomposes as $\Lambda = \sum_{i = 1}^{r'} {\bdu}_{\xi_i} \star
  \Lambda$ with ${\bdu}_{\xi_i} \star \Lambda \in
  \mathcal{A}_{\xi_i}^{\ast} = Q_i^{\bot}$. By Proposition
  \ref{prop:primaryortho}, $Q_i^{\bot} = D_i \tmmathbf{e}_{\xi_i}
  (\tmmathbf{y})$, where $D_i = D_{\xi_i} (Q_i)$ is the set of differential
  polynomials which vanish at $\xi_i$, on $Q_i$ and thus on $I$. Thus
  ${\bdu}_{\xi_i} \star \Lambda$ is of the form ${\bdu}_{\xi_i}
  \star \Lambda = \omega_i (\tmmathbf{y}) \tmmathbf{e}_{\xi_i} (\tmmathbf{y})$
  with $\omega_i (\tmmathbf{y}) \in D_i \subset \mathbbm{K}
  [\tmmathbf{y}]$. By Lemma \ref{lem:indep}, its decomposition as a sum of
  polynomial exponentials $\Lambda (\tmmathbf{y}) = \sum_{i =
  1}^{r'} \omega_i (\tmmathbf{y}) \tmmathbf{e}_{\xi_i} (\tmmathbf{y})$ is
  unique. This concludes the proof.
\end{proof}

Theorem \ref{thm:artindual} can be reformulated in terms of solutions of
partial differential equations, using the relation between Artinian algebras
and polynomial-exponentials $\PolExp$. This duality between
polynomial equations and partial differential equations with constant
coefficients goes back to {\cite{riquier_les_1910}} and has been further
studied and extended for instance in {\cite{grobner_uber_1937}},
{\cite{emsalem_geometrie_1978}}, {\cite{pedersen_basis_1999}},
{\cite{oberst_constructive_2001}}, {\cite{hakopian_partial_2004}}. In the case
of a non-Artinian algebra, the solutions on an open convex domain are in the
closure of the set of polynomial-exponential solutions
(see e.g. {\cite{malgrange_existence_1956}}[Th{\'e}or{\`e}me 2] or
{\cite{hormander_introduction_1990}}[Theorem 7.6.14]). The following result
gives an explicit description of the solutions of partial differential
equations associated to Artinian algebras, as special elements of
$\PolExp$, with polynomial weights in the inverse systems of
the points of the characteristic variety of the differential system:

\begin{theorem}
  \label{thm:edp}Let $p_1, \ldots, p_s \in \mathbbm{C} [x_1, \ldots, x_n]$ be
  polynomials such that $\mathbbm{C} [\tmmathbf{x}] / (p_1, \ldots, p_s)$ is
  finite dimensional over $\mathbbm{C}$. Let $\Omega \subset \mathbbm{R}^n$ be
  a convex open domain of $\mathbbm{R}^n$. A function $f \in C^{\infty}
  (\Omega)$ is a solution of the system of partial differential equations
  \begin{equation}
    p_1 (\tmmathbf{\partial}) (f) = 0, \ldots, p_s (\tmmathbf{\partial}) (f) = 0 \label{eq:pde}
  \end{equation}
  if and only if it is of the form
  \[ f (\tmmathbf{y}) = \sum_{i = 1}^r \omega_i (\tmmathbf{y}) e^{\xi_i \cdot
     \tmmathbf{y}}\]
  with $\mathcal{V}_{\mathbbm{C}} (p_1, \ldots, p_s) = \{ \xi_1, \ldots, \xi_r
  \} \subset \mathbbm{C}^n$ and $\omega_i (\tmmathbf{y}) \in D_i \subset
  \mathbbm{C} [\tmmathbf{y}]$ where $D_i = D_{\xi_i} ((p_1, \ldots, p_s))$ is
  the space of differential polynomials, which vanish on the ideal $(p_1,
  \ldots, p_s)$ at $\xi_i$.
\end{theorem}

\begin{proof}
  By a shift of the variables, we can assume that $\Omega$ contains $0$. A
  solution of $f$ of (\ref{eq:pde}) in $C^{\infty} (\Omega)$ has a Taylor
  series expansion $f (\tmmathbf{y}) \in \mathbbm{C} [[\tmmathbf{y}]]$ at $0
  \in \Omega$, which defines an element of $\mathbbm{C}
  [\tmmathbf{x}]^{\ast}$. By Lemma \ref{lem:derivative}, $f$ is a solution of
  the system (\ref{eq:pde}) if and only if we have $p_1 \star f (\tmmathbf{y})
  = 0, \ldots, p_s \star f (\tmmathbf{y}) = 0$. Equivalently, $f
  (\tmmathbf{y}) \in I^{\bot}$ where $I = (p_1, \ldots, p_s)$ is the ideal of
  $\mathbbm{K} [\tmmathbf{x}]$ generated by $p_1, \ldots, p_s$. If
  $\mathcal{A}=\mathbbm{C} [\tmmathbf{x}] / I$ is finite dimensional, i.e.
  Artinian, Theorem \ref{thm:artindual} implies that the Taylor series $f
  (\tmmathbf{y})$ is in $I^{\perp}$, if and only if, it is of the form:
  \begin{equation}
    f (\tmmathbf{y}) = \sum_{i = 1}^r \omega_i (\tmmathbf{y}) e^{\xi_i \cdot
    \tmmathbf{y}}\label{eq:solpolexp}
  \end{equation}
  with $\mathcal{V}_{\mathbbm{C}} (p_1, \ldots, p_s) = \{ \xi_1, \ldots, \xi_r
  \} \subset \mathbbm{C}^n$ and $\omega_i (\tmmathbf{y}) \in D_i = D_{\xi_i}
  (I) \subset \mathbbm{C} [\tmmathbf{y}]$ where $D_i$ is the space of
  differential polynomials which vanish on $I = (p_1, \ldots, p_s)$ at
  $\xi_i$. The polynomial-exponential function
  (\ref{eq:solpolexp}) is an analytic function with an infinite radius of
  convergence, which is a solution of the partial differential system
  (\ref{eq:pde}) on $\Omega$. By unicity of the solution with given
  derivatives at $0 \in \Omega$, $\sum_{i = 1}^r \omega_i (\tmmathbf{y})
  e^{\xi_i \cdot \tmmathbf{y}}$ coincides with $f$ on
  all the domain $\Omega \subset \mathbbm{R}^n$.
\end{proof}

Here is another reformulation of Theorem \ref{thm:artindual} in
terms of {\tmem{convolution}} or {\tmem{cross-correlation}} of sequences:

\begin{theorem}
  \label{thm:convol}Let $p_1, \ldots, p_s \in \mathbbm{C} [x_1, \ldots, x_n]$
  be polynomials such that $\mathbbm{C} [\tmmathbf{x}] / (p_1, \ldots, p_s)$
  is finite dimensional over $\mathbbm{C}$. The generating series of the
  sequences $\sigma = (\sigma_{\alpha}) \in \mathbbm{C}^{\mathbbm{N}^n}$ which
  satisfy the system of difference equations
  \begin{equation}
    p_1 \star \sigma = 0, \ldots, p_s \star \sigma = 0 \label{eq:convol}
  \end{equation}
  are of the form 
  \[ \sigma (\tmmathbf{y}) = \sum_{\alpha \in \mathbbm{N}^n} \sigma_{\alpha} 
     \frac{\tmmathbf{y}^{\alpha}}{\alpha !} = \sum_{i = 1}^r \omega_i
     (\tmmathbf{y}) e^{\xi_i \cdot \tmmathbf{y}}\]
  with $\mathcal{V}_{\mathbbm{C}} (p_1, \ldots, p_s) = \{ \xi_1, \ldots, \xi_r
  \} \subset \mathbbm{C}^n$ and $\omega_i (\tmmathbf{y}) \in D_i \subset
  \mathbbm{C} [\tmmathbf{y}]$ such that $D_i = D_{\xi_i} ((p_1, \ldots, p_s))$
  is the space of differential polynomials, which vanish on the ideal $(p_1,
  \ldots, p_s)$ at $\xi_i$.
\end{theorem}

\begin{proof}
  The sequence $\sigma$ is a solution of the system (\ref{eq:convol}) if and
  only if $\sigma (\tmmathbf{y}) = \sum_{\alpha \in \mathbbm{N}^n}
  \sigma_{\alpha}  \frac{\tmmathbf{y}^{\alpha}}{\alpha !} \in I^{\bot}$ where
  $I = (p_1, \ldots, p_s)$ is the ideal of $\mathbbm{K} [\tmmathbf{x}]$
  generated by $p_1, \ldots, p_s$. We deduce the form of $\sigma
  (\tmmathbf{y}) \in \PolExp (\tmmathbf{y})$ from Theorem
  \ref{thm:artindual}.
\end{proof}

The solutions $\mathcal{V} (I) = \{ \xi_1, \ldots, \xi_{r'} \}$ can be
recovered by linear algebra, from the multiplicative structure of
$\mathcal{A}$, using the properties of the following operators:

\begin{definition}
  Let $g$ be a polynomial in $\mathcal{A}$. The $g$-multiplication operator
  $\mathcal{M}_g$ is defined by
  \[ \begin{array}{cccl}
       \mathcal{M}_g : & \mathcal{A} & \to & \mathcal{A}\\
       & h & \mapsto & \mathcal{M}_g (h) = gh.
     \end{array} \]
  The transpose application $\mathcal{M}_g^t$ of the $g$-multiplication
  operator $\mathcal{M}_g$ is defined by
  \[ \begin{array}{cccl}
       \mathcal{M}_g^t : & \mathcal{A}^{\ast} & \to & \mathcal{A}^{\ast}\\
       & \Lambda & \mapsto & \mathcal{M}_g^t (\Lambda) = \Lambda \circ
       \mathcal{M}_g = g \star \Lambda.
     \end{array} \]
\end{definition}

Let $B = \{ b_1, \ldots, b_r \}$ be a basis in $\mathcal{A}$ and $B^{\ast}$
its dual basis in $\mathcal{A}^{\ast}$. We denote by $M_g^B$ \ (or simply
$M_g$ when there is no ambiguity on the basis) the matrix of $\mathcal{M}_g$
in the basis $B$. As the matrix $(M^B_g)^t$ of the transpose application
$\mathcal{M}_g^t$ in the dual basis $B^{\ast}$ in $\mathcal{A}^{\ast}$ is the
transpose of the matrix $M_g^B$ of the application $\mathcal{M}_g$ in the
basis $B$ in $\mathcal{A}$, the eigenvalues are the same for both matrices.

The main property we will use is the following (see e.g. {\cite{ElkMou07}}):

\begin{proposition}
  \label{prop:eigen}Let $I$ be an ideal of $R =\mathbbm{K} [\tmmathbf{x}]$ and
  suppose that $\mathcal{V} (I) = \{ \mathbf{\xi}_1, \mathbf{\xi}_2, \ldots,
  \mathbf{\xi}_r \}$. Then
  \begin{itemize}
    \item for all $g \in \mathcal{A}$, the eigenvalues of $\mathcal{M}_g$ and
    $\mathcal{M}_g^t$ are the values $g (\mathbf{\xi}_1), \ldots, g
    (\mathbf{\xi}_r)$ of the polynomial $g$ at the roots with multiplicities
    ${\mu}_i = \dim \mathcal{A}_{x_i}$.
    
    \item The eigenvectors common to all $\mathcal{M}_g^t$ with $g \in
    \mathcal{A}$ are - up to a scalar - the evaluations
    $\mathbf{\tmmathbf{e}}_{\mathbf{\xi}_1}, \ldots,
    \mathbf{\tmmathbf{e}}_{\mathbf{\xi}_r}$.
  \end{itemize}
\end{proposition}

\begin{remark}
  \label{rem:evalcoef}If $B = \{ b_1, \ldots, b_r \}$ is a basis of
  $\mathcal{A}$, then the coefficient vector of the evaluation
  \[ {\mathbf{e}}_{\mathbf{\xi}_i} = \sum_{\beta \in \mathbbm{N}^n}
     \mathbf{\xi}_i^{\beta}  \frac{\tmmathbf{y}^{\beta}}{\beta !} + \cdots \]
  in the dual basis of $\mathcal{A}^{\ast}$ is $\left[ \left\langle
  \tmmathbf{e}_{{\xi_i}} |   b_j \right\rangle
  \right]_{\beta \in B} = [b_j (\mathbf{\xi}_i)]_{i = 1 \ldots r} = B
  (\xi_i)$. The previous proposition says that if $M_g$ is the matrix of
  $\mathcal{M}_g$ in the basis $B$ of $\mathcal{A}$, then
  \[ M_g^t B (\mathbf{\xi}_i) = g (\xi_i) B (\mathbf{\xi}_i). \]
  If moreover the basis $B$ contains the monomials $1, x_1, x_2, \ldots, x_n$,
  then the common eigenvectors of $M_g^t$ are of the form ${\bdv}_i = c
  [1, \xi_{i, 1}, \ldots, \xi_{i, n}, \ldots]$ and the root $\xi_i$ can be
  computed from the coefficients of ${\bdv}_i$ by taking the ratio of
  the coefficients of the monomials $\nocomma x_1, \ldots, x_n$ by the
  coefficient of $1$: $\xi_{i, k} = \frac{{\bdv}_{i, k + 1}}{{\bdv}_{i, 1}}$. Thus computing the common eigenvectors of all the
  matrices $M_g^t$ for $g \in \mathcal{A}$ yield the roots $\tmmathbf{\xi}_i$
  ($i = 1, \ldots, r$). In practice, it is enough to compute the
  common eigenvectors of $M_{x_1}^t, \ldots, M_{x_n}^t$,
  since $\forall g \in \mathbbm{K} [\tmmathbf{x}], M_{g}^t = g (M_{x_1}^t,  \ldots, M_{x_n}^t)$.
\end{remark}

We are going to consider special Artinian algebras, called
{\tmem{Gorenstein}} algebras. They are defined as follows:

\begin{definition}
  A $\mathbbm{K}$-algebra $\mathcal{A}$ is Gorenstein if $\exists \sigma \in
  \mathcal{A}^{\ast} = \tmop{Hom}_{\mathbbm{K}} (\mathcal{A}, \mathbbm{K})$
  such that $\forall \Lambda \in \mathcal{A}^{\ast}, \exists a \in
  \mathcal{A}$ with $\Lambda = a \star \sigma$ and $a \star \sigma = 0$
  implies $a = 0$.
\end{definition}

In other words, $\mathcal{A}$ is Gorenstein, if and only if, $\mathcal{A}^{\ast}$ is a free
$\mathcal{A}$-module of rank 1.

\section{Correspondence between Artinian Gorenstein algebras and
$\PolExp$}\label{sec:aga}

In this section, we describe how polynomial-exponential functions
are naturally associated to Artinian Gorenstein Algebras. As this property is
preserved by tensorisation by $\bar{\mathbbm{K}}$, we will also assume
hereafter that $\mathbbm{K}= \bar{\mathbbm{K}}$ is {\tmem{algebraically
closed}}.

Given $\sigma \in \mathbbm{K} [[{\bdy}]]$, we consider its
Hankel operator $H_{\sigma} : p \in \mathbbm{K} [\tmmathbf{x}] \mapsto p \star
\sigma \in \mathbbm{K} [[{\bdy}]]$. The kernel $I_{\sigma}$ of
$H_{\sigma}$ is an ideal and the elements $p \star \sigma$ of $\tmop{im}
H_{\sigma}$ for $p \in \mathbbm{K} [\tmmathbf{x}]$ are in $I_{\sigma}^{\perp}
= \mathcal{A}_{\sigma}^{\ast}$ where $\mathcal{A}_{\sigma} =\mathbbm{K}
[\tmmathbf{x}] / I_{\sigma}$: $\forall q \in I_{\sigma}$, $\langle p \star
\sigma \mid q \rangle = \langle q \star \sigma \mid p \rangle = 0$. If
$\mathcal{A}_{\sigma}$ is Artinian of dimension $r$, then
\[ \tmop{im} H_{\sigma} = \{ p \star \sigma \mid p \in R \} \subset
   I_{\sigma}^{\perp} = \mathcal{A}_{\sigma}^{\ast} \]
is of dimension $\leqslant r$. Therefore, the injective map
\begin{eqnarray*}
  \mathcal{H}_{\sigma} : \mathcal{A}_{\sigma} & \rightarrow &
  \mathcal{A}^{\ast}_{\sigma}\\
  p (\tmmathbf{x}) & \mapsto & p (\tmmathbf{x}) \star \sigma
  (\tmmathbf{y} )
\end{eqnarray*}
induced by $H_{\sigma}$ is an isomorphism, and we have the exact
sequence:
\begin{equation}
  0 \rightarrow I_{\sigma} \rightarrow \mathbbm{K} [\tmmathbf{x}]
  \xrightarrow{H_{\sigma}} \mathcal{A}_{\sigma}^{\ast} \rightarrow 0.
  \label{eq:seq}
\end{equation}
Conversely, let $\mathcal{A}$ \ be an Artinian Gorenstein Algebra of dimension
$r$, generated by $n$ elements $a_1, \ldots, a_n$. It can be represented as
the quotient algebra $\mathcal{A} =\mathbbm{K} [\tmmathbf{x}] / I$ of
$\mathbbm{K} [\tmmathbf{x}]$ by an ideal $I$. As $\mathcal{A}$ is Gorenstein, there
exists $\sigma \in \mathcal{A}^{\ast}$ such that $\sigma$ is a basis of the
free $\mathcal{A}$-module $\mathcal{A}^{\ast}$. This implies that the kernel
of the map
\begin{eqnarray*}
  H_{\sigma} :\mathbbm{K} [\tmmathbf{x}] & \rightarrow &
  \mathcal{A}^{\ast}_{\sigma}\\
  p & \mapsto & p \star \sigma
\end{eqnarray*}
is the ideal $I$. Thus $\mathcal{A}=\mathcal{A}_{\sigma}$ and $H_{\sigma}$ is
of finite rank $r = \dim \mathcal{A}$.

This construction defines a correspondence between series $\sigma \in
\mathbbm{K} [[\tmmathbf{y}]]$ of finite rank or Hankel operators $H_{\sigma}$
of finite rank and Artinian Gorenstein Algebras.

\subsection{Hankel operators of finite rank}

Hankel operators of finite rank play an important role in functional analysis.
In one variable, they are characterized by Kronecker's theorem
{\cite{kronecker_zur_1880}} as follows (see e.g.
{\cite{peller_excursion_1998}} for more details). Let $\fsupp
(\mathbbm{K}^{\mathbbm{N}})$ be the vector space of sequences $\in
\mathbbm{K}^{\mathbbm{N}}$ of finite support and let $\sigma = (\sigma_k)_{k
\in \mathbbm{N}} \in \mathbbm{K}^{\mathbbm{N}}$. The Hankel operator
$H_{\sigma} : (p_l)_{l \in \mathbbm{N}} \in \fsupp (\mathbbm{K}^{\mathbbm{N}})
\mapsto \left( \sum_l \sigma_{k + l} p_l \right)_{k \in \mathbbm{N}} \in
\mathbbm{K}^{\mathbbm{N}}$ is of finite rank $r$, if and only if, there exist
polynomials $\omega_1 (u), \ldots, \omega_{r'} (u) \in \mathbbm{K} [u]$ and
$\xi_1, \ldots, \xi_{r'} \in \mathbbm{K}$ distinct such that
\[ \sigma_k = \sum_{i = 1}^{r'} \omega_i (k) \xi_i^k, \]
with $\sum_{i = 1}^{r'} \deg (\omega_i) + 1 = \tmop{rank}
H_{\sigma}$. Rewriting it in terms of generating
series, we have $H_{\sigma} : p = \sum_l p_l x^l \in \mathbbm{K} [x] \mapsto
\sum_{k \in \mathbbm{N}} \left( \sum_l \sigma_{k + l} p_l \right) 
\frac{y^k}{k!} = p \star \sigma$ is of finite rank, if and only if,
\[ \sigma (y) = \sum_{k \in \mathbbm{N}} \sigma_k  \frac{y^k}{k!} = \sum_{i =
   1}^{r'} \omega_i (y) e^{\xi_i y} \]
with $\omega_1, \ldots, \omega_{r'} \in \mathbbm{K} [y]$ and $\xi_1, \ldots,
\xi_{r'} \in \mathbbm{K}$ distinct such that $\sum_{i = 1}^r \deg (\omega_i) +
1 = \tmop{rank} H_{\sigma}$. Notice that $\deg (\omega_i) + 1$ is the
dimension of the vector space spanned by $\omega_i (y)$ and all its
derivatives.

The following result generalizes Kronecker's theorem,
by establishing a correspondence between Hankel operators of finite rank and
polynomial-exponential series and by connecting directly the rank of the Hankel
operator with the decomposition of the associated polynomial-exponential series.
\begin{theorem}
  \label{thm:gorenstein}Let $\sigma (\tmmathbf{y} ) \in
  \mathbbm{K} [[\tmmathbf{y}]]$. Then $\tmop{rank} H_{\sigma} < \infty$, if
  and only if, $\sigma \in \PolExp (\tmmathbf{y})$.
  
  If $\sigma (\tmmathbf{y}) = \sum_{i = 1}^{r'} \hspace{0.25em} \omega_i
  (\tmmathbf{y}) \tmmathbf{e}_{{\xi_i}} (\tmmathbf{y})$ $\text{with }
  \omega_i (\tmmathbf{y}) \in \mathbbm{K} [\tmmathbf{y}] \setminus \{0\}
  \tmop{and} \xi_i \in \mathbbm{K}^n  \text{pairwise distinct}$, then the rank
  of $H_{\sigma}$ is $r = \sum_{i = 1}^{r'} {\mu} (\omega_i)$ where
  ${\mu} (w_i)$ is the dimension of the inverse system spanned by
  $\hspace{0.25em} \omega_i (\tmmathbf{y})$ and all its derivatives
  $\tmmathbf{\partial}^{\alpha} \omega_i (\tmmathbf{y})$ for
  $\alpha = (\alpha_1, \ldots, \alpha_n) \in \mathbbm{N}^n$.
\end{theorem}

\begin{proof}
  If $H_{\sigma}$ is of finite rank $r$, then $\mathcal{A}_{\sigma}
  =\mathbbm{K} [\tmmathbf{x}] / I_{\sigma} =\mathbbm{K} [\tmmathbf{x}] / \ker
  H_{\sigma} \sim \tmop{Im} (H_{\sigma})$ is of dimension $r$ and
  $\mathcal{A}_{\sigma}$ is an Artinian algebra. By Theorem
  \ref{thm:artindec}, it can be decomposed as a direct sum of sub-algebras
  \[ \mathcal{A}_{\sigma} =\mathcal{A}_{\xi_1} \oplus \cdots \oplus
     \mathcal{A}_{\xi_{r'}} \]
  where $I_{\sigma} = Q_1 \cap \cdots \cap Q_r$ is a minimal primary
  decomposition, $\mathcal{V} (I_{\sigma}) = \{ \xi_1, \ldots, \xi_{r'} \}$
  and $\mathcal{A}_{\xi_i}$ is the local algebra for the maximal ideal
  $\tmmathbf{m}_{{\zeta i}}$ defining the root $\xi_i \in \mathbbm{K}^n$, such that
  $\mathcal{A}_{\xi_i} \equiv \mathbbm{K} [\tmmathbf{x}] / Q_i$ where $Q_i$ is a
  $\tmmathbf{m}_{\xi_i}$-primary component of $I_{\sigma}$.
  
  By Theorem \ref{thm:artindual}, the series $\sigma \in
  \mathcal{A}_{\sigma}^{\ast} = I_{\sigma}^{\bot}$ can be decomposed as
  \begin{equation}
    \sigma = \sum_{i = 1}^{r'} \omega_i (\tmmathbf{y}) \tmmathbf{e}_{{\xi_i}}
    (\tmmathbf{y}) \label{eq:decdualsubalg}
  \end{equation}
  with $\omega_i (\tmmathbf{y}) \in \mathbbm{K} [\tmmathbf{y}]$ and
  $\hspace{0.25em} \omega_i (\tmmathbf{y}) \tmmathbf{e}_{{\xi_i}}
  (\tmmathbf{y}) \in \mathcal{A}_{\xi_i}^{\ast} = Q_i^{\bot}$, i.e. $\sigma
  \in \PolExp (\tmmathbf{y})$.
  
  Conversely, let us show that if $\sigma (\tmmathbf{y}) = \sum_{i = 1}^{r'}
  \hspace{0.25em} \omega_i (\tmmathbf{y}) \tmmathbf{e}_{{\xi_i}}
  (\tmmathbf{y})$with $\omega_i (\tmmathbf{y}) \in \mathbbm{K}
  [\tmmathbf{y}] \setminus \{0\}$ and $\xi_i \in \mathbb{\mathbbm{K}}^n$
  pairwise distinct, the rank of $H_{\sigma}$ is finite. Using Lemma
  \ref{lem:diffpolexp}, we check that $I_{\sigma} = \ker H_{\sigma}$ contains
  $\cap_{i = 1}^r \tmmathbf{m}_{\xi_i}^{d_{i + 1}}$ where $d_i$ is the degree
  of $\omega_i (\tmmathbf{y})$. Thus $\mathcal{V} (I_{\sigma}) \subset \{
  \xi_1, \ldots, \xi_{r'} \}$, $\mathcal{A}_{\sigma}$ is an Artinian algebra
  and $\tmop{rank} H_{\sigma} = \dim (\tmop{Im} (H_{\sigma})) = \dim
  (\mathbbm{K} [\tmmathbf{x}] / I_{\sigma}) = \dim (\mathcal{A}_{\sigma}) <
  \infty$.
  
  Let us show now that rank $H_{\sigma} = \sum_{i = 1}^{r'} {\mu}
  (\omega_i)$. From the decomposition (\ref{eq:dualdec}) and Proposition
  \ref{prop:primaryortho}, we deduce that $\tmmathbf{u}_{\xi_i} \star \sigma =
  \omega_i (\tmmathbf{y}) \tmmathbf{e}_{{\xi_i}} (\tmmathbf{y})$. By the
  exact sequence (\ref{eq:seq}), $\mathcal{A}_{\sigma}^{\ast} = \tmop{Im}
  (H_{\sigma}) = \{ p \star \sigma \mid p \in \mathbbm{K} [\tmmathbf{x}] \}$.
  Therefore, $\mathcal{A}_{\xi_i}^{\ast} = Q_i^{\perp}$ is spanned by the
  elements $\tmmathbf{u}_{\xi_i} \star (p \star \sigma) = p \star
  (\tmmathbf{u}_{\xi_i} \star \sigma) = p \star \left( \omega_i (\tmmathbf{y})
  \tmmathbf{e}_{{\xi_i}} (\tmmathbf{y}) \right)$ for $p \in \mathbbm{K}
  [\tmmathbf{x}]$, that is, by $\omega_i (\tmmathbf{y})
  \tmmathbf{e}_{{\xi_i}} (\tmmathbf{y})$ and all its derivatives with respect
  to $\frac{d}{d y_i}$. This shows that $\mathcal{A}_{\xi_i}^{\ast} = D_i
  \tmmathbf{e}_{{\xi_i}} (\tmmathbf{y})$ where $D_i \subset \mathbbm{K}
  [\tmmathbf{y}]$ is the inverse system spanned by $\omega_i (\tmmathbf{y})$.
  It implies that
  the multiplicity ${\mu}_i = \dim \mathcal{A}^{\ast}_{\xi_i} = \dim
  \mathcal{A}_{\xi_i} $ of $\xi_i$ is the dimension ${\mu} (\omega_i)$ of
  the inverse system of $\hspace{0.25em} \omega_i (\tmmathbf{y})$. We deduce
  that $\dim \mathcal{A}_{\sigma} = \dim \mathcal{A}_{\sigma}^{\ast} = r =
  \sum_{i = 1}^{r'} {\mu} (\omega_i)$. This concludes the proof of the
  theorem.
\end{proof}

Let us give some direct consequences of this result.

\begin{proposition}
  \label{prop:iso}If $\sigma (\tmmathbf{y}) = \sum_{i = 1}^{r'}
  \hspace{0.25em} \omega_i (\tmmathbf{y}) \tmmathbf{e}_{{\xi_i}}
  (\tmmathbf{y})$ with $\omega_i (\tmmathbf{y}) \in \mathbbm{K}
  [\tmmathbf{y}] \setminus \{0\}$ and $\xi_i \in \mathbbm{K}^n$ pairwise
  distinct, then we have the following properties:
  \begin{itemize}
    \item The points \ $\mathbf{\xi}_1, \mathbf{\xi}_2, \ldots,
    \mathbf{\xi}_{r'} \in \mathbbm{K}^n$ are the common roots of the
    polynomials in $I_{\sigma} = \ker H_{\sigma} = \{ p \in \mathbbm{K}
    [{\bdx}] \mid \forall q \in \mathbbm{K} [\tmmathbf{x}], \langle
    \sigma |   p q \rangle = 0 \}$.
    
    \item The series $\omega_i (\tmmathbf{y}) \tmmathbf{e}_{\xi_i}$ is a
    generator of the inverse system of $Q_i^{\bot}$, where $Q_i$ is the
    primary component of $I_{\sigma}$ associated to $\xi_i$.
    
    
    \item The inner product $\langle \cdot , \cdot \rangle_{\sigma}$ is non-degenerate on $\mathcal{A}_{\sigma} =\mathbbm{K}
    [\tmmathbf{x}] / I_{\sigma}$.
  \end{itemize}
\end{proposition}

\begin{proof}
  Suppose that $\sigma (\tmmathbf{y}) = \sum_{i = 1}^{r'} \hspace{0.25em}
  \omega_i (\tmmathbf{y}) \tmmathbf{e}_{{\xi_i}} (\tmmathbf{y})$.
  
  To prove the first point, we construct polynomials $\delta_i (\tmmathbf{x})
  \in \mathbbm{K} [\tmmathbf{x}]$ such that $\delta_i \star \sigma
  =\tmmathbf{e}_{{\xi_i}} (\tmmathbf{y})$. We choose a polynomial $\delta_i
  (\tmmathbf{x})$ such that $\delta_i (\xi_1 + \partial_{y_1}, \ldots, \xi_n +
  \partial_{y_n}) (\omega_i) (\tmmathbf{y}) = 1$. We can take for instance a
  term of the form $c \prod_{j = 1}^n (x_j - \xi_{i, j})^{\alpha_j}$ where $c
  \in \mathbbm{K}$ and $\alpha = (\alpha_1, \ldots, \alpha_n) \in
  \mathbbm{N}^n$ is the exponent of a monomial of $\omega_i$ of highest
  degree. By Lemma \ref{lem:diffpolexp}, we have
  \[ \delta_i (x) \star \left( \omega_i (\tmmathbf{y}) \tmmathbf{e}_{{\xi_i}}
     (\tmmathbf{y}) \right) = \delta_i (\xi_i + \tmmathbf{\partial}) (\omega_i)
     (\tmmathbf{y}) \tmmathbf{e}_{{\xi_i}} (\tmmathbf{y})
     =\tmmathbf{e}_{{\xi_i}} (\tmmathbf{y}). \]
  Then $\forall p \in I_{\sigma}$, $\langle \delta_i \star
  (\tmmathbf{u}_{\xi_i} \star \sigma)  |  p \rangle = \left\langle
  \delta_i (x) \star \left( \omega_i (\tmmathbf{y}) \tmmathbf{e}_{{\xi_i}}
  (\tmmathbf{y}) \right)  |  p \right\rangle = \left\langle
  \tmmathbf{e}_{{\xi_i}} (\tmmathbf{y}) |  p \right\rangle = p
  (\xi_i) = 0$ and $\mathcal{V} (I_{\sigma}) \supset \{ \xi_1, \ldots,
  \xi_{r'} \}$. As $I_{\sigma}$ contains $\cap_{i = 1}^r
  \tmmathbf{m}_{\xi_i}^{d_{i + 1}}$ where $d_i$ is the degree of $\omega_i
  (\tmmathbf{y})$, we also have $\mathcal{V} (I_{\sigma}) \subset \{ \xi_1,
  \ldots, \xi_{r'} \}$, which proves the first point.
  
  The second point is a consequence of Theorem \ref{thm:gorenstein}, since we
  have $\omega_i (\tmmathbf{y}) \tmmathbf{e}_{\xi_i} \in D_i = Q_i^{\perp}$ so
  that ${\mu} (\omega_i) ={\mu} (\omega_i (\tmmathbf{y})
  \tmmathbf{e}_{\xi_i}) \leqslant \dim (D_i) = \dim (Q_i^{\perp})$ and
  \[ r = \sum_{i = 1}^{r'} {\mu} (\omega_i) \leqslant \sum_{i = 1}^{r'}
     \dim (Q_i^{\perp}) = \dim (I^{\perp}) = \dim \left( \mathcal{A}_{\sigma}
     \right) = r. \]
  Therefore, the inverse system spanned by $\omega_i (\tmmathbf{y})
  \tmmathbf{e}_{\xi_i}$, of dimension ${\mu} (\omega_i) = \dim
  (Q_i^{\perp}) = \dim D_i$, is equal to $D_i$.
  
  
  Finally, we prove the third point. By definition of $I_{\sigma}$, if $p \in \mathbbm{K} [\tmmathbf{x}]$ is such
  that $\forall q \in \mathbbm{K} [\tmmathbf{x}]$,
  \[
  \langle p (\tmmathbf{x}), q (\tmmathbf{x}) \rangle_{\sigma} =
     \langle p \star \sigma (\tmmathbf{y})  |   q (\tmmathbf{x})
     \rangle = 0,
  \]
  then $p \star \sigma (\tmmathbf{y}) = 0$ and $p \in I_{\sigma}$. We deduce
  that the inner product $\langle \cdummy, \cdummy \rangle_{\sigma}$ is
  non-generate on $\mathcal{A}_{\sigma} =\mathbbm{K} [\tmmathbf{x}] /
  I_{\sigma}$, which proves the last point. 
\end{proof}

If $(b_i)_{1 \leqslant i \leqslant r}$ \ and $(b'_i)_{1 \leqslant i
\leqslant r}$ are bases of $\mathcal{A}_{\sigma}$, then the matrix of
$\mathcal{H}_{\sigma}$ in the basis $(b_i)_{1 \leqslant i \leqslant \delta}$
and in the dual basis of $(b'_i)_{1 \leqslant i \leqslant r}$ is \
$[\mathcal{H}_{\sigma}] = (\langle \sigma |  b_i (\tmmathbf{x}) b_i'
(\tmmathbf{x}) \rangle)_{1 \leqslant i, j \leqslant r}$. In particular, if
$(\tmmathbf{x}^{\beta})_{\beta \in B}$ and $(\tmmathbf{x}^{\beta'})_{\beta'
\in B'}$ are bases of $\mathcal{A}_{\sigma}$, its matrix in the corresponding
bases is
\[ [\mathcal{H}_{\sigma}] = (\langle \sigma |   \tmmathbf{x}^{\beta
   + \beta'} \rangle)_{\beta \in B, \beta' \in B'} = (\sigma_{\beta +
   \beta'})_{\beta \in B, \beta' \in B'} = H_{\sigma}^{B, B'}. \]
It is a submatrix of the (infinite) matrix $[H_{\sigma}]$. Conversely, we have
the following property:

\begin{lemma}
  \label{lem:basis} Let $B = \{ b_1, \ldots, b_r \}$, $B' = \{ b_1', \ldots,
  b_r' \} \subset \mathbbm{K} [\tmmathbf{x}]$. If the matrix $H_{\sigma}^{B,
  B'} = (\langle \sigma |   b_i b_j' \rangle)_{\beta \in B, \beta' \in
  B'}$ is invertible, then B and $B'$ are linearly independent in
  $\mathcal{A}_{\sigma}$.
\end{lemma}

\begin{proof}
  Suppose that $H_{\sigma}^{B, B'}$ is invertible. If there exist $p = \sum_i
  p_i b_i$ ($p_i \in \mathbbm{K}$) such that $p \equiv 0$ in
  $\mathcal{A}_{\sigma}$. Then $p \star \sigma = 0$ and $\forall q \in R$,
  $\langle \sigma |   p q \rangle = 0$. In particular, for $j = 1
  \ldots r$ we have
  \[ \sum_{i = 1}^r \langle \sigma |   b_i b_j' \rangle p_i = 0. \]
  As $H_{\sigma}^{B, B'}$ is invertible, $p_i = 0$ for $i = 1, \ldots, r$ and
  $B$ is a family of linearly independent elements in $\mathcal{A}_{\sigma}$.
  Since we have $(H_{\sigma}^{B, B'})^t = H_{\sigma}^{B', B}$, \ we prove by a
  similar argument that $H_{\sigma}^{B, B'}$ invertible also implies that for
  $B'$ is linearly independent in $\mathcal{A}_{\sigma}$.
\end{proof}

Notice that the converse is not necessarily true, as shown by the following
example in one variable: if $\sigma = y$, then $I_{\sigma} = (x^2)$,
$\mathcal{A}_{\sigma} =\mathbbm{K} [x] / (x^2)$ and $B = B' = \{ 1 \}$ are
linearly independent in $\mathcal{A}_{\sigma}$, but $H_{\sigma}^{B, B'} =
(\langle \sigma |   1 \rangle) = (0)$ is not invertible.

This lemma implies that if $\dim \mathcal{A}_{\sigma} < + \infty$, $| B | = |
B' | = \dim \mathcal{A}_{\sigma}$ and $H_{\sigma}^{B, B'}$ is invertible, then
$(\tmmathbf{x}^{\beta})_{\beta \in B}$ and $(\tmmathbf{x}^{\beta'})_{\beta'
\in B'}$ are bases of $\mathcal{A}_{\sigma}$.


A special case of interest is when the roots are simple. We characterize
it as follows:

\begin{proposition}
  \label{prop:simpleroot}Let $\sigma (\tmmathbf{y}) \in \mathbbm{K}
  [[\tmmathbf{y}]]$. The following conditions are equivalent:
  \begin{enumeratenumeric}
    \item $\begin{array}{lll}
      \sigma (\tmmathbf{y}) & = & \sum_{i = 1}^r \hspace{0.25em} \omega_i
      \tmmathbf{e}_{{\xi_i}} (\tmmathbf{y}),
    \end{array}$with $\omega_i \in \mathbbm{K} \setminus \{0\}$ and
    $\hspace{0.25em} \xi_i \in \mathbb{\mathbbm{K}}^n  \text{pairwise
    distinct}$.
    
    \item The rank of $H_{\sigma}$ is $r$ and the multiplicity of the points
    $\mathbf{\xi}_1, \ldots, \mathbf{\xi}_r$ in $\mathcal{V} (I_{\sigma})$ is
    $1$.
    
    \item A basis of $\mathcal{A}_{\sigma}^{\ast}$ is
    $\mathbf{\tmmathbf{e}}_{\mathbf{\xi}_1}, \ldots,
    \mathbf{\tmmathbf{e}}_{\mathbf{\xi}_r}$.
  \end{enumeratenumeric}
\end{proposition}

  \begin{proof}
  $1 \Rightarrow 2.$ The dimension ${\mu} (\omega_i)$ of the inverse
  system spanned by $\omega_i \in \mathbbm{K} \setminus \{0\}$ and its
  derivatives is $1 $. By Theorem \ref{thm:gorenstein}, the rank
  $\mathcal{A}_{\sigma}$ is $r = \sum_{i = 1}^r {\mu} (\omega_i) = \sum_{i
  = 1}^r 1$ and the multiplicity of the roots $\mathbf{\xi}_1, \ldots,
  \mathbf{\xi}_r$ in $\mathcal{V} (I_{\sigma})$ is $1$.
  
  $2 \Rightarrow 3.$ As the multiplicity of the roots is $1$, by Theorem
  \ref{thm:gorenstein}, $\begin{array}{lll}
    \sigma (\tmmathbf{y}) & = & \sum_{i = 1}^r \hspace{0.25em} \omega_i\,
    \tmmathbf{e}_{{\xi_i}} (\tmmathbf{y})
  \end{array}$ with \ $\deg (\omega_i) = 0$. As $\mathcal{A}_{\sigma}^{\ast}$
  is spanned by the elements $p \star \sigma = \sum_{i = 1}^r \omega_i p (\xi_i) \,\tmmathbf{e}_{{\xi_i}}$ 
  for $p \in \mathbbm{K}[\tmmathbf{x}]$, $\mathbf{\tmmathbf{e}}_{\mathbf{\xi}_1}, \ldots,
  \mathbf{\tmmathbf{e}}_{\mathbf{\xi}_r}$ is a generating family of the vector
  space $\mathcal{A}_{\sigma}^{\ast}$. Thus
  $\mathbf{\tmmathbf{e}}_{\mathbf{\xi}_1}, \ldots,
  \mathbf{\tmmathbf{e}}_{\mathbf{\xi}_r}$ is a basis of
  $\mathcal{A}_{\sigma}^{\ast}$.
  
  $3 \Rightarrow 1.$ As $\mathbf{\tmmathbf{e}}_{\mathbf{\xi}_1}, \ldots,
  \mathbf{\tmmathbf{e}}_{\mathbf{\xi}_r}$ is a basis $\mathcal{A}_{\sigma}$,
  the points $\xi_i \in \mathbb{\mathbbm{K}}^n$ are pairwise distinct. As
  $\sigma \in \mathcal{A}_{\sigma}^{\ast}$, there exists
  $\omega_i \in \mathbbm{K}$ such that$\begin{array}{lll}
    \sigma & = & \sum_{i = 1}^r \hspace{0.25em} \omega_i
    \tmmathbf{e}_{{\xi_i}}
  \end{array}$. If one of these coefficients $\omega_i$ vanishes then $\dim
  (\mathcal{A}_{\sigma}^{\ast}) < r$, which is contradicting point 3. Thus
  $\omega_i \in \mathbbm{K} \setminus \{0\}$.
\end{proof}

In the case where all the coefficients of $\sigma$ are in $\mathbbm{R}$, we
can consider the following property of positivity:

\begin{definition}
  An element $\sigma \in \mathbbm{R} [[\tmmathbf{y}]] =\mathbbm{R}
  [\tmmathbf{x}]^{\ast}$ is semi-definite positive if $\forall p \in
  \mathbbm{R} [\tmmathbf{x}], \langle p, p \rangle_{\sigma} = \langle
  \sigma |  p^2 \rangle \geqslant 0$. It is denoted $\sigma
  \succcurlyeq 0$.
\end{definition}

The positivity of $\sigma$ induces a nice property
of its decomposition, which is an important ingredient of
polynomial optimisation. It is saying that a positive measure on
$\mathbbm{R}^n$ with an Hankel operator of finite rank $r$ is a convex
combination of $r$ distinct Dirac measures of $\mathbbm{R}^n$. See e.g.
{\cite{laurent_sums_2009}} for more details. For the sake of
completeness, we give here its simple proof (see also
{\cite{lasserre_moment_2013}}[prop. 3.14]).

\begin{proposition}
  \label{prop:positive}Let $\sigma \in \mathbbm{R} [[\tmmathbf{y}]]$ of finite
  rank. Then $\sigma \succcurlyeq 0$, if and only if,
  \[ \sigma (\tmmathbf{y}) = \sum_{i = 1}^r \omega_i \hspace{0.25em}
     \tmmathbf{e}_{{\xi_i}} (\tmmathbf{y}) \]
  with $\omega_i > 0$, $\xi_i \in \mathbbm{R}^n$.
\end{proposition}

{\color{black} \begin{proof}
  If $\sigma (\tmmathbf{y}) = \sum_{i = 1}^r \omega_i \hspace{0.25em}
  \tmmathbf{e}_{{\xi_i}}$ with $\omega_i > 0$, $\xi_i \in \mathbbm{R}^n$,
  then clearly $\forall p \in \mathbbm{R} [\tmmathbf{x}]$,
  \[ \langle \sigma |  p^2 \rangle = \sum_{i = 1}^r \omega_i
     \hspace{0.25em} p^2 (\xi_i) \geqslant 0 \]
  and $\sigma \succcurlyeq 0$.
  
  Conversely suppose that $\forall p \in \mathbbm{R} [\tmmathbf{x}]$, $\langle
  \sigma |  p^2 \rangle \geqslant 0$. Then $p \in I_{\sigma}$, if and only if,
  $\langle \sigma |  p^2 \rangle = 0$. We check that $I_{\sigma}$ is
  real radical: If $p^{2 k} + \sum_j q_j^2 \in I_{\sigma}$ for some $k \in
  \mathbbm{N}$, $p, q_j \in \mathbbm{R} [\tmmathbf{x}]$ then
  \[ \left\langle \sigma |  p^{2 k} + \sum_j q_j^2 \right\rangle =
     \langle \sigma |  p^{2 k} \rangle + \sum_j \langle \sigma |
      q_j^2 \rangle = 0 \]
  which implies that $\langle \sigma |   (p^k)^2 \rangle = 0$ ,
  $\langle \sigma |  q_j^2 \rangle = 0$ and  that $p^k, q_j \in
  I_{\sigma}$. Let $k' = \lceil \frac{k}{2} \rceil$. We have \ $\langle \sigma
  |   (p^{k'})^2 \rangle = 0$, which implies that $p^{k'} \in
  I_{\sigma}$. Iterating this reduction, we deduce that $p \in I_{\sigma}.$
  This shows that $I_{\sigma}$ is real radical and $\mathcal{V} (I_{\sigma})
  \subset \mathbbm{R}^n$. By Proposition \ref{prop:simpleroot}, we deduce that
  $\sigma = \sum_{i = 1}^r \omega_i \hspace{0.25em} \tmmathbf{e}_{{\xi_i}}$
  with $\omega_i \in \mathbbm{C} \setminus \{0\}$ and $\hspace{0.25em} \xi_i
  \in \mathbb{R}^n $. Let $\tmmathbf{u}_i \in \mathbbm{R} [\tmmathbf{x}]$ be a
  family of interpolation polynomials at $\xi_i \in \mathbbm{R}^n$:
  $\tmmathbf{u}_i (\xi_i) = 1$, $\tmmathbf{u}_i (\xi_j) = 0$ for $j \neq i$.
  Then $\langle \sigma |  \tmmathbf{u}_i^2 \rangle = \omega_i \in
  \mathbbm{R}_+$. This proves that $\sigma (\tmmathbf{y}) = \sum_{i = 1}^r
  \omega_i \hspace{0.25em} \tmmathbf{e}_{{\xi_i}} (\tmmathbf{y})$ with
  $\omega_i > 0$, $\xi_i \in \mathbbm{R}^n$.
\end{proof}}

\subsection{The decomposition of $\sigma$}\label{sec:4.2}

The sparse decomposition problem of the series $\sigma \in \mathbbm{K}
[[\tmmathbf{y}]]$ consists in computing the support $\{ \xi_1,
\ldots, \xi_r \}$ and the weights $\omega_i \hspace{0.25em} (\tmmathbf{y}) \in
\mathbbm{K} [\tmmathbf{y}]$ of so that $\sigma = \sum_{i = 1}^r \omega_i(\tmmathbf{y}) \, \tmmathbf{e}_{{\xi_i}} (\tmmathbf{y})$. In
this section, we describe how to compute it from the Hankel operator
$H_{\sigma}$.

We recall classical results on the resolution of polynomial equations by
eigenvalue and eigenvector computation, that we will use to compute the
decomposition. Hereafter, $\mathcal{A} =\mathbbm{K} [\tmmathbf{x}] / I$ is the
quotient algebra of $\mathbbm{K} [\tmmathbf{x}]$ by any ideal $I$
and $\mathcal{A}^{\ast} = \tmop{Hom}_{\mathbbm{K}} \left( \mathcal{A},
\mathbbm{K} \right)$ is the dual of $\mathcal{A}$. It is naturally identified
with the orthogonal $I^{\perp} = \{ \Lambda \in \mathbbm{K} [[\tmmathbf{y}]]
\nocomma \mid \forall p \in I, \langle \Lambda, p \rangle = 0 \}$. \
In the reconstruction problem, we will take $I = I_{\sigma}$.

By Proposition \ref{prop:iso}, $\nocomma H_{\sigma}$ induces the isomorphism
\begin{eqnarray*}
  \mathcal{H}_{\sigma} :  \mathcal{A}_{\sigma} & \rightarrow &
  \mathcal{A}^{\ast}_{\sigma}\\
  p (\tmmathbf{x}) & \mapsto & p (\tmmathbf{x}) \star \sigma
  (\tmmathbf{y} ).
\end{eqnarray*}
\begin{lemma}
  \label{eq:HankelMult}For any $g \in \mathbbm{K} [\tmmathbf{x}],$ we have
  \begin{equation}
    \mathcal{H}_{g \star \sigma} = \mathcal{M}_g^t \circ \mathcal{H}_{\sigma}
    =\mathcal{H}_{\sigma} \circ \mathcal{M}_g. \label{eq:Hankel}
  \end{equation}
\end{lemma}

\begin{proof}
  This is a direct consequence of the definitions of $\mathcal{H}_{g \star
  \sigma}, \mathcal{H}_{\sigma}, \mathcal{M}_g^t$ and
  $\mathcal{M}_g$. 
\end{proof}

From Relation (\ref{eq:Hankel}) and Proposition \ref{prop:eigen}, we have the
following property.

\begin{proposition}
  \label{prop:geneigen}If $\sigma (\tmmathbf{y}) = \sum_{i = 1}^r
  \hspace{0.25em} \omega_i (\tmmathbf{y}) \tmmathbf{e}_{\xi_i}
  (\tmmathbf{y})$with $\omega_i \in \mathbbm{K} [\tmmathbf{y}] \setminus
  \{0\}$ and $\hspace{0.25em} \xi_i \in \mathbbm{K}^n$ distinct, then
  \begin{itemize}
    \item for all $g \in \mathcal{A}$, the generalized eigenvalues of
    $(\mathcal{H}_{g \star \sigma}, \mathcal{H}_{\sigma})$ are $g
    (\mathbf{\xi}_i)$ with multiplicity ${\mu}_i ={\mu} (\omega_i)$,
    $i = 1 \ldots r$,
    
    \item the generalized eigenvectors common to all $(\mathcal{H}_{g \star
    \sigma}, \mathcal{H}_{\sigma})$ with $g \in \mathcal{A}$ are - up to a
    scalar - \ $\mathbf{\mathcal{H}_{\sigma}^{- 1}}
    (\tmmathbf{e}_{\mathbf{\xi}_1}), \ldots, \mathbf{\mathcal{H}_{\sigma}^{-
    1}} (\tmmathbf{e}_{\mathbf{\xi}_r})$.
  \end{itemize}
\end{proposition}

\begin{remark}
  If we take $g = x_i$, then the eigenvalues are the $i$-th coordinates of the
  points $\mathbf{\xi}_j$.
\end{remark}

\subsubsection{The case of simple roots}

In the case of simple roots, that is $\sigma$ is of the form $\sigma
(\tmmathbf{y}) = \sum_{i = 1}^r \hspace{0.25em} \omega_i \tmmathbf{e}_{\xi_i}
(\tmmathbf{y})$ with $\omega_i \in \mathbbm{K} \setminus \{0\}$ and
$\hspace{0.25em} \xi_i \in \mathbbm{K}^n$ distinct, computing the
decomposition reduces to a simple eigenvector computation, as we will see.

By Proposition \ref{prop:simpleroot}, \ $\{ \tmmathbf{e}_{\xi_1}, \ldots,
\tmmathbf{e}_{\xi_r} \}$ is a basis of $\mathcal{A}_{\sigma}^{\ast}$. We
denote by $\{ \mathbf{{\bdu}}_{\mathbf{\xi}_1}, \ldots,
\mathbf{{\bdu}}_{\mathbf{\xi}_r} \} $ the basis of
$\mathcal{A}_{\sigma}$, which is dual to $\{ \tmmathbf{e}_{\xi_1}, \ldots,
\tmmathbf{e}_{\xi_r} \}$, so that $\forall a (\tmmathbf{x}) \in
\mathcal{A}_{\sigma},$
\begin{equation}
  a (\tmmathbf{x}) \equiv \sum_{i = 1}^r \hspace{0.25em} \langle
  \mathbf{{\bde}}_{\mathbf{\xi}_i} (\tmmathbf{y}) |   a
  (\tmmathbf{x}) \rangle \tmmathbf{u}_{{\xi_i}} (\tmmathbf{x}) \equiv \sum_{i
  = 1}^r \hspace{0.25em} a (\mathbf{\xi}_i) \tmmathbf{u}_{{\xi_i}}
  (\tmmathbf{x}). \label{eq:proj}
\end{equation}
From this formula, we easily verify that the polynomials \
$\mathbf{\tmmathbf{u}}_{\mathbf{\xi}_1},
\mathbf{\tmmathbf{u}}_{\mathbf{\xi}_2}, \ldots,
\mathbf{\tmmathbf{u}}_{\mathbf{\xi}_r} $ are the {\tmem{interpolation
polynomials}} at the points $\mathbf{\xi}_1, \mathbf{\xi}_2, \ldots,
\mathbf{\xi}_r$, and satisfy the following relations in
$\mathcal{A}_{\sigma}$:
\begin{itemize}
  \item $\mathbf{{\bdu}}_{\mathbf{\xi}_i} (\mathbf{\xi}_j)$=$ \left\{
  \begin{array}{ll}
    1 & \tmop{if} i = j,\\
    0 & \tmop{otherwise}.
  \end{array} \right.$
  
  \item $\mathbf{{\bdu}}_{\mathbf{\xi}_i} (\mathbf{\tmmathbf{x}})^2
  \equiv \mathbf{\tmmathbf{u}}_{\mathbf{\xi}_i} (\mathbf{\tmmathbf{x}})$.
  
  \item $\sum_{i = 1}^r {\mathbf{u}}_{\mathbf{\xi}_i} (\tmmathbf{x})
  \equiv 1$.
\end{itemize}
\begin{proposition}
  \label{prop:37}Let $\begin{array}{lll}
    \sigma & = & \sum_{i = 1}^r \hspace{0.25em} \omega_i
    \tmmathbf{e}_{{\xi_i}} (\tmmathbf{y})
  \end{array}$with $\xi_i$ pairwise distinct and $\omega_i \in \mathbbm{K}
  \setminus \{ 0 \}$. The basis $\{ \mathbf{\tmmathbf{u}}_{\mathbf{\xi}_1},
  \ldots, \mathbf{\tmmathbf{u}}_{\mathbf{\xi}_r} \}$ is an orthogonal basis of
  $\mathcal{A}_{\sigma}$ for the inner product $\langle \cdot,\cdot \rangle_{\sigma}$ and satisfies $\langle
  \mathbf{\tmmathbf{u}}_{\mathbf{\xi}_i}, 1 \rangle_{\sigma} =
  \langle \sigma |   \tmmathbf{u}_{\mathbf{\xi}_i} \rangle =
  \omega_i$ for $i = 1 \ldots, r$.
\end{proposition}

\begin{proof}
  For $i, j = 1 \ldots r$, we have $\langle {\bdu}_{\mathbf{\xi}_i}, {{\bdu}}_{\mathbf{\xi}_j}
  \rangle_{\sigma} = \langle \sigma |    \mathbf{{\bdu}}_{\mathbf{\xi}_i}  {\mathbf{u}}_{\mathbf{\xi}_j} \rangle =
  \sum_{k = 1}^r \hspace{0.25em} \omega_k  \mathbf{{\bdu}}_{\mathbf{\xi}_i} (\xi_k) \mathbf{{\bdu}}_{\mathbf{\xi}_j}
  (\xi_k)$. Thus
  \begin{eqnarray*}
    \langle \mathbf{{\bdu}}_{\mathbf{\xi}_i}, \mathbf{{\bdu}}_{\mathbf{\xi}_j} \rangle_{\sigma} & = & \left\{ \begin{array}{l}
      \omega_i \tmop{if} i = j\\
      0 \tmop{otherwise}
    \end{array} \right.
  \end{eqnarray*}
  and $\{ \mathbf{\tmmathbf{u}}_{\mathbf{\xi}_1}, \ldots,
  \mathbf{\tmmathbf{u}}_{\mathbf{\xi}_r} \}$ is an orthogonal basis of
  $\mathcal{A}_{\sigma}$. Moreover, $\langle
  \mathbf{\tmmathbf{u}}_{\mathbf{\xi}_i}, 1 \rangle_{\sigma} =
  \langle \sigma |   \tmmathbf{u}_{\mathbf{\xi}_i} \rangle = \sum_{k
  = 1}^r \hspace{0.25em} \omega_k  \mathbf{{\bdu}}_{\mathbf{\xi}_i}
  (\xi_k) = \omega_i$.
\end{proof}

Proposition \ref{prop:eigen} implies the following result:

\begin{corollary}
  \label{cor:22}If $g \in \mathbbm{K} [\tmmathbf{x}]$ is separating the roots
  $\mathbf{\xi}_1, \ldots, \mathbf{\xi}_r$ (i.e. $g (\mathbf{\xi}_i) \neq g
  (\mathbf{\xi}_j) $ when $i \neq j$), then
  \begin{itemize}
    \item the operator $\mathcal{M}_g$ is diagonalizable and its eigenvalues
    are $g (\mathbf{\xi}_1), \ldots, g (\mathbf{\xi}_r)$,
    
    \item the corresponding eigenvectors of $\mathcal{M}_g$ are, up to a
    non-zero scalar, the interpolation polynomials $\mathbf{{\bdu}}_{\mathbf{\xi}_1}, \ldots, {\mathbf{u}}_{\mathbf{\xi}_r}$.
    
    \item the corresponding eigenvectors of $\mathcal{M}_g^t$ are, up to a
    non-zero scalar, the evaluations $\mathbf{{\bde}}_{\mathbf{\xi}_1},
    \ldots, \mathbf{{\bde}}_{\mathbf{\xi}_r}$.
  \end{itemize}
\end{corollary}

A simple computation shows that $H_{\sigma} (\tmmathbf{u}_{\xi_i}) = \omega_i
\tmmathbf{e}_{{\xi_i}} (\tmmathbf{y})$ for $i = 1, \ldots, r$. This leads to
the following formula for the weights of the decomposition of $\sigma$:

\begin{proposition}
  \label{prop:geneigeninterpol}If $\sigma = \sum_{i = 1}^r \hspace{0.25em}
  \omega_i \tmmathbf{e}_{{\xi_i}} (\tmmathbf{y})$ with $\xi_i$ pairwise
  distinct and $\omega_i \in \mathbbm{K} \setminus \{ 0 \}$ and $g \in
  \mathbbm{K} [\tmmathbf{x}]$ is separating the roots $\mathbf{\xi}_1, \ldots,
  \mathbf{\xi}_r$, then there are $r$ linearly independent generalized
  eigenvectors $\tmmathbf{v}_1 (\tmmathbf{x}), \ldots, \tmmathbf{v}_r
  (\tmmathbf{x})$ of $(\mathcal{H}_{g \star \sigma}, \mathcal{H}_{\sigma})$,
  which satisfy the relations:
  \begin{eqnarray*}
     \mathbf{\langle  \sigma} |  x_j
    \tmmathbf{v}_i \rangle & = & \xi_{i, j}   \mathbf{\langle
     \sigma} |  \tmmathbf{v}_i \rangle \tmop{for} j = 1,
    \ldots, n, i = 1, \ldots, r\\
    \sigma (\tmmathbf{y}) & = & \sum_{i = 1}^r
    \frac{1}{\tmmathbf{v}_i (\xi_i)}  \mathbf{\langle 
    \sigma} |  \tmmathbf{v}_i \rangle \tmmathbf{e}_{{\xi_i}}
    (\tmmathbf{y})
  \end{eqnarray*}
\end{proposition}

\begin{proof}
  By the relations (\ref{eq:Hankel}) and Corollary \ref{cor:22}, the
  eigenvectors $\mathbf{{\bdu}}_{\mathbf{\xi}_1}, \ldots,
  \mathbf{{\bdu}}_{\mathbf{\xi}_r}$ of $\mathcal{M}_g$ are the
  generalized eigenvectors of $(\mathcal{H}_{g \star \sigma},
  \mathcal{H}_{\sigma})$. By Corollary \ref{cor:22}, $\tmmathbf{v}_i$ is a
  multiple of the interpolation polynomial $\tmmathbf{u}_{\mathbf{\xi}_i}$,
  and thus of the form $\tmmathbf{v}_i (\tmmathbf{x}) =\tmmathbf{v}_i (\xi_i)
  \tmmathbf{u}_{\xi_i} (\tmmathbf{x})$ since $\tmmathbf{u}_{\xi_i} (\xi_i)$=1.
  We deduce that $\tmmathbf{u}_{\xi_i} (\tmmathbf{x}) =
  \frac{1}{\tmmathbf{v}_i (\xi_i)} \tmmathbf{v}_i (\tmmathbf{x})$. By
  Proposition \ref{prop:37}, we have
  \[ \omega_i = \mathbf{\langle \sigma } | {\bdu}_{\xi_i}
     \rangle = \frac{1}{\tmmathbf{v}_i (\xi_i)}  \mathbf{\langle
      \sigma} |  \tmmathbf{v}_i \rangle, \]
  from which, we dedue the decomposition of $\sigma = \sum_{i = 1}^r
  \frac{1}{\tmmathbf{v}_i (\xi_i)}  \mathbf{\langle 
  \sigma} |  \tmmathbf{v}_i \rangle \tmmathbf{e}_{{\xi_i}}
  (\tmmathbf{y})$. It implies that
  \[ \mathbf{\langle \sigma } | x_j  {\bdu}_{\xi_i} \rangle =
     \sum_{k = 1}^r \hspace{0.25em} \omega_k  {\xi}_{k, j}
     \mathbf{{\bdu}}_{\tmmathbf{\xi}_i} (\xi_k) = {\xi}_{i, j}
     \omega_i = {\xi}_{i, j} \langle \sigma  | {\bdu}_{\xi_i} \rangle. \]
  Multiplying by $\tmmathbf{v}_i (\xi_i)$, we obtain the first relations.
\end{proof}

This property shows that the decomposition of $\sigma$ can be deduced directly
from the generalized eigenvectors of $(\mathcal{H}_{g \star \sigma},
\mathcal{H}_{\sigma})$. In particular, it is not necessary to solve a
Vandermonde linear system to recover the weights $\omega_i$ as
in the pencil method (see Section \ref{sec:pronyunivar}). We summarize it in the
following algorithm, which computes the decomposition of $\sigma$, assuming a
basis $B$ of $\mathcal{A}_{\sigma}$ is known.
Hereafter $B^+ = \cup_{i = 1}^n x_i B \cup B$.
\\
\begin{algorithm}[H]\caption{Decomposition of polynomial-exponential series with constant weights}\label{algo:genprony1}
\Indm{\tmstrong{Input:}} the (first) coefficients
$\sigma_{\alpha}$ of a series $\sigma \in \mathbbm{K} [[\tmmathbf{y}]]$ for
$\alpha \in \tmmathbf{a} \subset \mathbbm{N}^n$ and  bases
$B = \{ b_1, \ldots, b_r \}$, $B' = \{ b_1', \ldots, b_r' \}$
of $\mathcal{A}_{\sigma}$ such that $\langle B'\cdot B^{+}\rangle\subset \langle \tmmathbf{x}^{\tmmathbf{a}}\rangle$.
\begin{itemizeminus}
   \item Construct the matrices $H_0 = (\mathbf{\langle \sigma } |
    b_i' b_j \rangle)_{1 \leqslant i, j \leqslant r}$ (resp. $H_k =
    (\mathbf{\langle \sigma } | x_k b_i' b_j \rangle)_{1 \leqslant i,
    j \leqslant r}$) of $\mathcal{H}_{\sigma}$ (resp. $\mathcal{H}_{x_k \star
    \sigma}$) in the bases $B, B'$ of $\mathcal{A}_{\sigma}$\;
    
    \item  Take a separating linear form ${\bdl} (\tmmathbf{x}) = l_1
    x_1 + \cdots + l_n x_n$ and construct $H_{{\bdl}} = \sum_{i = 1}^n
    l_i H_i = (\mathbf{\langle \sigma } | \tmmathbf{l}b_i' b_j
    \rangle)_{1 \leqslant i, j \leqslant r}$\;
    
     \item Compute the generalized eigenvectors $\tmmathbf{v}_1,
    \ldots, \tmmathbf{v}_r$ of $(H_{{\bdl}}, H_0)$ where $r = | B |$\;
    
     \item Compute $\xi_{i, j}$ such that $  \mathbf{\langle
     \sigma} |  x_j \tmmathbf{v}_i \rangle = \xi_{i, j} 
     \mathbf{\langle  \sigma} |  \tmmathbf{v}_i
    \rangle $for $j = 1, \ldots, n$, $i = 1, \ldots, r$\;
    
     \item Compute $\tmmathbf{u}_i (\tmmathbf{x}) =
    \frac{1}{\tmmathbf{v}_i (\xi_i)} \tmmathbf{v}_i (\tmmathbf{x})$ where
    $\xi_i = (\xi_{i, 1}, \ldots, \xi_{i, n})$\;
    
    \item  Compute $\langle \sigma |   \tmmathbf{u}_i \rangle =
    \omega_i$\;
\end{itemizeminus} 
{\tmstrong{Output:}} the decomposition $\sigma (\tmmathbf{y}) = \sum_{i = 1}^r
\omega_i \tmmathbf{e}_{\xi_i} (\tmmathbf{y})$.
\end{algorithm}

To apply this algorithm, one needs to compute a basis $B$ of
$\mathcal{A}_{\sigma}$ such that $\sigma$ is defined on $B \cdummy B^+$. In Section \ref{sec:5}, we will detail an
efficient method to compute such a basis $B$ and a characterization of the
sequences $(\sigma_{\alpha})_{\alpha \in A}$, which admits a decomposition of
rank $r$.

The second step of the algorithm consists in taking a linear form ${\bdl} (\tmmathbf{x}) = l_1 x_1 + \cdots + l_n x_n$, which separates the roots in
the decomposition (${\bdl} (\xi_i) \neq {\bdl}
(\xi_j)$ if $i \neq j$). A generic choice of ${\bdl}$ yields a
separating linear form.
This separating property can be verified by checking (in the third step of the algorithm) that $(H_{\bdl},H_{0})$ has $r$ distinct generalized eigenvalues. Notice that we only need to compute the matrix $H_{{\bdl}}$ of
$\mathcal{H}_{{\bdl} \star \sigma}$ in the basis $B$ of
$\mathcal{A}_{\sigma}$ and not necessarily all the matrices $H_k$.

The third step is the computation of generalized eigenvectors of a Hankel
pencil. The other steps involve the application of $\sigma$ on
polynomials in $B^+$.

We illustrate the method on a sequence $\sigma_{\alpha}$ obtained by
evaluation of a sum of exponentials on a grid.

\begin{example}
  \label{ex:40}We consider the function $h (u_1, u_2) = {\color{green}
  {\color{green} 2 + 3} \hspace{0.25em}} \cdummy {\color{red} 2^{u_1}
  2^{u_2} {\color{green} -} 3^{u_1}}$. Its associated
  generating series is $\sigma = \sum_{\alpha \in \mathbbm{N}^2} h (\alpha) 
  \frac{\tmmathbf{y}^{\alpha}}{\alpha !} $. Its (truncated) moment matrix is
  \[ H^{[1, x_1, x_2, x_1^2, x_1 x_2, x_2^2]}_{\sigma} = \left[
     \begin{array}{cccc}
       h (0, 0) & h (1, 0) & h (0, 1) & \cdots\\
       h (1, 0) & h (2, 0) & h (1, 1) & \cdots\\
       h (0, 1) & h (1, 1) & h (0, 2) & \cdots\\
       \vdots & \vdots & \vdots & \\
       \vdots & \vdots & \vdots & 
     \end{array} \right] = \left[ \begin{array}{cccccc}
       4 & 5 & 7 & 5 & 11 & 13\\
       5 & 5 & 11 & - 1 & 17 & 23\\
       7 & 11 & 13 & 17 & 23 & 25\\
       5 & - 1 & 17 & - 31 & 23 & 41\\
       11 & 17 & 23 & 23 & 41 & 47\\
       13 & 23 & 25 & 41 & 47 & 49
     \end{array} \right] \]
  We compute $B = \{1, x_1, x_2 \}$. The generalized eigenvalues of $(H_{x_1
  \star \sigma}, H_{\sigma})$ are $[{\color{red} 1, 2, 3]}$ and
  corresponding eigenvectors are represented by the columns of
  \[ {\bdu} \assign \left[ \begin{array}{ccc}
       2 & - 1 & 0\\
       - \frac{1}{2} & 0 & \frac{1}{2}\\
       - \frac{1}{2} & 1 & - \frac{1}{2}
     \end{array} \right], \]
  associated to the polynomials ${\bdu} (x) = [2 - \frac{1}{2} 
  \hspace{0.25em} x_1 - \frac{1}{2}  \hspace{0.25em} x_2, - 1 + x_2,
  \frac{1}{2}  \hspace{0.25em} x_1 - \frac{1}{2}  \hspace{0.25em} x_2]$. By
  computing the Hankel matrix
  \[ H_{\sigma}^{[1, x_1, x_2], {\bdu}} = \left[ \begin{array}{ccc}
       {\color{darkgreen} } \langle \sigma |  {\bdu}_1
       \rangle & {\color{darkgreen} } \langle \sigma | 
       {\bdu}_2 \rangle {\color{darkgreen} } & {\color{darkgreen}
       } \langle \sigma |  {\bdu}_3 \rangle\\
       {\color{darkgreen} } \langle \sigma |  x_1 {\bdu}_1
       \rangle & {\color{darkgreen} } \langle \sigma |  x_1
       {\bdu}_2 \rangle & {\color{darkgreen} } \langle \sigma |
        x_1 {\bdu}_3 \rangle\\
       {\color{green} } {\color{darkgreen} } \langle \sigma |  x_2 
       {\bdu}_1 \rangle & {\color{darkgreen} } \langle \sigma |
        x_2 {\bdu}_2 \rangle & {\color{darkgreen} }
       \langle \sigma |  x_2 {\bdu}_3 \rangle
     \end{array} \right] = \left[ \begin{array}{ccc}
       {\color{darkgreen} 2} & {\color{darkgreen} 3} & {\color{darkgreen} -
       1}\\
       {\color{green} {\color{darkgreen} 2 \times}}  {\color{red} 1} &
       {\color{green} {\color{darkgreen} 3 \times}}  {\color{red} 2} &
       {\color{green} {\color{darkgreen} - 1 \times}} {\color{red} 3}\\
       {\color{green} {\color{darkgreen} 2 \times}}  {\color{red} 1} &
       {\color{darkgreen} 3 \times}  {\color{red} 2} & {\color{green}
       {\color{darkgreen} - 1 \times}} {\color{red} 1}
     \end{array} \right] \]
  we deduce the weights {\color{darkgreen} ${\color{green} {\color{darkgreen} 2, 3, - 1}}$} and the frequencies ${\color{red} (1, 1), (2, 2), (3,
  1)}$, which correspond to the decomposition $\sigma = e^{y_1 + y_2} + 3
  e^{2 y_1 + 2 y_2} - e^{2 y_1 + y_2} $ and $h (u_1, u_2) = 2 \noplus + 3
  \cdummy 2^{u_1 + u_2} - 3^{u_1}$.
\end{example}

Let us recall other relations between the structured matrices involved in this
eigenproblem, that are useful to analyse the numerical behavior of the method.
For more details, see e.g. {\cite{mp-jcomplexity-2000}}. Such decompositions,
referred as Carath{\'e}odory-Fej{\'e}r-Pisarenko decompositions in 
{\cite{yang_vandermonde_2016}}, are induced by rank deficiency conditions or
flat extension properties (see Section \ref{sec:5}). They can be used to
recover the decomposition of the series in Pencil-like methods.

\begin{definition}
  Let $B = \{ b_1, \ldots, b_r \}$ be a family of polynomials. We
  define the $B$-Vandermonde matrix of the points $\mathbf{\xi}_1, \ldots,
  \mathbf{\xi}_r \in \mathbb{C}^n$ as
  \[ V_{B, \underset{}{{\mathbf\xi} }} = (
     \mathbf{\langle  {\bde}}_{\mathbf{\xi}_j} | b_i
       \rangle)_{1 \leq i, j \leq r} = (b_i (\mathbf{\xi}_j))_{1
     \leq i, j \leq r}. \]
\end{definition}

By remark \ref{rem:evalcoef}, if $\{ \tmmathbf{e}_{\xi_1}, \ldots,
\tmmathbf{e}_{\xi_r} \}$ is a basis of $\mathcal{A}_{\sigma}^{\ast}$ and $B$
is a basis of $\mathcal{A}_{\sigma}$, then $V_{B, {\mathbf\xi}}$ is the
matrix of coefficients of $\tmmathbf{e}_{\xi_1}, \ldots, \tmmathbf{e}_{\xi_r}$
in the dual basis of $B$ in $\mathcal{A}_{\sigma}^{\ast}$ and it is
invertible. Conversely, if $\{ \tmmathbf{e}_{\xi_1}, \ldots,
\tmmathbf{e}_{\xi_r} \}$ is a basis of $\mathcal{A}_{\sigma}^{\ast}$, we check
that $V_{B, {\mathbf\xi}}$ is invertible implies that $B = \{ b_1, \ldots,
b_r \}$ is a basis of $\mathcal{A}_{\sigma}$.

\begin{proposition}
  \label{dec:vdm:simple}Suppose that $\sigma = \sum_{k = 1}^r \hspace{0.25em}
  \omega_k \tmmathbf{e}_{{\xi_k}} (\tmmathbf{y})$ with $\mathbf{\xi}_1,
  \ldots, \mathbf{\xi}_r \in \mathbb{K}^n$ pairwise distinct and $\omega_1,
  \ldots, \omega_r \in \mathbbm{K} \setminus \{ 0 \}$. Let
  $D_{\tmmathbf{\omega}} = \tmop{diag} (\omega_1, \ldots, \omega_r)$ be the
  diagonal matrix associated to the weights $\omega_i$ and let $D_g =
  \tmop{diag} (g (\xi_1), \ldots, g (\xi_r))$ be the diagonal
  matrices associated to $g (\xi_1), \ldots, g (\xi_r)$. For any
  family $B, B'$ of $\mathbbm{K} [\tmmathbf{x}]$, we have
  \[ \begin{array}{lll}
       H^{B, B'}_{\sigma} & = & V_{B', {\tmmathbf\xi}} D_{\tmmathbf{\omega}}
       V_{B, {\tmmathbf\xi}}^t\\
       H_{g \star \sigma}^{B, B'} & = & V_{B', {\tmmathbf\xi}}
       D_{\tmmathbf{\omega}} D_g V_{B, {\tmmathbf\xi}}^t = V_{B', {\tmmathbf\xi}} D_g D_{\tmmathbf{\omega}} V_{B, {\tmmathbf\xi}}^t
     \end{array} \]
  If moreover $B$ is a basis of $\mathcal{A}_{\sigma}$, then $V_{B, {\tmmathbf\xi}}$ is invertible and
  \[ \begin{array}{lll}
       (M_g^B)^t & = & V_{B, {\tmmathbf\xi}} D_g V_{B, {\tmmathbf\xi}}^{-1}
                       \hspace{1em} \hspace{1em} \hspace{1em} \hspace{1em} \hspace{1em}
                       \hspace{1em} \hspace{1em} \hspace{1em} \hspace{1em} \hspace{1em}

     \end{array} \]
\end{proposition}

\begin{proof}
  If $\sigma = \sum_{k = 1}^r \hspace{0.25em} \omega_k \tmmathbf{e}_{{\xi_k}}
  (\tmmathbf{y})$ and $B = \{ b_1, \ldots, b_r \}, B' = \{ b_1', \ldots, b_r'
  \}$ are bases of $\mathcal{A}_{\sigma}$, then
  \[ H^{B, B'}_{\sigma} = \left[ \sum_{k = 1}^r \omega_k b'_i (\xi_k) b_j
     (\xi_k) \right]_{i, j = 1, \ldots, r} = V_{B', {\tmmathbf\xi}}
     D_{\tmmathbf{\omega}} V_{B, {\tmmathbf\xi}}^t. \]
  By a similar explicit computation, we check that $H_{g \star \sigma}^{B, B'}
  = V_{B', {\tmmathbf\xi}} D_{\tmmathbf{\omega}} D_g V_{B, {\tmmathbf\xi}}^t$. Equation (\ref{eq:Hankel}) implies that $(M_g^B)^t = H_{g \star
  \sigma}^{B, B} (H_{\sigma}^{B, B})^{- 1} = V_{B, {\tmmathbf\xi}} D_g V_{B,
  {\tmmathbf\xi}}^{- 1}$.
\end{proof}

\subsubsection{The case of multiple roots}

We consider now the general case where $\sigma$ is of the form $\sigma =
\sum_{i = 1}^{r'} \hspace{0.25em} \omega_i (\tmmathbf{y}) \tmmathbf{e}_{\xi_i}
(\tmmathbf{y})$ with $\omega_i (\tmmathbf{y}) \in \mathbbm{K} [\tmmathbf{y}]$
and $\hspace{0.25em} \xi_i \in \mathbb{K}^n$ pairwise distinct. By Theorem
\ref{thm:artindec}, we have
\[ \mathcal{A}_{\sigma} = \mathcal{A}_{\sigma, \xi_1} \oplus \cdots \oplus
   \mathcal{A}_{\sigma, \xi_{r'}} \]
where $\mathcal{A}_{\sigma, \xi_i} \simeq \mathbbm{K} [\tmmathbf{x}] / Q_i$ is
the local algebra associated to the $\mathfrak{m}_{\xi_i}$-primary component $Q_{i}$
of $I_{\sigma}$. The the decomposition (\ref{eq:dualdec}) and
Proposition \ref{prop:primaryortho} imply that $\mathcal{A}_{\sigma, \xi_i}$
is a local Artinian Gorenstein Algebra such that $\tmmathbf{u}_{\xi_i} \star
\sigma$ is a basis of $\mathcal{A}_{\sigma, \xi_i}^{\ast}$. The operators
$\mathcal{M}_{x_j}$ of multiplication by the variables $x_j$ in
$\mathcal{A}_{\sigma}$ for $j = 1, \ldots, n$ are commuting and have a block
diagonal decomposition, corresponding to the decomposition of
$\mathcal{A}_{\sigma}$.

It turns out that the operators $\mathcal{M}_{x_j}$ have common eigenvectors
$\tmmathbf{v}_i (\tmmathbf{x}) \in \mathcal{A}_{\sigma, \xi_i}$. Such an
eigenvector is an element of the \tmtextit{socle} $(0 : \mathfrak{m}_{\xi_i})
= \{ v \in \mathcal{A}_{\sigma, \xi_i} \mid (x_j - \xi_{i, j}) \tmmathbf{v}
\equiv 0, j = 1, \ldots, n \} = (Q_i : \mathfrak{m}_{\xi_i}) / Q_i$.

In the case of an Artinian Gorenstein algebra $\mathcal{A}_{\sigma, \xi_i}$,
the socle $(0 : \mathfrak{m}_{\xi_i})$ is a vector space of dimension $1$ (see
e.g. {\cite{ElkMou07}} [Sec. 7.1.5 and Sec. 9.5] for a simple proof). A basis
element can be computed as a common eigenvector of the commuting operators
$\mathcal{M}_{x_j}$. The corresponding eigenvalues are the coordinates
$\xi_{i, 1}, \ldots, \xi_{i, n}$ of the roots $\xi_i$, $i = 1, \ldots, r$.

For a separating linear form ${\bdl} (\tmmathbf{x}) = l_1 x_1 + \cdots
+ l_n x_n$ (such that ${\bdl} (\xi_i) \neq {\bdl}
(\xi_j)$ if $i \neq j$), the eigenspace of $\mathcal{M}_{{\bdl}}$ for
the eigenvalue ${\bdl} (\xi_i)$ is the local algebra
$\mathcal{A}_{\xi_i}$ associated to the root $\xi_i$. Let $B_i = \{ b_{i, 1},
\ldots, b_{i, {\mu}_i} \}$ be a basis of this eigenspace. It spans the
elements of $\mathcal{A}_{\xi_i}$, which are of the form ${\bdu}_{\xi_i} a$ for $a \in \mathcal{A}_{\sigma}$ where ${\bdu}_{\xi_i}$
is the idempotent associated to $\xi_i$ (see Theorem \ref{thm:artindec}). In
particular, the eigenspace of $\mathcal{M}_{\tmmathbf{l} (\tmmathbf{x})}$
associated to the eigenvalue ${\bdl} (\xi_i)$ contains the
idempotent ${\bdu}_{\xi_i}$, which can be recovered as follows:

\begin{lemma}
  Let $B_i = \{ b_{i, 1}, \ldots, b_{i, {\mu}_i} \}$ be a basis of
  $\mathcal{A}_{\xi_i}$ and $\mathbf{U_i} = ( \mathbf{\langle
   \sigma} |  b_{i, k} \rangle)_{k = 1, \ldots,
  {\mu}_i}$. Then $(H_{\sigma}^{B_i, B_i})^{- 1} U_i$ is the coefficient
  vector of the idempotent ${\bdu}_{\xi_i}$ in the basis $B_i$ of
  $\mathcal{A}_{\xi_i}$.
\end{lemma}

\begin{proof}
  As the idempotent $\tmmathbf{u}_{\xi_i}$ satisfies the relation ${\bdu}_{\xi_i}^2 \equiv {\bdu}_{\xi_i}$ in $\mathcal{A}_{\sigma}$ and
  $\mathcal{A}_{\xi_i} =\tmmathbf{u}_{\xi_i} \mathcal{A}_{\sigma}$, we have
  \[  \langle  {\bdu}_{\xi_i} \star \sigma
     |  b_{i, k}  \rangle =  \mathbf{\langle 
     \sigma} | {\bdu}_{\xi_i}   b_{i, k}  \rangle =
      \mathbf{\langle  \sigma} |  b_{i, k} 
     \rangle, \]
  and $\mathbf{U_i} = ( \mathbf{\langle  \sigma} |
   b_{i, k}  \rangle)_{k = 1, \ldots, {\mu}_i}$ is the
  coefficient vector of ${\bdu}_{\xi_i}{\star}{\sigma}$ in the dual
  basis of $B_i$ in $\mathcal{A}_{\xi_i}^{\ast}$. By Lemma \ref{lem:basis}, as
  $B_i$ is a basis of $\mathcal{A}_{\xi_i}$, $H_{\sigma}^{B_i, B_i}$ is
  invertible and $(H_{\sigma}^{B_i, B_i})^{- 1} U_i$ is the coefficient vector
  of ${\bdu}_{\xi_i}$ in the basis $B_i$ of $\mathcal{A}_{\xi_i}$.
\end{proof}

Using the idempotent $\bdu_{\xi_{i}}$, we have the following formula for the weights $\omega_i (\tmmathbf{y})$ in
the decomposition of $\sigma$:

\begin{proposition}
  The polynomial coefficient of ${\bde}_{\xi_i} (\tmmathbf{y})$ in the
  decomposition of $\sigma$ is
  \begin{equation}
    \omega_i (\tmmathbf{y}) = \sum_{\alpha \in \mathbbm{N}^n} \mathbf{\langle
    {\bdu}_{\xi_i} \star \sigma } | (\tmmathbf{x}-
    \xi_i)^{\alpha}  \rangle \frac{\tmmathbf{y}^{\alpha}}{\alpha !}.
    \label{eq:decomp}
  \end{equation}
\end{proposition}

\begin{proof}
  By Theorem \ref{thm:gorenstein} and relation (\ref{eq:dualdec}), we have
  \[ {\bdu}_{\xi_i} \star \sigma = \omega_i (\tmmathbf{y})  {\bde}_{\xi_i} (\tmmathbf{y}). \]
  As $\mathbf{\langle \tmmathbf{y}^{\beta}  {\bde}_{\xi_i}
  (\tmmathbf{y}) } | (\tmmathbf{x}- \xi_i)^{\alpha}  \rangle =
  \left\{ \begin{array}{ll}
    \alpha ! & \tmop{if} \alpha = \beta\\
    0 & \tmop{otherwise}
  \end{array} \right.$, we deduce the decomposition (\ref{eq:decomp}), which
  is a finite sum since $\omega_i (\tmmathbf{y}) \in \mathbbm{K}
  [\tmmathbf{y}]$.
\end{proof}

This leads to the following decomposition algorithm in the case of multiple roots.

\begin{algorithm}[H]\caption{Decomposition of polynomial-exponential series with polynomial weights}\label{algo:genprony2}
\Indm{\tmstrong{Input:}} the coefficients $\sigma_{\alpha}$ of a series $\sigma \in \mathbbm{K} [[\tmmathbf{y}]]$ for
$\alpha \in \tmmathbf{a} \subset \mathbbm{N}^n$ and  bases
$B = \{ b_1, \ldots, b_r \}$, $B' = \{ b_1', \ldots, b_r' \}$
of $\mathcal{A}_{\sigma}$ such that $\langle B'\cdot B^{+}\rangle\subset \langle \tmmathbf{x}^{\tmmathbf{a}}\rangle$.
\begin{itemizeminus}
    \item Construct the matrices $H_0 = (\mathbf{\langle \sigma } |
    b_i' b_j \rangle)_{1 \leqslant i, j \leqslant r}$ (resp. $H_k =
    (\mathbf{\langle \sigma } | x_k b_i' b_j \rangle)_{1 \leqslant i,
    j \leqslant r}$) of $\mathcal{H}_{\sigma}$ (resp. $\mathcal{H}_{x_k \star
    \sigma}$) in the bases $B, B'$ of $\mathcal{A}_{\sigma}$\;
    
    \item Compute common eigenvectors $\tmmathbf{v}_1, \ldots,
    \tmmathbf{v}_r$ of all the pencils $(H_{k}, H_0)$, $k = 1, \ldots, n$
    and $\xi_i = (\xi_{i, 1}, \ldots, \xi_{i, n})$ such that $(H_{k} -
    \xi_{i, k} H_0) \tmmathbf{v}_i = 0$;
    \item Take a separating linear form ${\bdl} (\tmmathbf{x}) = l_1
    x_1 + \cdots + l_n x_n$\;
    
    \item Compute bases $B_i, i = 1, \ldots, r'$ of the generalized
    eigenspaces of $(H_{{\bdl}}, H_0)$\;
    
    \item For each basis $B_i = \{ b_{i, 1}, \ldots, b_{i, {\mu}_i} \}$,
    compute $\mathbf{U}_i = ( \langle  \sigma |  b_{i, k}  \rangle)_{k = 1, \ldots, {\mu}_i}$ and
    ${\bdu_i} = (H_{\sigma}^{B_i, B_i})^{- 1} \mathbf{U}_i$\;
    
    \item Compute $\omega_i (\tmmathbf{y}) = \sum_{\alpha \in \mathbbm{N}^n}
    \mathbf{\langle {\bdu}_i \star \sigma } | (\tmmathbf{x}-
    \xi_i)^{\alpha}  \rangle \frac{\tmmathbf{y}^{\alpha}}{\alpha !}$;
\end{itemizeminus}
{\tmstrong{Output:}} the decomposition $\sigma (\tmmathbf{y}) = \sum_{i = 1}^r
\omega_i (\tmmathbf{y}) \tmmathbf{e}_{\xi_i} (\tmmathbf{y})$.
\end{algorithm}

Here also, we need to know a basis $B$ of $\mathcal{A}_{\sigma}$ such that
$\sigma$ is known on $B'\cdummy B^+$. The second step is the computation of
common eigenvectors of pencil of matrices. Efficient methods as in
{\cite{graillat_new_2009}} can be used to computed them. The other steps
generalize the computation of the simple root case. The solution of a Hankel
system is required to compute the coefficient vector of the
idempotent ${\bdu}_i$ in the basis $B_i$ of $\mathcal{A}_{\xi_i}$.

The relations between Vandermonde matrices and Hankel matrices (Proposition
\ref{dec:vdm:simple}) can be generalized to the case of multiple roots. Let
$\sigma = \sum_{k = 1}^{r'} \hspace{0.25em} \omega_k (\tmmathbf{y})
\tmmathbf{e}_{{\xi_k}} (\tmmathbf{y})$ with $\mathbf{\xi}_1, \ldots,
\mathbf{\xi}_{r'} \in \mathbb{K}^n$ pairwise distinct, $\omega_1
(\tmmathbf{y}), \ldots, \omega_{r'} (\tmmathbf{y}) \in \mathbbm{K}
[\tmmathbf{y}] \setminus \{ 0 \}$. To deduce a decomposition of
$H_{\sigma}^{B, B'}$ similar to the decomposition of Proposition
\ref{dec:vdm:simple} for multiple roots, we introduce the Wronskian of a set
$B = \{ b_1, \ldots, b_l \} \subset \mathbbm{K} [\tmmathbf{x}]$ and a set of
exponents $\Gamma = \{ \gamma_1, \ldots, \gamma_s \} \subset \mathbbm{N}^n$ at
a point $\xi \in \mathbbm{K}^n$:
\[ W_{B, \Gamma, \xi} = \left[ \frac{1}{\gamma_j !} \tmmathbf{\partial}^{\gamma_j} (b_i)
   (\xi) \right]_{1 \leqslant i \leqslant r, 1 \leqslant j \leqslant s}. \]
For a collection $\tmmathbf{\Gamma}= \{ \Gamma_1, \ldots, \Gamma_{r'} \}$ with
$\Gamma_1, \ldots, \Gamma_{r'} \subset \mathbbm{N}^n$ and points
$\tmmathbf{\xi}= \{ \xi_1, \ldots, \xi_{r'} \} \subset \mathbbm{K}^n$ let
\[ W_{B, \tmmathbf{\Gamma}, \tmmathbf{\xi}} = [W_{B, \Gamma_1, \xi_1}
  , \ldots, W_{B, \Gamma_{r'}, \xi_{r'}}] \]
be the matrix obtained by concatenation of the columns of $W_{B, \Gamma_k,
\xi_k}$, $k = 1, \ldots, r'$. \

We consider the monomial decomposition $\omega_k (\tmmathbf{y}) =
\sum_{\alpha \in A_k} \omega_{k, \alpha}  (\tmmathbf{x}- \xi_k)^{\alpha}$ with
$\omega_{k, \alpha} \neq 0$. We denote by $\Gamma_k$ the set of all the
exponents $\alpha \in A_k$ in this decomposition and all their divisors $\beta
= (\beta_1, \ldots, \beta_n)$ with $\beta \ll \alpha$. Let us denote
by $\gamma_1, \ldots, \gamma_{s_k}$ the elements of $\Gamma_k$.

Let $\Delta_{\omega_k}^{\Gamma_k} = [(\gamma_i + \gamma_j) ! \omega_{k,
\gamma_i + \gamma_j}]_{1 \leqslant i, j \leqslant s_k}$ with the convention
that $\omega_{k, \gamma_i + \gamma_j} = 0$ if $\gamma_i + \gamma_j \not\in
A_k$ is not a monomial exponent of $\omega_k (\tmmathbf{y})$. Let
$\Delta_{\tmmathbf{\omega}}^{\tmmathbf{\Gamma}}$ be the block diagonal matrix,
which diagonal blocks are $\Delta_{\omega_k}^{\Gamma_k}$, $k = 1, \ldots, r'$.

The following decomposition generalizes the Carath{\'e}odory-Fej{\'e}r decomposition in
the case of multiple roots (it is also implied by rank deficiency conditions, see Section \ref{sec:flat}):
\begin{proposition}
  \label{dec:vdm:multiple}Suppose that $\sigma = \sum_{k = 1}^{r'}
  \hspace{0.25em} \omega_k (\tmmathbf{y}) \tmmathbf{e}_{{\xi_k}}
  (\tmmathbf{y})$ with $\mathbf{\xi}_1, \ldots, \mathbf{\xi}_{r'} \in
  \mathbb{K}^n$ pairwise distinct, $\omega_1 (\tmmathbf{y}), \ldots,
  \omega_{r'} (\tmmathbf{y}) \in \mathbbm{K} [\tmmathbf{y}] \setminus \{ 0
  \}$. For $g \in \mathbbm{K} [\tmmathbf{x}]$, $g \circledast
  \tmmathbf{\omega}= [g (\xi_1 + \tmmathbf{\partial}) (\omega_1), \ldots,
  g (\xi_{r'} + \tmmathbf{\partial}) (\omega_{r'})]$. For any set $B, B'
  \subset \mathbbm{K} [\tmmathbf{x}]$ of size $l$, we have
  \[ \begin{array}{lll}
       H^{B, B'}_{\sigma} & = & W_{B', \tmmathbf{\Gamma}, {\tmmathbf\xi}}
       \Delta_{\tmmathbf{\omega}}^{\tmmathbf{\Gamma}} W_{B, \tmmathbf{\Gamma},
       {\tmmathbf\xi}}^t\\
       H_{g \star \sigma}^{B, B'} & = & W_{B', \tmmathbf{\Gamma}, {\tmmathbf\xi}} \Delta_{g \circledast \tmmathbf{\omega}}^{\tmmathbf{\Gamma}}
       W_{B, \tmmathbf{\Gamma}, \tmmathbf{\xi}}^t
     \end{array} \]
  If moreover $B$ is a basis of $\mathcal{A}_{\sigma}$, then $W_{B',
  \tmmathbf{\Gamma}, \tmmathbf{\xi}}$ and
  $\Delta_{\tmmathbf{\omega}}^{\tmmathbf{\Gamma}} $ are invertible and the
  matrix of multiplication by $g$ in the basis $B$ of $\mathcal{A}_{\sigma}$
  is
  \[ \begin{array}{lll}
       M_g^B & = & W_{B, \tmmathbf{\Gamma}, {\tmmathbf\xi}}^{-t}
       (\Delta_{\tmmathbf{\omega}}^{\tmmathbf{\Gamma}})^{- 1} \Delta_{g
       \circledast \tmmathbf{\omega}} W_{B, \tmmathbf{\Gamma}, {\tmmathbf\xi}}^{- t}. \hspace{1em} \hspace{1em} \hspace{1em} \hspace{1em}
     \end{array} \]
\end{proposition}

\begin{proof}
  By the relation (\ref{eq:der}), we have
  \[ H^{B, B'}_{\sigma} = \left[ \sum_{k = 1}^{r'} \omega_k (\tmmathbf{\partial}) (b_i' b_j) (\xi_k) \right]_{1 \leqslant i, j
     \leqslant l}. \]
  By expansion, we obtain
  \[ \omega_k (\tmmathbf{\partial}) (b_i' b_j) (\xi_k) =
     \sum_{\alpha \in A_k} \omega_{k, \alpha} \partial^{\alpha}_{\tmmathbf{x}}
     (b_i' b_j) (\xi_k). \]
  By Leibniz rule, we have
  \[ \partial^{\alpha}_{\tmmathbf{x}} (b_i' b_j) = \sum_{\beta \ll \alpha}
     \frac{\alpha !}{\beta ! (\alpha - \beta) !}
     \partial^{\beta}_{\tmmathbf{x}} (b_i') \partial^{\alpha -
     \beta}_{\tmmathbf{x}} (b_j) = \alpha ! \sum_{\beta \ll \alpha}
     \frac{\partial^{\beta}_{\tmmathbf{x}} (b_i')}{\beta !} 
     \frac{\partial^{\alpha - \beta}_{\tmmathbf{x}} (b_j)}{(\alpha - \beta) !}.
   \]
  We deduce that
  \begin{eqnarray*}
    \omega_k (\tmmathbf{\partial}) (b_i' b_j) (\xi_k) & = &
    \sum_{\alpha \in A_k} \omega_{k, \alpha} \tmmathbf{\partial}^{\alpha}
    (b_i' b_j) (\xi_k)\\
    & = & \sum_{\alpha \in A_k} \alpha ! \omega_{k, \alpha} \sum_{\beta \ll
    \alpha} \frac{\tmmathbf{\partial}^{\beta} (b_i')}{\beta !} (\xi_k) 
    \frac{\tmmathbf{\partial}^{\alpha - \beta} (b_j)}{(\alpha - \beta) !}
    (\xi_k)\\
    & = & W_{B', \Gamma_k, \xi_k} \Delta_{\omega_k}^{\Gamma_k}  \hspace{1em}
    W_{B, \Gamma_k, \xi_k}^t.
  \end{eqnarray*}
  By concatenation of the columns of $W_{B, \Gamma_k, \xi_k}$ and $W_{B',
  \Gamma_k, \xi_k}$, using the block diagonal matrix
  $\Delta_{\tmmathbf{\omega}}^{\tmmathbf{\Gamma}} $, we obtain the
  decomposition of $H^{B, B'}_{\sigma} = W_{B', \Gamma_k, \xi_k}
  \Delta_{\omega_k}^{\Gamma_k}  \hspace{1em} W_{B, \Gamma_k, \xi_k}^t$.
  
  By Lemma \ref{lem:diffpolexp}, we have
  \[ g \star \sigma = \sum_{k = 1}^{r'} g (\xi_k + \tmmathbf{\partial})  (\omega_k) \tmmathbf{e}_{\xi_k} = \sum_{k = 1}^{r'} (g \circledast
     \tmmathbf{\omega})_k \tmmathbf{e}_{\xi_k}. \]
  Thus, a similar computation yields the decomposition: $H^{B,  B'}_{g \star \sigma} = W_{B', \tmmathbf{\Gamma}, {\tmmathbf\xi}} \Delta_{g  \circledast \tmmathbf{\omega}}^{\tmmathbf{\Gamma}} W_{B, \tmmathbf{\Gamma}, {\tmmathbf\xi}}^t$.
  
  If $B$ is a basis of $\mathcal{A}_{\sigma}$, then by Proposition
  \ref{prop:iso}, $H^{B, B}_{\sigma}$ is invertible, which implies that
  $W_{B', \tmmathbf{\Gamma}, {\tmmathbf\xi}}$ and
  $\Delta_{\tmmathbf{\omega}}^{\tmmathbf{\Gamma}} $ are invertible. By
  Relation (\ref{eq:Hankel}), we have
  \[ M_g^B = (H^{B, B}_{\sigma})^{- 1} H^{B, B}_{g \star \sigma} = W_{B,
     \tmmathbf{\Gamma}, {\tmmathbf\xi}}^{- t}
     (\Delta_{\tmmathbf{\omega}}^{\tmmathbf{\Gamma}})^{- 1} \Delta_{g
     \circledast \tmmathbf{\omega}} W_{B, \tmmathbf{\Gamma}, {\tmmathbf\xi}}^{- t}. \]
\end{proof}

\section{Reconstruction from truncated Hankel operators}\label{sec:5}

Given the first terms $\sigma_{\alpha}, \alpha \in \tmmathbf{a}$ of a series
$\sigma = \sum_{\alpha \in \mathbbm{N}^n} \sigma_{\alpha} 
\frac{\tmmathbf{y}^{\alpha}}{\alpha !} \in \mathbbm{K} [[\tmmathbf{y}]]$,
where $\tmmathbf{a} \subset \mathbbm{N}^n$ is a finite set of exponents, the
reconstruction problem consists in finding points $\mathbf{\xi}_1, \ldots,
\mathbf{\xi}_r \in \mathbbm{K}^n$ and polynomial $\omega_i (\tmmathbf{y}) \in
\mathbbm{K} [\tmmathbf{y}]$ such that the coefficient of
${\frac{\tmmathbf{y}^{\alpha}}{\alpha!}}$, in the series
\begin{equation}
  \sum_{i = 1}^r \hspace{0.25em} \omega_i (\tmmathbf{y})
  \tmmathbf{e}_{{\xi_i}} (\tmmathbf{y}) \label{eq:dec}
\end{equation}
coincides with $\sigma_{\alpha}$, for all $\alpha \in \tmmathbf{a}$.

It is natural to ask when this extension problem has a solution. This problem
has been studied for the reconstruction of measures from moments. The first
answer is probably due to {\cite{fischer_uber_1911}} \ in the case of positive
sequences indexed by $\mathbbm{N}$ (see
{\cite{akhiezer_questions_1962}}[Theorem 5]). The extension to several
variables involves the notion of flat extension, which corresponds to rank
conditions on submatrices of the truncated Hankel matrix. See
{\cite{curto_solution_1996}} for semi-definite positive Hankel matrices and
its generalisation in {\cite{laurent_generalized_2009}}.

This result is closely connected to the factorisation of positive semidefinite
(PSD) Toeplitz matrices as the product of Vandermonde matrices with a diagonal
matrix {\cite{caratheodory_uber_1911}}, which has been recently
generalized to positive semidefinite multi-level block Toeplitz matrices in
{\cite{yang_vandermonde_2016}} and {\cite{andersson_kronecker_2015}}. 
Theorem \ref{thm:flatext} gives a condition under which a truncated series can
be extended in low rank. In view of Proposition \ref{dec:vdm:simple} and
\ref{dec:vdm:multiple}, it also provides a rank condition for a generalized
Carath{\'e}dory-Fej{\'e}r decomposition of general multivariate Hankel
matrices. Its specialization to multivariate positive semidefinite Hankel
matrices is given in Theorem \ref{thm:flatextsdp}.

We are going to use this flat extension property to determine when $\sigma$
has an extension of the form (\ref{eq:dec}) and to compute a basis $B$ of
$\mathcal{A}_{\sigma}$. The decomposition of $\sigma$ will be deduced from
eigenvector computation of submatrices of $H_{\sigma}$, using the algorithms
described in Section \ref{sec:4.2}.

In this section, we also analyze the problem of finding a (monomial) basis $B$ of
$\mathcal{A}_{\sigma}$ from the coefficients $\sigma_{\alpha}$, $\alpha \in  \tmmathbf{a}$.
We first characterize the series $\sigma$, which admit a decomposition of
the form (\ref{eq:dec}) such that $B$ is a basis of $\mathcal{A}_{\sigma}$.

\subsection{Flat extension}\label{sec:flat}

The flat extension property is defined as follows for general matrices:

\begin{definition}
  For any matrix $H$ which is a submatrix of another matrix $H'$, we say that
  $H'$ is a {\tmem{flat extension}} of $H$ if $\tmop{rank} H = \tmop{rank}
  H'$.
\end{definition}

Before applying it to Hankel matrices, we need to introduce the following
constructions. For a vector space $V \subset \mathbbm{K} [\tmmathbf{x}]$, we
denote by $V^+$ the vector space $V^+ = V \noplus + x_1 V \noplus + \cdots +
x_n V$. We say that $V$ is {\tmem{connected to}} $1$, if there exists an
increasing sequence of vector spaces $V_0 \subset V_1 \subset \cdots \subset
V_s = V$ such that $V_0 = \langle 1 \rangle$ and $V_{l + 1} \subset
V_l^+$. The {\tmem{index}} of an element $v \in V$ is the smallest $l$ such
that $v \in V_l$.

We say that a set of polynomials $B \subset \mathbbm{K} [\tmmathbf{x}]$ is
connected to $1$ if the vector space $\langle B \rangle$ spanned by $B$ is
connected to $1$. In particular, \ a monomial set $B = \{
\tmmathbf{x}^{\beta_1}, \ldots, \tmmathbf{x}^{\beta_r} \}$ is connected to $1$
if \ for all $m \in B$, either $m = 1$ or there exists $m' \in B$ and $i_0 \in
[1, \ldots, n]$ such that $m = x_{i_{0}} m'$.

The following result generalizes the sparse flat extension results of
{\cite{laurent_generalized_2009}} to distinct vector spaces connected to 1. We
give its proof for the sake of completeness, since the hypotheses are slightly
different.

\begin{theorem}
  \label{thm:flatext}Let $V, V' \subset \mathbbm{K} [\tmmathbf{x}]$ vector
  spaces connected to $1$ and let $\sigma \in \langle V \cdummy V' \rangle^{\ast}$.
  Let $U \subset V$, $U' \subset V'$ be vector spaces such
  that $1 \in U$, $U^+ \subset V, U'^+ \subset V'$. If $\tmop{rank}
  H_{\sigma}^{V, V'} = \tmop{rank} H_{\sigma}^{U, U'} = r$, then there is a
  unique extension $\tilde{\sigma} \in \mathbbm{K} [[\tmmathbf{y}]]$ such that
  $\tilde{\sigma}$ coincides with $\sigma$ on $\langle V \cdot V' \rangle$ and
  $\tmop{rank} H_{\tilde{\sigma}} = r$. In this case, $\tilde{\sigma} =
  \sum_{i = 1}^{r'} \hspace{0.25em} \omega_i (\tmmathbf{y})
  \tmmathbf{e}_{{\xi_i}} (\tmmathbf{y})$ with $\omega_i (\tmmathbf{y}) \in 
  \mathbbm{K} [\tmmathbf{y}]$, $\xi_i \in \mathbbm{K}^n$, $r = \sum_{i =
  1}^{r'} {\mu} (\omega_i)$ and $I_{\tilde{\sigma}} = (\ker
  H_{\sigma}^{U^+, U'})$.
\end{theorem}

\begin{proof}
  Let $K = \ker H_{\sigma}^{V, V'}$. The condition $\tmop{rank} H_{\sigma}^{U,
  U'} = \tmop{rank} H_{\sigma}^{V, V'}$ implies that
  \begin{eqnarray}
    \kappa \in K & \Leftrightarrow & \forall v' \in V', \langle \sigma \mid \kappa v' \rangle = 0  \label{eq:rk15}\\
    & \Leftrightarrow & \forall u' \in U', \langle \sigma \mid \kappa u' \rangle = 0.  \label{eq:rk16}
  \end{eqnarray}
  In particular,
  \begin{eqnarray}
    \kappa \in K \tmop{and}\, x_i \kappa \in V & \tmop{implies} & x_i \kappa \in
    K,  \label{eq:coherent}
  \end{eqnarray}
  since $U^{'^+} \subset V'$, $\forall u' \in U', \langle \sigma \mid \kappa
  x_i u' \rangle = \langle \sigma \mid x_i \kappa b' \rangle = 0$, which
  implies by the relation (\ref{eq:rk16}) that $x_i \kappa \in K$.
  
  Suppose first that $1 \in K$. Then as $V$ is connected to $1$, using
  Relation (\ref{eq:coherent}) we prove by induction on the index $l$ that
  every element of $V$ of index $> 0$, which is a sum of terms of $V$ of the
  form $x_i \kappa$ with $\kappa \in K$, is in $K$. Then $\sigma = 0$ and the
  result is obviously true for $\tilde{\sigma} = 0$.
  
  Assume now that $1 \not\in K$. The condition $\tmop{rank}
  H_{\sigma}^{U, U'} = \tmop{rank} H_{\sigma}^{V, V'}$ implies that $V = K +
  U$ and we can assume that $U \cap K = \{ 0 \}$ and $\dim (U) = \dim (U') =
  \tmop{rank} H_{\sigma}^{U, U'} = \tmop{rank} H_{\sigma}^{V, V'} = r$ with $1
  \in U$. In this case, $H_{\sigma}^{U, U'}$ is invertible.
  
  Let $M_i \assign (H_{\sigma}^{U, U'})^{- 1} H_{x_i \star \sigma}^{U, U'}$.
  It is a linear operator of $U$. As $H_{x_{i \star} \sigma}^{U, U'} =
  H_{\sigma}^{U, U'} \circ M_i$, we have $\forall u \in U$, $u' \in U'$
  \[ \langle \sigma \mid x_i u u' \rangle = \langle \sigma \mid M_i (u) u'
     \rangle \]
  As $\tmop{rank} H_{\sigma}^{V, V'} = \tmop{rank} H_{\sigma}^{U, U'} = r$
  and $U^+ \subset V, U'^+ \subset V'$, we also have $\forall j = 1, \ldots,
  n$, $\forall u \in U, \forall u' \in U'$
  \[ \langle \sigma \mid x_i x_j u u' \rangle = \langle \sigma \mid x_i u x_j
     u' \rangle = \langle \sigma \mid M_i (u) x_j u' \rangle = \langle \sigma
     \mid M_j \circ M_i (u) u' \rangle. \]
  We deduce that $\langle \sigma \mid M_j \circ M_i (u) u' \rangle = \langle
  \sigma \mid M_i \circ M_j (u) u' \rangle$ and the operators $M_i$, $M_j$
  commute: $M_j \circ M_i = M_i \circ M_j$.
  
  Let us define the operator
  \begin{eqnarray*}
    \phi : \mathbbm{K} [\tmmathbf{x}] & \rightarrow & U\\
    p & \mapsto & p (M) (1)
  \end{eqnarray*}
  and the linear form
  \begin{eqnarray*}
    \tilde{\sigma} : \mathbbm{K} [\tmmathbf{x}] & \rightarrow & \mathbbm{K}\\
    p & \mapsto & \langle \sigma \mid p (M) (1) \rangle
  \end{eqnarray*}
  We are going to show that $\tilde{\sigma}$ extends $\sigma$ and that
  $I_{\tilde{\sigma}} = (\ker H_{\tilde{\sigma}}^{U, U'^+})$. As the operators
  $M_i$ commute, the operator $p (M)$ obtained by substituting the variable
  $x_i$ by $M_i$ in the polynomial $p \in \mathbbm{K} [\tmmathbf{x}]$ is
  well-defined and the kernel $J$ of $\phi$ is an ideal of \ $\mathbbm{K}
  [\tmmathbf{x}]$.
 
  We first prove that $\tilde{\sigma}$ coincides with $\sigma$ on $\langle V
  \cdot V' \rangle$. We prove by induction on the index that $\forall v \in
  V$, $\forall u' \in U'$, $\langle \sigma \mid v u' \rangle = \langle \sigma
  \mid \phi (v) u' \rangle$. If $v$ is of index $0$, then $v = 1$ (up to a
  scalar) and $\phi (1) = 1$ so that the property is true.
  
  Let us assume that the property is true for the elements of $V$ of index $l
  - 1 \geqslant 0$ and let $v \in V$ of index $l$: there exists $v_i \in V$ of
  index $l - 1$ such that $v = \sum_i x_i v_i$. By induction hypothesis and
  the relations (\ref{eq:rk15}) and (\ref{eq:rk16}), we have $\forall u' \in
  U'$,
  \begin{eqnarray*}
    \langle \sigma \mid v u' \rangle & = &  \sum_i \langle \sigma \mid v_i x_i
    u' \rangle = \sum_i \langle \sigma \mid \phi (v_i) x_i u' \rangle = \sum_i
    \langle \sigma \mid M_i \circ \phi (v_i) u' \rangle\\
    & = & \left\langle \sigma \mid \left( \sum_i \phi (x_i v_i)  \right) u'
    \right\rangle = \langle \sigma \mid \phi (v) u' \rangle.
  \end{eqnarray*}
  Using relations (\ref{eq:rk15}) and (\ref{eq:rk16}), we also have
  \begin{equation}
    \forall v \in V, \forall v' \in V', \langle \sigma \mid
    v v' \rangle = \langle \sigma \mid \phi (v) v' \rangle \label{eq:17}.
  \end{equation}
  In a similar way, we prove that
  \begin{equation}
    \forall u \in U, \forall v' \in V', \langle \sigma \mid u v'
    \rangle = \langle \sigma \mid v' (M)  (u) \rangle. \label{eq:18}
  \end{equation}
  The property is true for $v' = 1$. Let us assume that it is true for the
  elements of $V'$ of index $l - 1 > 0$ and let $v' \in V'$ be an element of
  index $l$. There exist \ $v'_i \in V'$ of index $l - 1$ such
  that $v' = \sum_i x_i v'_i$. By induction hypothesis and the relations
  (\ref{eq:rk15}) and (\ref{eq:rk16}), we have $\forall v \in V$,
  \begin{eqnarray*}
    \langle \sigma \mid u v' \rangle & = &  \sum_i \langle \sigma \mid v'_i
    x_i u \rangle = \sum_i \langle \sigma \mid v_i' M_i (u) \rangle = \sum_i
    \langle \sigma \mid v_i' (M) M_i (u)  \rangle\\
    & = & \left\langle \sigma \mid \left( \sum_i M_i \circ v_i'  (M)  \right)
    (u) \right\rangle = \langle \sigma \mid v' (M)  (u) \rangle.
  \end{eqnarray*}
  By the relations (\ref{eq:17}) and (\ref{eq:18}), we have $\forall v \in V$,
  $\forall v' \in V'$,
  \begin{eqnarray*}
    \langle \sigma \mid v v' \rangle & = & \langle \sigma \mid v' v (M) (1)
    \rangle = \langle \sigma \mid v' (M) \circ v (M) (1) \rangle =
    \langle \sigma \mid \phi (v v') \rangle = \langle \tilde{\sigma} \mid v v'
    \rangle.
  \end{eqnarray*}
  This shows that $\tilde{\sigma}$ coincides with $\sigma$ on $\langle V \cdot
  V' \rangle$.
  
  We deduce from relation (\ref{eq:17}) that $\forall u \in U$, $\forall u'
  \in U'$,  $\langle \sigma \mid (u - \phi (u)) u' \rangle = 0$ and $\phi (u)
  = u$ since $H_{\sigma}^{U, U'}$ is invertible. Therefore $\phi$ is the
  projection of $\mathbbm{K} [\tmmathbf{x}]$ on $U$ along its kernel $J$ and
  we have the exact sequence
  \[ 0 \rightarrow J \rightarrow \mathbbm{K} [\tmmathbf{x}] \xrightarrow{\phi}
     U \rightarrow 0. \]
  Let $I_{\sigma} = \ker H_{\tilde{\sigma}}$ and $\mathcal{A}_{\sigma}
  =\mathbbm{K} [\tmmathbf{x}] / I_{\sigma}$. As $J \subset I_{\sigma}$, we
  have $\dim_{\mathbbm{K}} \mathcal{A}_{\sigma} \leqslant \dim_{\mathbbm{K}}
  \mathbbm{K} [\tmmathbf{x}] / J = \dim U = r$ and $U$ is generating
  $\mathcal{A}_{\sigma}$. Since $\tilde{\sigma}$ coincides with $\sigma$ on
  $\langle U \cdot U' \rangle$ and $H_{\sigma}^{U, U'}$ is invertible, \ by
  Lemma \ref{lem:basis} a basis \ of the vector space $U \subset \mathbbm{K}
  [\tmmathbf{x}]$ is linearly independent in $\mathcal{A}_{\sigma}$. This
  shows that $\dim_{\mathbbm{K}} \mathcal{A}_{\sigma} = r$ and that $J =
  I_{\sigma}$.
  
  Since $U$ contains $1$ and $\phi$ is the projection of $\mathbbm{K}
  [\tmmathbf{x}]$ on $U$ along $I_{\sigma} = J$, we check that $I_{\sigma}$ is
  generated by the element $x_i u - \phi (x_i u)$ for $u \in U$, $i = 1,
  \ldots, n$, that is by the elements of $\ker H_{\sigma}^{U^+, U'} \subset
  K$.
  
  If there is another $\tilde{\sigma}' \in \mathbbm{K} [[\tmmathbf{y}]]$
  which coincides with $\sigma$ on $\langle V \cdot V' \rangle$ and such that
  $\tmop{rank} H_{\tilde{\sigma}'} = r$, then $\ker H_{\sigma}^{U^+, U'}
  \subset \ker H_{\tilde{\sigma}'} = I_{\tilde{\sigma}'}$ and $J = (\ker
  H_{\sigma}^{U^+, U'}) \subset I_{\tilde{\sigma}'}$. Therefore, we have
  $\forall p \in \mathbbm{K} [\tmmathbf{x}]$, $\langle \tilde{\sigma}' \mid p
  \rangle = \langle \tilde{\sigma}' \mid \phi (p) \rangle = \langle \sigma
  \mid \phi (p) \rangle = \langle \tilde{\sigma} \mid p \rangle$, so that
  $\tilde{\sigma}' = \tilde{\sigma}$, which conclude the proof of the theorem.
\end{proof}

\begin{remark}
  Let $B$ be a monomial set. If $H_{\sigma}^{B^+, B'^+}$ is a flat extension
  of $H_{\sigma}^{B, B'}$ and $H_{\sigma}^{B, B'}$ is invertible, then a basis
  of the kernel of $H_{\sigma}^{B^+, B'}$ is given by the columns of $\left(
  \begin{array}{c}
    - P\\
    I
  \end{array} \right)$, where $H_{\sigma}^{B, B' } P = H_{\sigma}^{\partial B,
  B'}$. The columns of this matrix represent polynomials of the form
  \[ \tmmathbf{x}^{\alpha} - \sum_{\beta \in B} p_{\alpha, \beta}
     \hspace{0.25em} \tmmathbf{x}^{\beta} \]
  for $\alpha \in \partial B$. These polynomials are the border relations
  which project monomials of $\partial B$ on the vector space spanned by $B$,
  modulo $\ker H_{\sigma}^{B^+}$. Using Theorem \ref{thm:flatext} and the
  characterization of border bases in terms of \ commutation relations
  {\cite{m-99-nf}}, {\cite{MoTr:2005:issac}}, we prove that they form a border
  basis of the ideal generated by $\ker \hspace{0.25em} H_{\sigma}^{B^+}$, if and only if,
  $\tmop{rank} \hspace{0.25em} H_{\sigma}^B = \tmop{rank} \hspace{0.25em}
  H_{\sigma}^{B^+} = |B|$, or in other words, if and only if, $H_{\sigma}^{B^+}$ has a
  flat extension. As shown in {\cite{BCMT:2009:laa}} (see also  {\cite{bernardi_general_2013}}), the flat extension
  condition is equivalent to the commutation property of formal multiplication
  operators.
\end{remark}

Let us consider the case where $\mathbbm{K}=\mathbbm{R}$, $V = V'$. We say that
$H_{\sigma}^{V, V}$ is semi-definite positive or simply that
$\sigma$ is semi-definite positive on V, \ if $\forall p \in V$, $\langle
H_{\sigma}^{V, V} (p) \mid p \rangle = \langle \sigma \mid p^2 \rangle
\geqslant 0$.

We have the following flat extension version:
\begin{theorem}
  \label{thm:flatextsdp}Let $V \subset \mathbbm{R} [\tmmathbf{x}]$ be a vector
  space connected to $1$ and $\sigma \in \langle V \cdummy V
  \rangle^{\ast}$. Let $U \subset V$ such that $1 \in U$, $U^+ \subset V$,
  $\tmop{rank} H_{\sigma}^{V, V} = \tmop{rank} H_{\sigma}^{U, U}$ and
  $H_{\sigma}^{V, V} \succcurlyeq 0$. Then there is a unique extension
  $\tilde{\sigma} \in \mathbbm{R} [[\tmmathbf{y}]]$ which coincides with
  $\sigma$ on $\langle V \cdot V' \rangle$ and such that $r = \tmop{rank}
  H_{\tilde{\sigma}}$, which is of the form
  \[ \tilde{\sigma} = \sum_{i = 1}^r \omega_i \tmmathbf{e}_{\xi_i}
     (\tmmathbf{y}) \]
  with \ $\omega_i > 0$ and $\xi_i \in \mathbbm{R}^n$.
\end{theorem}

\begin{proof}
  By Theorem \ref{thm:flatext}, there is a unique extension $\tilde{\sigma}
  \in \mathbbm{R} [[\tmmathbf{y}]]$ which coincides with $\sigma$ on $\langle
  V \cdot V' \rangle$ and such that $r = \tmop{rank} H_{\tilde{\sigma}}$. As
  $H_{\sigma}^{V, V} \succcurlyeq 0$, for any $p \in \mathbbm{R}
  [\tmmathbf{x}]$ $\langle \tilde{\sigma} \mid p^2 \rangle = \langle \sigma
  \mid \phi (p)^2 \rangle \geqslant 0$ and $\tilde{\sigma}$ is semi-definite
  positive (on $\mathbbm{R} [\tmmathbf{x}]$). The real decomposition of
  $\tilde{\sigma} $ with positive weights is then a consequence of Proposition
  \ref{prop:positive}.
\end{proof}

\subsection{Orthogonal bases of $\mathcal{A}_{\sigma}$}

An important step in the decomposition method consists in computing a basis
$B$ of $\mathcal{A}_{\sigma}$. In this section, we describe how to compute a
monomial basis $B = \{ \tmmathbf{x}^{\beta} \}$ and two other bases
$\tmmathbf{p}= (p_{\beta})$ and $\tmmathbf{q}= (q_{\beta})$, which are
pairwise orthogonal for the inner product $\langle \cdummy, \cdummy
\rangle_{\sigma}$:
\[ \langle p_{\beta}, q_{\beta'} \rangle_{\sigma} = \left\{ \begin{array}{ll}
     1 & \tmop{if} \beta = \beta'\\
     0 & \tmop{otherwise}.
   \end{array} \right. \]
Such pairwise orthogonal bases of $\mathcal{A}_{\sigma}$ exist, since
$\mathcal{A}_{\sigma}$ is an Artinian Gorenstein algebra and $\langle \cdummy,
\cdummy \rangle_{\sigma}$ is non-degenerate (Proposition \ref{prop:iso}).

To compute these pairwise orthogonal bases, we will use a projection process,
similar to Gram-Schmidt orthogonalization process. The main difference is that
we compute pairs $p_{\beta} , q_{\beta}$ of orthogonal polynomials. As
the inner product $\langle \cdummy, \cdummy \rangle_{\sigma}$ may be
isotropic, the two polynomials $p_{\beta} , q_{\beta}$ may not be
equal, up to a scalar.

The method proceeds inductively starting from $\tmmathbf{b}= []$, extending
the monomials basis $\tmmathbf{b}$ with new monomials $\tmmathbf{x}^{\alpha}$,
projecting them onto the space spanned by $\tmmathbf{b}$:
\[ p_{{\alpha}} =\tmmathbf{x}^{\alpha} - \sum_{\beta \in \tmmathbf{b}} \langle
   \tmmathbf{x}^{\alpha}, q_{\beta} \rangle_{\sigma} p_{\beta} \]
and computing $q_{\alpha}$, if it exists, such that $\langle
p_{\alpha}, q_{\alpha} \rangle_{\sigma} = 1$ and $\langle
\tmmathbf{x}^{\beta}, q_{\alpha} \rangle_{\sigma} = 0$ for $\beta \in \tmmathbf{b}$. Here
is a more detailled description of the algorithm:

\begin{algorithm}[H]\caption{Orthogonal bases}\label{algo:orthobasis}
{\tmstrong{Input:}} the coefficients $\sigma_{\alpha}$ of a series $\sigma \in \mathbbm{K}
[[\tmmathbf{y}]]$ for $\alpha \in \tmmathbf{a} \subset \mathbbm{N}^n$.
\begin{itemizeminus}
  \item[] Let $\tmmathbf{b} \assign []$; $\tmmathbf{b}' \assign []$;
  $\tmmathbf{d}= []$; $\tmmathbf{n} \assign [\tmmathbf{0}]$; $\tmmathbf{s}
  \assign \tmmathbf{a}$; $\tmmathbf{s}' \assign \tmmathbf{a}$; $l \assign 0$;
  
  \item[] \While{$\tmmathbf{n} \neq \emptyset$} {
   \Indp \Indp$l \assign l + 1$\;
    \For{{\bf each} $\alpha \in \tmmathbf{n}$}{
    \begin{enumeratealpha}
      \item compute $p_{{\alpha}} =\tmmathbf{x}^{\alpha} - \sum_{\beta \in B}
      \langle \tmmathbf{x}^{\alpha}, q_{\beta} \rangle_{\sigma} p_{\beta}$;
      
      \item find the first $\alpha' \in \tmmathbf{s}'$ such that $\tmmathbf{x}^{\alpha'} p_{\alpha} \in \langle \tmmathbf{a} \rangle$ and
      $\langle \tmmathbf{x}^{\alpha'}, p_{\alpha} \rangle_{\sigma} \neq 0$;
      
      \item \uIf{ such an $\alpha'$ exists}{}

        \ \  let $q_{\alpha} \assign \frac{1}{\langle \tmmathbf{x}^{\alpha'},
      p_{\alpha} \rangle_{\sigma}}  \left( \tmmathbf{x}^{\alpha'} -
      \sum_{\beta \in B} \langle \tmmathbf{x}^{\alpha'}, p_{\beta}
      \rangle_{\sigma} q_{\beta} \right)$\;
      
      \ \ add $\alpha$ to $\tmmathbf{b}$; remove $\alpha$ from $\tmmathbf{s}$\;
      
      \ \ add $\alpha'$ to $\tmmathbf{b}'$; remove $\alpha'$ from
      $\tmmathbf{s}'$\;
      
      \uElse{}
     \ \  add $\alpha$ to $\tmmathbf{d}$\;
    \end{enumeratealpha}
    }
    $\tmmathbf{n} \assign \tmop{next} (\tmmathbf{b}, \tmmathbf{d},  \tmmathbf{s}) $\;
}
\end{itemizeminus}
{\tmstrong{Output:}}
\begin{itemizeminus}
  \item monomial sets $\tmmathbf{b}= [\beta_1, \ldots, \beta_r] \subset
  \tmmathbf{a}$, $\tmmathbf{b}' = [\beta_1', \ldots, \beta_r'] \subset
  \tmmathbf{a}$.
  
  \item pairwise orthogonal bases $\tmmathbf{p}= (p_{\beta_i})$,
  $\tmmathbf{q}= (q_{\beta_i})$ for $\langle \cdummy, \cdummy
  \rangle_{\sigma}$.
  
  \item the relations $p_{\alpha} : =\tmmathbf{x}^{\alpha} - \sum_{i =
  1}^{r} \langle \tmmathbf{x}^{\alpha}, q_{\beta_i} \rangle_{\sigma}
  p_{\beta_i}$ for $\alpha \in \tmmathbf{d}$.
\end{itemizeminus}
\end{algorithm}

The algorithm manipulates the ordered lists $\tmmathbf{b}, \tmmathbf{d},
\tmmathbf{s}, \tmmathbf{s}'$ of exponents, identified with monomials. The
monomials are ordered according to a total order denoted $\prec$. The index
$l$ is the loop index.

The algorithm uses the function $\tmop{next} (\tmmathbf{b}, \tmmathbf{d}, \tmmathbf{s})$,
which computes the set of monomials $\alpha$ in
$\partial \tmmathbf{b} \cap \tmmathbf{s}$, which are not in $\tmmathbf{d}$ and
such that $\alpha +\tmmathbf{b}' \subset \tmmathbf{a}$.

We verify that at each loop of the algorithm, the lists
$\tmmathbf{b}$ and $\tmmathbf{s}$ (resp. $\tmmathbf{b}'$ and $\tmmathbf{s}'$)
are disjoint and $\tmmathbf{b} \cup \tmmathbf{s}=\tmmathbf{a}$ (resp.
$\tmmathbf{b}' \cup \tmmathbf{s}' =\tmmathbf{a}$).

We also verify by induction that at each loop, $\langle
\tmmathbf{x}^{\tmmathbf{b}} \rangle = \langle p_{\beta} \mid \beta \in
\tmmathbf{b} \rangle$ and $\langle \tmmathbf{x}^{\tmmathbf{b}'} \rangle =
\langle q_{\beta} \mid \beta \in \tmmathbf{b} \rangle$.

The following properties are also satisfied at the end of the algorithm:

\begin{theorem}
  \label{thm:flatextalgo}Let $\tmmathbf{b}= [\beta_1, \ldots, \beta_r]$,
  $\tmmathbf{b}' = [\beta_1', \ldots, \beta_r'],$ $\tmmathbf{p}= [p_{\beta_1},
  \ldots, p_{\beta_r}]$, $\tmmathbf{q}= [q_{\beta_1}, \ldots, q_{\beta_r}]$ be
  the output of Algorithm \ref{algo:orthobasis}. Let $V = \langle
  \tmmathbf{x}^{\tmmathbf{b}^+} \rangle$. If there exists a vector space $V'$ connected to $1$ such that
  $\tmmathbf{x}^{(\tmmathbf{b}')^+} \subset V'$ and $V \cdummy V' = \langle
  \tmmathbf{x}^{\tmmathbf{a}} \rangle$. Then $\sigma$ coincides on $\langle
  \tmmathbf{x}^{\tmmathbf{a}} \rangle$ with the unique series $\tilde{\sigma}
  \in \mathbbm{K} [[\tmmathbf{y}]]$ such that $\tmop{rank} H_{\tilde{\sigma}}
  = r$, and we have the following properties:
  \begin{itemizedot}
    \item $(\tmmathbf{p}, \tmmathbf{q})$ are pairwise orthogonal bases of
    $\mathcal{A}_{\tilde{\sigma}}$ for the inner product $\langle \cdummy,
    \cdummy \rangle_{\tilde{\sigma}}$.
    
    \item The family $\left\{ p_{\alpha} =\tmmathbf{x}^{\alpha} - \sum_{i =
    1}^{r} \langle \tmmathbf{x}^{\alpha}, q_{\beta_i} \rangle_{\sigma}
    p_{\beta_i}, \alpha \in \tmmathbf{d} \right\}$ is a border basis
    of the ideal $I_{\tilde{\sigma}}$, with respect to
    $\tmmathbf{x}^{\tmmathbf{b}}$.
    
    \item The matrix of multiplication by $x_k$ in the basis $\tmmathbf{p}$
    (resp. {\tmem{{\tmstrong{q}}}}) of $\mathcal{A}_{\tilde{\sigma}}$ is $M_k
    \assign (\langle \sigma | x_k p_{\beta_j} q_{\beta_i} \rangle)_{1
    \leqslant i, j \leqslant r} $ (resp. $M_k^t$).
  \end{itemizedot}
\end{theorem}

\begin{proof}
  By construction, $V = \langle \tmmathbf{x}^{\tmmathbf{b}^+} \rangle$ is
  connected to $1$ and $\tmmathbf{x}^{\tmmathbf{b}}$ contains 1, otherwise
  $\sigma = 0$. As $V'$ contains $\tmmathbf{x}^{\tmmathbf{b}'}$ and $V \cdummy
  V' = \langle \tmmathbf{x}^{\tmmathbf{a}} \rangle$, we have $\forall \alpha \in \partial
  \tmmathbf{b}, \tmmathbf{x}^{\alpha} \cdot
  \tmmathbf{x}^{\tmmathbf{b}'} \subset \tmmathbf{x}^{\tmmathbf{a}}$. Thus when
  the algorithm stops, we have $\tmmathbf{n}= \emptyset$ and $\partial
  \tmmathbf{b} =\tmmathbf{d}$. By construction, for $\alpha \in
  \tmmathbf{d}$ the polynomials $p_{{\alpha}} =\tmmathbf{x}^{\alpha} -
  \sum_{\beta \in \tmmathbf{b}} \langle \tmmathbf{x}^{\alpha}, q_{\beta}
  \rangle_{\sigma} p_{\beta}$ are orthogonal to $\langle q_{\beta} \mid \beta
  \in \tmmathbf{b} \rangle = \langle \tmmathbf{x}^{\tmmathbf{b}'} \rangle$. As
  $\alpha \in \tmmathbf{d}$, for each $v' \in V'$, we have moreover $\langle  p_{\alpha}, v' \rangle_{\sigma} = 0$.
  
  A basis of $V$ is formed by the polynomials $p_{\alpha}$ for $\alpha \in
  \tmmathbf{b}^+$ since $\langle p_{\beta} \mid \beta \in \tmmathbf{b} \rangle
  = \langle \tmmathbf{x}^{\tmmathbf{b}} \rangle$ and $p_{{\alpha}}
  =\tmmathbf{x}^{\alpha} \noplus + b_{\alpha}$ with $b_{\alpha} \in \langle
  \tmmathbf{x}^{\tmmathbf{b}} \rangle$ for $\alpha \in \tmmathbf{d}= \partial
  \tmmathbf{b}$. The matrix of $H_{\sigma}^{V, V'}$ in this basis of $V$
  and in a basis of $V',$ which first elements are $q_{\beta_1},
  \ldots, q_{\beta_r}$, is of the form 
  \[ H_{\sigma}^{V, V'} = \left(\begin{array}{cc}
       \mathbbm{I}_r & 0\\
       \ast & 0
     \end{array}\right) \]
  where $\mathbbm{I}_r$ is the identity matrix of size $r$. The kernel of
  $H_{\sigma}^{V, V'}$ is generated by the polynomials $p_{\alpha}$ for
  $\alpha \in \tmmathbf{d}$.
  
  By Theorem \ref{thm:flatext}, $\sigma$ coincides on $V \cdummy V' =
  \langle \tmmathbf{x}^{\tmmathbf{a}} \rangle$ with a series $\tilde{\sigma}$
  such that $\tmmathbf{x}^{\tmmathbf{b}}$ is a basis of
  $\mathcal{A}_{\bar{\sigma}} =\mathbbm{K} [\tmmathbf{x}] /
  I_{\tilde{\sigma}}$ and $I_{\tilde{\sigma}} = (\ker H_{\tilde{\sigma}}^{V,
  V'}) = (p_{\alpha})_{\alpha \in \tmmathbf{d}}$.
  
  As $p_{\alpha} =\tmmathbf{x}^{\alpha} + b_{\alpha}$ with $\alpha \in
  \partial \tmmathbf{b}$ and $b_{\alpha} \in \langle
  \tmmathbf{x}^{\tmmathbf{b}} \rangle$, $(p_{\alpha})_{\alpha \in \partial
  \tmmathbf{b}}$ is a border basis with respect to
  $\tmmathbf{x}^{\tmmathbf{b}}$ for the ideal $I_{\tilde{\sigma}}$, since
  $\tmmathbf{x}^{\tmmathbf{b}}$ is a basis of of $\mathcal{A}_{\bar{\sigma}}$.
  
  This shows that $\tmop{rank} H_{\tilde{\sigma}} = \dim
  \mathcal{A}_{\tilde{\sigma}} = | \tmmathbf{b} | = r$. By
  construction, $(\tmmathbf{p}, \tmmathbf{q})$ are pairwise orthogonal for the
  inner product $\langle \cdummy, \cdummy \rangle_{\sigma}$,
  which coincides with $\langle \cdummy, \cdummy \rangle_{\tilde{\sigma}}$ on
  $\langle \tmmathbf{x}^{\tmmathbf{a}} \rangle$. Thus they are pairwise
  orthogonal bases of $\mathcal{A}_{\tilde{\sigma}}$ for the inner product
  $\langle \cdummy, \cdummy \rangle_{\tilde{\sigma}}$.
  
  As we have $x_k p_{\beta_j} \equiv \sum_{i = 1}^{r} \langle x_k
  p_{\beta_j} , q_{\beta_i} \rangle_{\sigma} 
  p_{\beta_i}$, the matrix of multiplication by $x_k$ in the basis
  $\tmmathbf{p}$ of $\mathcal{A}_{\tilde{\sigma}}$ is $M_k \assign (\langle
  x_k p_{\beta_j} , q_{\beta_i} \rangle_{\sigma} )_{1
  \leqslant i, j \leqslant r} = (\langle \sigma | x_k p_{\beta_j} q_{\beta_i}
  \rangle)_{1 \leqslant i, j \leqslant r} $. Exchanging the role of
  {\tmstrong{p}} and {\tmstrong{q}}, we obtain $M_k^t$ for the matrix of
  multiplication by $x_k$ in the basis {\tmem{{\tmstrong{q}}}}.
\end{proof}

\begin{remark}
  If the polynomials $p_{\alpha}, q_{\alpha}$ are at most of degree $d$, then only the
  coefficients of $\sigma$ of degree $\leqslant 2 d + 1$ are
  involved in this computation. In this case, the border basis and the
  decomposition of the series $\sigma$ as a sum of exponential polynomials can
  be computed from these first coefficients.
\end{remark}

\begin{remark}
  When the monomials in $\tmmathbf{s}$ are chosen according to a monomial
  ordering $\prec$, the polynomials $p_{\alpha} =\tmmathbf{x}^{\alpha} +
  b_{\alpha}$, $\alpha \in \tmmathbf{d}$ are constructed in such a way that
  their leading term is $\tmmathbf{x}^{\alpha}$. They form a Gr{\"o}bner basis
  of the ideal $I_{\tilde{\sigma}}$. To construct a minimal Gr{\"o}bner basis
  of $I_{\tilde{\sigma}}$ for the monomial ordering $\prec$, it suffices to
  keep the elements $p_{\alpha}$ with $\alpha \in \tmmathbf{d}$ minimal for
  the division.
\end{remark}

\begin{remark}
  The computation can be simplified, when $\langle \cdummy, \cdummy
  \rangle_{\sigma}$ is semi-definite, that is, when for all $p \in \langle
  \tmmathbf{x}^a \rangle$ such that $p^2 \in \langle
  \tmmathbf{x}^{\tmmathbf{a}} \rangle$, we have $\langle p, p \rangle_{\sigma}
  = 0$ implies that $\forall \alpha \in \tmmathbf{a}$ with
  $\tmmathbf{x}^{\alpha} p \in \langle \tmmathbf{x}^{\tmmathbf{a}} \rangle$,
  $\langle p, \tmmathbf{x}^{\alpha} \rangle_{\sigma} = 0$. In this case, the
  algorithm constructs a family of orthogonal polynomials $\tmmathbf{p}=
  [p_{\beta_1}, \ldots, p_{\beta_r}]$ and $\tmmathbf{q}= [q_{\beta_1}, \ldots,
  q_{\beta_r}]$ with $q_{\beta_i} = \frac{1}{\langle p_{\beta_i}, p_{\beta i}
  \rangle_{\sigma}} p_{\beta_i}$ and we have $\tmmathbf{b}=\tmmathbf{b}'$.
  Indeed, in the while loop for each $\alpha \in \tmmathbf{n}$, either
  $\langle p_{\alpha}, p_{\alpha} \rangle_{\sigma} = 0$, which implies that
  $\forall \alpha' \in \tmmathbf{t}$ with $\tmmathbf{x}^{\alpha'} p_{\alpha}
  \in \langle \tmmathbf{x}^{\tmmathbf{a}} \rangle$,  $\langle \tmmathbf{x}^{\alpha'}, 
  p_{\alpha} \rangle_{\sigma} = 0$, so that $\alpha \in \tmmathbf{d}$,
  or $\langle p_{\alpha}, p_{\alpha} \rangle_{\sigma} = \langle
  \tmmathbf{x}^{\alpha}, p_{\alpha} \rangle_{\sigma} \neq 0$ and the first
  $\alpha' \in \tmmathbf{t}$ such that $\langle \tmmathbf{x}^{\alpha'},
  p_{\alpha} \rangle_{\sigma}$ is $\alpha' = \alpha \in \tmmathbf{b}$.
  
  If $\mathbbm{K}=\mathbbm{R}$ and $\sigma$ is semi-definite positive, then
  the polynomials $\frac{1}{\sqrt{\langle p_{\beta_i}, p_{\beta i}
  \rangle_{\sigma}}} p_{\beta_i}$ are classical orthogonal polynomials for
  $\langle \cdummy, \cdummy \rangle_{\sigma}$.
\end{remark}

We can now describe the decomposition algorithm of polynomial-exponential series,
obtained by combining the algorithm for computing  bases of $\mathcal{A}_{\sigma}$
and the algorithm for computing the frequency points and the weights:

\begin{algorithm}[H]\caption{Polynomial-Exponential decomposition}\label{algo:decomposition}
{\tmstrong{Input:}} the coefficients $\sigma_{\alpha}$ of a series $\sigma \in \mathbbm{K}[[\tmmathbf{y}]]$ for $\alpha \in \tmmathbf{a} \subset \mathbbm{N}^n$.
\begin{itemizeminus}
 \item Apply Algorithm \ref{algo:orthobasis} to compute bases $B=\bdx^{\bdb}$, $B'=\bdx^{\bdb'}$ of $\mathcal{A}_{\sigma}$\;
 \item \uIf{$\exists V'\supset B'$ s.t. $\langle V'\cdot B^{+}\rangle= \langle \tmmathbf{x}^{\tmmathbf{a}}\rangle$}{ 
\ \ \ \ \ Apply Algorithm \ref{algo:genprony2} (or Algorithm \ref{algo:genprony1} if the weights are constant)\;
}
\end{itemizeminus}
 {\tmstrong{Output:}} the polynomial-exponential series $\sum_{i = 1}^r
 \omega_i (\tmmathbf{y}) \tmmathbf{e}_{\xi_i} (\tmmathbf{y})$ with $\omega_i (\tmmathbf{y}) \in 
 \mathbbm{K} [\tmmathbf{y}]$, $\xi_i \in \mathbbm{K}^n$ with the same Taylor
 coefficients $\sigma_{\alpha}$ as $\sigma$ for $\alpha \in \tmmathbf{a} \subset \mathbbm{N}^n$.
\end{algorithm}
 
\subsection{Examples}

\begin{example}
  Let $n = 1$ and $\sigma (y) = \frac{y^d}{d!} \in \mathbbm{K}
  [[y]]$ with $0 < d$ and $a \neq 0 \in \mathbbm{K}$.
  
  In the first step of the algorithm, we take $p_1 = 1$ and compute the first
  $i$ such that $\langle x^i, p_1 \rangle_{\sigma}$ is not zero. This
  yields $\tmmathbf{b}= [1]$, $\tmmathbf{b}' = [x^d]$ and $q_1 = x^d$.
  
  In a second step, we have $p_x = x - \langle x, q_1 \rangle_{\sigma} p_1 =
  x$. The first $i$ such that $\langle x^i, p_x \rangle_{\sigma}$ is not
  zero yields $\tmmathbf{b}= [1, x]$, $\tmmathbf{b}' = [x^d, x^{d - 1}]$ and
  $q_x = x^{d - 1} - \langle x^{d - 1}, p_1 \rangle_{\sigma} q_1 = x^{d - 1}$.
  
  We repeat this computation until $\tmmathbf{b}= [1, \ldots, x^d]$,
  $\tmmathbf{b}' = [x^d, x^{d - 1}, \ldots, 1]$ with $p_{x^i} = x^i $,
  $q_{x^i} = x^{d - i}$ for $i = 0, \ldots, d$.
  
  In the following step, we have $p_{x^{d + 1}} = x^{d + 1} - \langle x^{d +
  1}, q_1 \rangle_{\sigma} p_1 - \cdots - \langle x^{d + 1}, q_{x^d}
  \rangle_{\sigma} p_{x^d} = x^{d + 1}$. The algorithm stops and outputs
  $\tmmathbf{b}= [1, \ldots, x^d]$, $\tmmathbf{b}' = [x^d, x^{d - 1}, \ldots,
  1]$, $p_{x^{d + 1}} = x^{d + 1}$.
\end{example}

\begin{example}
  \label{ex:62}We consider the function $h (u_1, u_2) = {\color{green}
  {\color{darkgreen} 2 + 3} \hspace{0.25em}} \cdummy {\color{red} 2^{u_1}
  2^{u_2} {\color{green} {\color{darkgreen} -}} 3^{u_1}}$. Its associated
  generating series is $\sigma = \sum_{\alpha \in \mathbbm{N}^2} h (\alpha) 
  \frac{\tmmathbf{y}^{\alpha}}{\alpha !} $. Its (truncated) moment matrix is
  \[ H^{[1, x_1, x_2, x_1^2, x_1 x_2, x_2^2]}_{\sigma} = \left[
     \begin{array}{cccc}
       h (0, 0) & h (1, 0) & h (0, 1) & \cdots\\
       h (1, 0) & h (2, 0) & h (1, 1) & \cdots\\
       h (0, 1) & h (1, 1) & h (0, 2) & \cdots\\
       \vdots & \vdots & \vdots & \\
       \vdots & \vdots & \vdots & 
     \end{array} \right] = \left[ \begin{array}{cccccc}
       4 & 5 & 7 & 5 & 11 & 13\\
       5 & 5 & 11 & - 1 & 17 & 23\\
       7 & 11 & 13 & 17 & 23 & 25\\
       5 & - 1 & 17 & - 31 & 23 & 41\\
       11 & 17 & 23 & 23 & 41 & 47\\
       13 & 23 & 25 & 41 & 47 & 49
     \end{array} \right]. \]
  At the first step, we have $\tmmathbf{b}= [1]$, $\tmmathbf{p}= [1]$,
  $\tmmathbf{q}= \left[ \frac{1}{4} \right]$. At the second step, we compute
  $\tmmathbf{b}= [1, x_1, x_2]$, $\tmmathbf{p}= [1, x_1 - \frac{5}{4}, x_2 +
  \frac{9}{5} x_1 - 4] = [p_1, p_{x_1}, p_{x_2}]$ and $\tmmathbf{q}= \left[
  \frac{1}{4} p_1, - \frac{4}{5} p_{x_1}, \frac{5}{24} p_{x_2} \right]$. At
  the third step, $\tmmathbf{d}= [x_1^2, x_1 x_2, x_2^2]$ and the algorithm
  stops. We obtain the following generators of $\ker H_{\sigma}$:
  \begin{eqnarray*}
    p_{x_1^2} & = & x_1^2 + x_2 - 4 x_1 + 2\\
    \hspace{0.25em} p_{x_1 x_2} & = & x_1 x_2 - 2 x_2 - x_1 + 2\\
    p_{x_2^2} & = & x_2^2 - 3 x_2 + 2
  \end{eqnarray*}
  We have modulo $\ker H_{\sigma}$:
  \begin{eqnarray*}
    x_1  \hspace{0.25em} p_1 & \equiv & \frac{5}{4} p_1 + p_{x_1}\\
    x_1  \hspace{0.25em} p_{x_1} & \equiv & - \frac{5}{16} p_1 + \frac{91}{20}
    p_{x_1} - p_{x_2}\\
    x_1 p_{x_2} & \equiv & \sum_i \frac{\langle x_1 p_3, p_i
    \rangle_{\sigma}}{\langle p_i, p_i \rangle_{\sigma}} b_i = \frac{96}{25}
    p_{x_1} + \frac{1}{5} p_{x_2}
  \end{eqnarray*}
  The matrix of multiplication by $x_1$ in the basis $\tmmathbf{p}$ is
  \[ M_1 = \left[ \begin{array}{ccc}
       \frac{5}{4} & - \frac{5}{16} & 0\\
       1 & \frac{91}{20} & \frac{96}{25}\\
       0 & - 1 & \frac{1}{5}
     \end{array} \right] \]
  Its eigenvalues are ${\color{red} [ 1, 2, 3]}$ and the
  corresponding matrix of eigenvectors is
  \[ U \assign \left[ \begin{array}{ccc}
       \frac{1}{2} & \frac{3}{4} & - \frac{1}{4}\\
       \frac{2}{5} & - \frac{9}{5} & \frac{7}{5}\\
       - \frac{1}{2} & 1 & - \frac{1}{2}
     \end{array} \right], \]
  that is, the polynomials $U (x) = [2 - \frac{1}{2}  \hspace{0.25em} x_1 -
  \frac{1}{2}  \hspace{0.25em} x_2, - 1 + x_2, \frac{1}{2}  \hspace{0.25em}
  x_1 - \frac{1}{2}  \hspace{0.25em} x_2]$. By computing the Hankel matrix
  \[ H_{\sigma}^{U, [1, x_1, x_2]} = \left[ \begin{array}{ccc}
       {\color{darkgreen} 2} & {\color{darkgreen} 3} & {\color{darkgreen} -
       1}\\
       {\color{green} {\color{darkgreen} 2 \times}}  {\color{red} 1} &
       {\color{green} {\color{darkgreen} 3 \times}}  {\color{red} 2} &
       {\color{green} {\color{darkgreen} - 1 \times}} {\color{red} 3}\\
       {\color{green} {\color{darkgreen} 2 \times}}  {\color{red} 1} &
       {\color{darkgreen} 3 \times}  {\color{red} 2} & {\color{green}
       {\color{darkgreen} - 1 \times}} {\color{red} 1}
     \end{array} \right] \]
  we deduce the weights {\color{darkgreen} ${\color{green} {\color{darkgreen} 2, 3, - 1}}$} and the frequencies ${\color{red} (1, 1), (2, 2), (3,
  1)}$, which corresponds to the decomposition $\sigma = e^{y_1 + y_2} + 3
  e^{2 y_1 + 2 y_2} - e^{2 y_1 + y_2} $ associated to $h (u_1, u_2) = 2
  \noplus + 3 \cdummy 2^{u_1 + u_2} - 3^{u_1}$.
\end{example}

\section{Sparse decomposition from generating series}\label{sec:dec}

To exploit the previous results in the context of functional analysis or
signal processing, we need to transform functions into series or sequences in
$\mathbbm{K}^{\mathbbm{N}^n}$. Here is the general context that we consider,
which extends the approach of {\cite{peter_generalized_2013}} to
multi-index sequences. We assume that $\mathbbm{K}$ is \ algebraically close.
\begin{itemize}
  \item Let $\mathcal{F}$ be a functional space (in which ``the
  functions, distributions or signals live'').
  
  \item Let $S_1, \ldots, S_n : \mathcal{F} \rightarrow \mathcal{F}$ be
  linear operators of $\mathcal{F}$, which are commuting: $S_i \circ S_j =
  S_j \circ S_i$.
  
  \item Let $\Delta : h \in \mathcal{F} \mapsto \Delta [h] \in \mathbbm{K}$
  be a linear functional on $\mathcal{F}$.
\end{itemize}
We associate to an element $h \in \mathcal{F}$, its generating series:

\begin{definition}
  For $h \in \mathcal{F}$, the generating series associated to $h$ is
  \begin{eqnarray}
    \sigma_h (\tmmathbf{y}) & = &  \hspace{0.25em} \hspace{0.25em}
    \sum_{\alpha \in \mathbbm{N}^n} \frac{1}{\alpha !} \Delta [S^{\alpha} (h)]
    \tmmathbf{y}^{\alpha} 
  \end{eqnarray}
  where $S^{\alpha} = S_1^{\alpha_1} \circ \cdots \circ S_n^{\alpha_n}$ for
  $\alpha = (\alpha_1, \ldots, \alpha_n) \in \mathbbm{N}^n$.
\end{definition}

\begin{definition}
  We say that the {\tmem{regularity condition}} is satisfied if the map $h \in
  \mathcal{F} \mapsto \sigma_h (\tmmathbf{y}) \in \mathbbm{K}
  [[\tmmathbf{y}]]$ is injective.
\end{definition}

We are interested in the decomposition of a function $h \in \mathcal{F}$ in
terms of (generalized) eigenfunctions of the operators $S_i$. An
eigenfunction of the operators $S_i$ is a function $E \in \mathcal{F}$ such
that $S_j (E) = \xi_j E$ for $j = 1, \ldots, n$ with $\xi = (\xi_1
\nocomma, \ldots, \xi_n) \in \mathbbm{K}^n$. Generalized eigenfunctions of the
operators $S_i$ are functions $E_1, \ldots, E_{{\mu}} \in \mathcal{F}$
such that $S_j (E_k) = \xi_j E_k \noplus + \sum_{k' < k} m_{j, k'} E_{k'}$ for
$k = 1, \ldots, {\mu}$ and $\xi_1, \ldots, \xi_n \in \mathbbm{K}$.

The following proposition shows that if a function is a linear combination of
generalized eigenfunctions, then its generating series is a sum of
polynomial-exponential series.

\begin{theorem}
  Let $S_1, \ldots, S_n$ be commuting operators of $\mathcal{F}$. Let $E_{1,
  1}, \ldots, E_{1, {\mu}_1}, \ldots,, E_{r,
  1}, \ldots, E_{r, {\mu}_r} \in \mathcal{F}$ be generalized
  eigenfunctions of $S_1, \ldots, S_n$ such that for \ $i = 1, \ldots, r$, $j
  = 1, \ldots, n$, $k = 1, \ldots, {\mu}_i$,
  \[ S_j (E_{i, k}) = \xi_{i, j} E_{i, k} \noplus + \sum_{k' < k} m^{i,
     j}_{k', k} E_{i, k'} \]
  with $\xi_i = (\xi_{i, 1}, \ldots,
  \xi_{i, n}) \in \mathbbm{K}^n$ pairwise distinct. If $h = \sum_{i = 1}^r
  \hspace{0.25em} \sum_{k = 1}^{^{{\mu}_i}} h_{i, k} E_{i, k}$, then the generating series $\sigma_h$ has a unique decomposition as:
  \begin{eqnarray*}
    \sigma_h (\tmmathbf{y}) & = &  \hspace{0.25em} \sum_{i = 1}^r
    \hspace{0.25em} \omega_i (\tmmathbf{y}) \hspace{0.25em}
    \tmmathbf{e}_{{\xi_i}} (\tmmathbf{y})
  \end{eqnarray*}
  where $\omega_i (\tmmathbf{y}) \in \mathbbm{K} [\tmmathbf{y}]$. If the
  regularity condition is satisfied, the decomposition uniquely determines the
  coefficients $h_{i, k}$ of the decomposition of $h \in \mathcal{F}$.
\end{theorem}

\begin{proof}
  By Lemma \ref{lem:indep}, in a decomposition of series as a
  polynomial-exponential function $\sum_{i = 1}^r \omega_i  (\tmmathbf{y}) \, \tmmathbf{e}_{{\xi_i}} (\tmmathbf{y})$, the
  polynomials $\omega_i (\tmmathbf{y}) \in \mathbbm{K} [\tmmathbf{y}]$ and the
  support $\{ \xi_1, \ldots, \xi_r \}$ are unique. Let $N_{i, j} = S_j -
  \xi_{i, j} \tmop{Id}$ be the linear operator of $\tmmathbf{E}_i = \langle
  E_{i, 1}, \ldots, E_{i, {\mu}_i} \rangle$ such that $N_{i, j} (E_{j, k})
  = \sum_{k' < k} m^i_{j, k'} E_{j, k'}$. By construction, $N_{i, j}$ is
  nilpotent of order $\leqslant {\mu}_i + 1$ and its matrix in the basis
  $\{ E_{i, 1}, \ldots, E_{i, {\mu}_i} \}$ of $\tmmathbf{E}_i$ is $(m^{i,
  j}_{k, k'})_{k, k'}$ (with $m^{i, j}_{k, k'} = 0$ if $k \geqslant k'$). As
  the operators $S_j$ restricted to $\tmmathbf{E}_i$ are $\xi_{i, j}
  \tmop{Id} + N_{i, j}$ and commute, we deduce that the operators $N_{i, j}$
  commute for $j = 1, \ldots, n$. By the binomial expansion of $S^{\alpha} =
  S_1^{\alpha_1} \cdots S_n^{\alpha_n}$ for $\alpha = (\alpha_1, \ldots,
  \alpha_n) \in \mathbbm{N}^n$ and the commutation of the matrices \ $N_{i,
  j}$, we have
  \[ S^{\alpha} (E_{i, k}) = \sum_{\beta \ll \alpha, \beta_j \leqslant
     {\mu}_j} \binom{\alpha}{\beta} \xi_i^{\alpha - \beta} N_i^{\beta}
     (E_{i, k}), \]
  where $\binom{\alpha}{\beta} = \binom{\alpha_1}{\beta_1} \cdots
  \binom{\alpha_n}{\beta_n}$ and $N_i^{\beta} = N_{i, 1}^{\beta_1} \cdots
  N_{i, n}^{\beta_n}$. As $N_{i, j}$ is nilpotent of order ${\mu}_i$+1,
  this sum involves \ at most $({\mu}_i + 1)^n$ terms such that $\beta_j
  \leqslant {\mu}_j$, $j = 1, \ldots, n$.
  
  The generating series of $E_{i, k}$ is then
  \begin{eqnarray*}
    \sigma_{E_{i, k}} (\tmmathbf{y}) & = & \hspace{0.25em} \sum_{\alpha \in
    \mathbbm{N}^n} \sum_{\beta \ll \alpha, \beta_j \leqslant {\mu}_j}
    \Delta [N_i^{\beta} (E_{i, k})] \binom{\alpha}{\beta} \xi_i^{\alpha -
    \beta}  \frac{\tmmathbf{y}^{\alpha}}{\alpha !}\\
    & = & \sum_{\beta_i \leqslant {\mu}_i} \Delta [N_i^{\beta} (E_{i,
    k})] \frac{\tmmathbf{y}^{\beta}}{\beta !} \sum_{\alpha' \in \mathbbm{N}^n}
    \xi_i^{\alpha'}  \frac{\tmmathbf{y}^{\alpha'}}{\alpha' !}\\
    & = & \sum_{\beta_i \leqslant {\mu}_i} \Delta [N_i^{\beta} (E_{i,  k})]  \frac{\tmmathbf{y}^{\beta}}{\beta !} \tmmathbf{e}_{{\xi_i}}
    (\tmmathbf{y}) = \omega_{i, k} (\tmmathbf{y}) \tmmathbf{e}_{{\xi_i}}
    (\tmmathbf{y}),
  \end{eqnarray*}
  using the relation $\frac{1}{\alpha !} \binom{\alpha}{\beta} =
  \frac{1}{\beta !}  \frac{1}{(\alpha - \beta) !} $, exchanging the summation
  order and setting $\alpha' = \alpha - \beta$. We deduce that if $h = \sum_{i
  = 1}^r \hspace{0.25em} \sum_{k = 1}^{^{{\mu}_i}} h_{i, k} E_{i, k}$, then $\sigma_h (\tmmathbf{y}) = \sum_{i = 1}^r \omega_i
  (\tmmathbf{y}) \tmmathbf{e}_{{\xi_i}} (\tmmathbf{y})$ with $\omega_i
  (\tmmathbf{y}) = \sum_k h_{i, k} \omega_{i, k} (\tmmathbf{y}) \in
  \mathbbm{K} [\tmmathbf{y}]$. If the regularity condition is satisfied, the
  map $h \in \mathcal{F} \mapsto \sigma_h (\tmmathbf{y}) \in \mathbbm{K}
  [[\tmmathbf{y}]]$ is injective and the polynomials $\omega_{i, k}
  (\tmmathbf{y})$ $k = 1, \ldots, {\mu}_i$ are linearly independent.
  Therefore, the coefficients $h_{i, k}$, $k = 1, \ldots, {\mu}_i$ are
  uniquely determined by the polynomial $\omega_i (\tmmathbf{y}) = \sum_k
  h_{i, k} \omega_{i, k} (\tmmathbf{y})$.
\end{proof}

\begin{definition}
  We say that the {\tmem{completness}} condition is satisfied if for any
  polynomial-exponential series $\omega (\tmmathbf{y}) \tmmathbf{e}_{{\xi}}
  (\tmmathbf{y})$ with $\omega (\tmmathbf{y}) \in \mathbbm{K} [\tmmathbf{y}]$
  and $\xi \in \mathbbm{K}^n$, there exists a linear combination $h \in
  \mathcal{F}$ of generalized eigenfunctions of the operators $S_i$, such
  that its generating function is $\omega (\tmmathbf{y}) \tmmathbf{e}_{{\xi}}
  (\tmmathbf{y})$.
\end{definition}

Under the completness condition and the regularity condition, any function $h
\in \mathcal{F}$ with a generating series of finite rank can be decomposed
into a linear combination of eigenfunctions. We analyse several cases, for
which this framework applies.

\subsection{Reconstruction from moments}

Let $\mathcal{E}= C^{\infty} (\mathbbm{R}^n)$, $\mathcal{S}$ be the set of
functions in $\mathcal{E}$ with fast decrease at infinity ($\forall f\in \mathcal{S}, \forall p\in \mathbbm{C}[\bdx]$,  $|p f|$ is bounded on $\mathbbm{R}^{n}$),
$\mathcal{O}_M$ be the set of functions in $\mathcal{E}$ with slow increase at infinity
($\forall f\in \mathcal{O}_{M}, |f(\bdx)|< C(1+|\bdx|)^{N}$ for some $C\in \mathbbm{R}, N\in \mathbbm{N}$),
$\mathcal{E}'$ be the set of distributions with compact support (dual to
$\mathcal{E}$), $\mathcal{S}'$ be the set of tempered distribution (dual to
$\mathcal{S}$) and $\mathcal{O}'_C$ be the space of distributions with rapid
decrease at infinity (see {\cite{schwartz_theorie_1966}}).

In this problem, we consider the following space and operators:
\begin{itemize}
  \item $\mathcal{F}=\mathcal{O}'_C$ is the space of distributions with rapid
  decrease at infinity;
  
  \item $S_i : h (\tmmathbf{x}) \in \mathcal{O}'_C \mapsto x_i h
  (\tmmathbf{x}) \in \mathcal{O}'_C$ is the multiplication by $x_j$;
  
  \item $\Delta : h (\tmmathbf{x}) \in \mathcal{O}'_C \mapsto \int h
  (\tmmathbf{x}) d\tmmathbf{x} \in \mathbbm{C}$.
\end{itemize}
For any $h \in  \mathcal{O}'_C$, for any $\alpha \in \mathbbm{N}^n$,
\[ \Delta [S^{\alpha} (h)] = \int \tmmathbf{x}^{\alpha} h (\tmmathbf{x}) 
   d\tmmathbf{x} \]
is the $\alpha^{\tmop{th}}$ moment of $h$. For $h \in \mathcal{O}'_C$ and 
$\sigma_h = \sum_{\alpha \in \mathbbm{N}^n} \int h (\tmmathbf{x}) 
\tmmathbf{x}^{\alpha} \frac{\tmmathbf{y}^{\alpha}}{\alpha !} d\tmmathbf{x}$
its generating series, we verify that $\forall p \in \mathbbm{C}[\tmmathbf{x}]$, $\langle \sigma_h \mid p \rangle = \int h (\tmmathbf{x}) p 
(\tmmathbf{x}) d\tmmathbf{x}$ (i.e. the distribution $h$ applied to $p$).
We check that 
\begin{itemize}
 \item the operators $S_j$ are well defined and commute
\item a Dirac measure $\delta_{\xi}$ with $\xi = (\xi_1, \ldots,
\xi_n) \in \mathbbm{C}^n$ is an eigenfunction of $S_j$: $S_j (\delta_{\xi}) =
\xi_j \delta_{\xi}$. Similarly for $\alpha = (\alpha_1, \ldots, \alpha_n) \in
\mathbbm{N}^n$ and
  \item the Dirac derivation $\delta_{\xi}^{(\alpha)}$ ($\forall f \in C^{\infty} (\Omega)$, $\langle \delta_{\xi}^{(\alpha)}, f
\rangle = (- 1)^{| \alpha |} \partial_{x_1}^{\alpha_1} \cdots
\partial_{x_n}^{\alpha_n} (f) (\xi)$) satisfies 
\[ S_i (\delta_{\xi}^{(\alpha)}) = x_i \delta_{\xi}^{(\alpha)} = \xi_i
   \delta_{\xi}^{(\alpha)} \noplus + \delta^{(\alpha - e_i)}_{\xi} \]
with the convention that $\delta^{(\alpha - e_i)}_{\xi} = 0$ if $\alpha_i =
0$. It is a generalized eigenfunction of the operators $S_j$.
\end{itemize}
By the relation (\ref{eq:integ}), the
generating series of $\delta_{\xi}^{(\alpha)}$ is
\[ \sigma_{\delta_{\xi}^{(\alpha)}} = \langle \delta_{\xi}^{(\alpha)},
   {\bde}^{\tmmathbf{x} \cdummy \tmmathbf{y}} \rangle
   =\tmmathbf{y}^{\alpha} \tmmathbf{e}_{{\xi}} (\tmmathbf{y}). \]
This shows that the completeness condition is satisfied.

To check the regularity condition, we use the {\tmem{Fourier transform}}
$\mathscr{F}: f \in \mathcal{O}_M \mapsto \int f (\tmmathbf{x})
e^{-\tmmathbf{i}\tmmathbf{x} \cdummy \tmmathbf{z}} d\tmmathbf{x} \in
\mathcal{O}'_C$. It is a bijection between $\mathcal{O}_M$ and
$\mathcal{O}_C'$ (see {\cite{schwartz_theorie_1966}}[Th{\'e}or{\`e}me XV]).
Its inverse is $\mathscr{F}^{- 1} : f \in \mathcal{O}'_C \mapsto (2 \pi)^n
\int f (\tmmathbf{x}) e^{\tmmathbf{i}\tmmathbf{x} \cdummy \tmmathbf{z}}
d\tmmathbf{x} \in \mathcal{O}_M$. \ Let $\iota : f (\tmmathbf{y}) \in
\mathbbm{C} [[\tmmathbf{y}]] \mapsto f (\tmmathbf{i}\tmmathbf{y}) \in
\mathbbm{C} [[\tmmathbf{y}]]$.

The generating series of $f \in \mathcal{O}'_C$ is
\begin{equation}
  \sigma_f (\tmmathbf{i}\tmmathbf{y}) = \iota \circ \sigma_f (\tmmathbf{y}) =
  \sum_{\alpha \in \mathbbm{N}^n} \int f (\tmmathbf{x}) \tmmathbf{x}^{\alpha}
  \frac{\tmmathbf{i}^{| \alpha |} \tmmathbf{y}^{\alpha}}{\alpha !}
  d\tmmathbf{x}= \int f (\tmmathbf{x})  {\bde}^{\tmmathbf{i}\tmmathbf{x} \cdummy \tmmathbf{y}} d\tmmathbf{x}= (2 \pi)^n
  \mathscr{F}^{- 1} (f). \label{eq:integ}
\end{equation}
This shows that the map $f \in \mathcal{O}'_C \mapsto \sigma_f \in \mathbbm{C}
[[\tmmathbf{y}]]$ is injective and the regularity condition is satisfied.

For $f \in \mathcal{O}'_C$, the Hankel operator $H_{\sigma_f}$ is such that
$\forall g \in \mathbbm{C} [\tmmathbf{x}]$,
\begin{eqnarray*}
  H_{\iota \circ \sigma_f} (g)  & = & \sum_{\alpha \in \mathbbm{N}^n} \int f
  (\tmmathbf{x}) g (\tmmathbf{x}) \tmmathbf{x}^{\alpha} \frac{\tmmathbf{i}^{|
  \alpha |} \tmmathbf{y}^{\alpha}}{\alpha !} d\tmmathbf{x}\\
  & = & \int f (\tmmathbf{x}) g (\tmmathbf{x})  {\bde}^{\tmmathbf{i}\tmmathbf{x} \cdummy \tmmathbf{y}} d\tmmathbf{x}= (2 \pi)^n
  \mathscr{F}^{- 1} (f g).
\end{eqnarray*}
Using Relation (\ref{eq:integ}), we rewrite it as $\forall g \in \mathbbm{C}
[\tmmathbf{x}]$,
\begin{equation}
  H_{\mathscr{F}^{- 1} (f)} (g) =\mathscr{F}^{- 1} (f g) \label{eq:H-F}
\end{equation}
with $\varphi =\mathscr{F}^{- 1} (f) \in \mathcal{O}_M$.

From this relation, we see that the operator $H_{\mathscr{F}^{- 1} (f)}$ can
be extended by continuity to an operator \ $H_{\mathscr{F}^{- 1} (f)} :
\mathcal{O}_M \mapsto \mathcal{O}_M$.

The Hankel operator $H_{\iota \circ \sigma}$ (or $H_{\mathscr{F}^{- 1} (f)}$)
can be related to integral operators on functions defined in terms of
convolution products or cross-correlation. For $\varphi \in \mathcal{S}'$,
the convolution with a distribution $\psi \in \mathcal{O}'_C$ is well-defined
{\cite{schwartz_theorie_1966}}. The convolution operator associated to
$\varphi$ on $\mathcal{O}'_C$ is:
\[ \mathfrak{H}_{\varphi} : \psi \in \mathcal{O}'_C \mapsto \varphi \star \psi
   = \int \varphi (\tmmathbf{x}-\tmmathbf{t}) \psi (\tmmathbf{t})
   d\tmmathbf{t} \in \mathcal{S}'. \]
The image of an element $\psi \in \mathcal{O}'_C$ is a tempered distribution
in $\mathcal{S}'$. The distribution $\varphi$ is the {\tmem{symbol}} of the
operator $\mathfrak{H}_{\varphi}$.

Using the property that $\forall \varphi \in \mathcal{S}', \forall \psi \in
\mathcal{O}'_C, \mathscr{F} (\varphi \star \psi) =\mathscr{F} (\varphi)
\mathscr{F} (\psi) \in \mathcal{S}'$ and the relation (\ref{eq:H-F}), we have
for any $\psi \in \mathcal{O}'_C$,
\[ H_{\mathscr{F}^{- 1} (f)} (g) =\mathscr{F}^{- 1} (f g) = \varphi \star \psi
   =\mathfrak{H}_{\varphi} (\psi), \]
with $f =\mathscr{F} (\varphi) \in \mathcal{S}'$, $g =\mathscr{F} (\psi)
\in \mathcal{O}_M$. We deduce that
\begin{equation}
  \mathfrak{H}_{\varphi} = H_{\varphi} \circ \mathscr{F}
  \label{eq:conv-hankel-fourier}
\end{equation}
with $H_{\varphi} : g \in \mathcal{O}_M \mapsto \mathscr{F}^{- 1} (\mathscr{F}
(\varphi) g) \in \mathcal{S}'$.

In the case where $\varphi \in \PolExp \cap \mathcal{O}_M$,
the operator is of finite rank:

\begin{proposition}
  \label{prop:rank1}Let $\varphi = \omega (\tmmathbf{y})
  \tmmathbf{e}_{\tmmathbf{i} \xi} (\tmmathbf{y})$ with $\omega \in \mathbbm{C}
  [\tmmathbf{y}]$ and $\xi \in \mathbbm{R}^n$. Then $\tmop{rank}
  \mathfrak{H}_{\varphi} \leqslant {\mu} (\omega)$.
\end{proposition}

\begin{proof}
  By Taylor expansion of the polynomial $\omega$ at $\tmmathbf{x}$, we have $\forall \psi \in \mathcal{O}'_C$
  \begin{eqnarray*}
    \mathfrak{H}_{\varphi} (\psi) & = & \int \omega
    (\tmmathbf{x}-\tmmathbf{t}) e^{\tmmathbf{i} \xi \cdummy
    (\tmmathbf{x}-\tmmathbf{t})} \psi (\tmmathbf{t}) d\tmmathbf{t}\\
    & = & \sum_{\alpha \in \mathbbm{N}^n} \partial^{\alpha} (\omega)
    (\tmmathbf{x}) e^{\tmmathbf{i} \xi \cdummy \tmmathbf{x}}  \int (-
    1)^{\alpha} \frac{\tmmathbf{t}^{\alpha}}{\alpha !} \psi (\tmmathbf{t})
    e^{-\tmmathbf{i} \xi \cdummy \tmmathbf{t}} d\tmmathbf{t}.
  \end{eqnarray*}
  This shows that $\mathfrak{H}_{\varphi} (\psi)$ belongs to the space spanned
  by $\partial^{\alpha} (\omega) (\tmmathbf{x}) e^{\xi \cdummy \tmmathbf{x}}$
  for $\alpha \in \mathbbm{N}^n$, which is of dimension ${\mu} (\omega)$
  and thus rank $\mathfrak{H}_{\varphi} \leqslant {\mu} (\omega)$.
\end{proof}

The converse is also true:

\begin{theorem}
  \label{thm:conv1}Suppose that $\varphi \in \mathcal{S}'$ is such that the
  convolution operator $\mathfrak{H}_{\varphi}$ is of finite rank $r$. Then
  its symbol $\varphi$ is of the form
  \[ \varphi = \sum_{i = 1}^{r'} \omega_i (\tmmathbf{y})
     \tmmathbf{e}_{\tmmathbf{i} \xi_i} (\tmmathbf{y}). \]
  with $\xi_i = (\xi_{i, 1}, \ldots, \xi_{i, n}) \in \mathbbm{R}^n$,
  $\omega_i (\tmmathbf{y}) \in \mathbbm{C} [\tmmathbf{y}]$. The rank $r$ of $\mathfrak{H}_{\varphi}$ is the
  sum of the dimension of the vector spaces spanned by $\omega_i
  (\tmmathbf{y})$ and all its derivatives $\tmmathbf{\partial}^{\gamma} \omega_i (\tmmathbf{y})$, $\gamma \in \mathbbm{N}^n$.
\end{theorem}

\begin{proof}
  Since $\mathscr{F}$ is a bijection between $\mathcal{O}'_C$ and
  $\mathcal{O}_M$, the relation (\ref{eq:conv-hankel-fourier}) implies that
  $\mathfrak{H}_{\varphi}$ is of finite rank $r$, if and only if, $H_{\varphi} :
  \mathcal{O}_M \mapsto \mathcal{O}_M$ is of rank $r$. As the
  restriction of $H_{\varphi}$ to the set of polynomials $\mathbbm{C}
  [\tmmathbf{x}] \subset \mathcal{O}_M$ is of rank  $\tilde{r} \leq r = \mathrm{rank} \, \mathfrak{H}_{\varphi}$, Theorem
  \ref{thm:gorenstein} implies that
  \[ \varphi = \sum_{i = 1}^{r'} \sum_{\alpha \in A_i} \omega_{i, \alpha}
     \tmmathbf{y}^{\alpha} \tmmathbf{e}_{\xi'_i} (\tmmathbf{y}) \]
  with $\xi_i' \in \mathbbm{C}^n$ distincts, $A_i \subset \mathbbm{N}^n$
  finite and $\tilde{r} = \sum_{i = 1}^{r'} {\mu} (\omega_i)$ where
  ${\mu} (\omega_i)$ is the dimension of the inverse system of $\omega_i =
  \sum_{\alpha \in A_i} \omega_{i, \alpha} \tmmathbf{y}^{\alpha}$, spanned by
  $\omega_i (\tmmathbf{y})$ and all its derivatives. As $\varphi \in
  \mathcal{S}'$ is a distribution with slow increase at infinity, we have
  $\xi_i' =\tmmathbf{i} \xi_i$ with $\xi_i \in \mathbbm{R}^n$.
  
  By Proposition \ref{prop:rank1}, we have $r = \tmop{rank}
  \mathfrak{H}_{\varphi} \leqslant \sum_{i = 1}^{r'} {\mu} (\omega_i) =
  \tilde{r}$. This shows that $\tmop{rank} \mathfrak{H}_{\varphi} = \sum_{i =
  1}^{r'} {\mu} (\omega_i)$ and concludes the proof of the theorem.
\end{proof}

We can derive a similar result for the convolution by functions
or distributions with support in a bounded domain $\Omega$ of $\mathbbm{R}^n$.
The main ingredient is the decomposition $\mathfrak{H}_{\varphi} = H_{\varphi}
\circ \mathscr{F}$, which extends the construction used in
{\cite{rochberg_toeplitz_1987}} for Hankel and Toeplitz operators on
$L^2 (I)$ where $I$ is a bounded interval in $\mathbbm{R}$.

By the generalized Paley-Wiener theorem (see
{\cite{schwartz_theorie_1966}}[Th{\'e}or{\`e}me XVI]), the
Fourier transform $\mathscr{F}$ is a bijection between the set $\mathcal{E}'$
of distributions with a compact support and \ the set of continuous functions
$f \in C (\mathbbm{R}^n)$ with an analytic extension of exponential type
(there exists $A \in \mathbbm{R}, C \in \mathbbm{R}_+$ \ such that
$\forall \tmmathbf{z} \in \mathbbm{C}^n, | f (\tmmathbf{z}) | \leqslant C e^{A
(| z_1 | + \cdots + | z_n |)}$). Let us denote by $\mathcal{E}' (\Omega)$ the
set of distributions with a support in $\Omega$, and by $\mathcal{{PW}}
(\Omega) = \{ \mathscr{F} (\varphi) \mid \varphi \in \mathcal{E}' (\Omega)
\}$ the set of their Fourier transforms.

\begin{theorem}
  \label{thm:conv2}Let $\Omega, \Xi$ be open bounded domains of
  $\mathbbm{R}^n$ with , $\Upsilon = \Xi + \Omega \subset \mathbbm{R}^n$ and
  $\varphi \in \mathcal{E}' (\Omega)$. The convolution operator
  \[ \mathfrak{H}_{\varphi} : \psi \in \mathcal{E}' (\Xi) \mapsto \int
     \varphi (x - t) \psi (t) \tmop{dt} \in \mathcal{E}' (\Upsilon) \]
  is of finite rank $r$, if and only if, the symbol $\varphi$ is of the form
  \[ \varphi = \mathbbm{1}_{\Omega}  \sum_{i = 1}^{r'} \omega_i
     (\tmmathbf{y}) \tmmathbf{e}_{\xi_i} (\tmmathbf{y}) \]
  where $\xi_i = (\xi_{i, 1}, \ldots, \xi_{i, n}) \in \mathbbm{C}^n$,
  $\omega_i (\tmmathbf{y}) \in \mathbbm{C} [\tmmathbf{y}]$. The rank $r$ of
  $\mathfrak{H}_{\varphi}$ is the sum of the dimensions ${\mu} (\omega_i)$
  of the vector spaces spanned by $\omega_i (\tmmathbf{y})$ and all the
  derivatives $\partial^{\gamma}_{\tmmathbf{y}} \omega_i (\tmmathbf{y})$,
  $\gamma \in \mathbbm{N}^n$.
\end{theorem}

\begin{proof}
  Using the relations $\forall \varphi \in \mathcal{E}' (\Omega), \psi \in
  \mathcal{E}' (\Xi)$,
  \[ \mathscr{F} (\varphi \star \psi) =\mathscr{F} (\varphi) \mathscr{F}
     (\psi), \]
  and (\ref{eq:H-F}), we still have the decomposition
  \[ \mathfrak{H}_{\varphi} = H_{\varphi} \circ \mathscr{F}. \]
  with $H_{\varphi} : g \in \mathcal{{PW}} (\Xi) \rightarrow
  \mathscr{F}^{- 1} (\mathscr{F} (\varphi) g) \in \mathcal{E}' (\Upsilon)$.
  Thus $\mathfrak{H}_{\varphi}$ is of finite rank $r$, if and only if, $H_{\varphi}$ is of
  finite rank $r$.  
  As the rank of the restriction of $H_{\varphi}$ to
  $\mathbbm{C} [\tmmathbf{x}] \subset \mathcal{{PW}} (\Xi)$ is at most
  $r$, we conclude by using Theorem \ref{thm:gorenstein}, a result similar to
  Proposition \ref{prop:rank1} for elements $\psi \in \mathcal{E}' (\Xi)$ and
  the relation $\mathscr{F}^{- 1} (\mathscr{F} (\varphi)) = \varphi$ on
  $\Omega$. 
\end{proof}

Similar results also apply for the cross-correlation operator defined as
\[ \tilde{\mathfrak{H}}_{\varphi} : \psi \in \mathcal{E}' \mapsto \varphi \ast
   \psi = \int \varphi (x + t)  \bar{\psi} (t) \tmop{dt} \in \mathcal{S}'.
\]
Using the relation $\mathscr{F} (\varphi \ast \psi) =\mathscr{F} (\varphi)
\bar{\mathscr{F}} (\psi)$ (with $\bar{\mathscr{F}} = \varsigma \circ
\mathscr{F}$ where $\varsigma : z \in \mathbbm{C} \mapsto \bar{z} \in
\mathbbm{C}$ is the complex conjugation), we have
$\tilde{\mathfrak{H}}_{\varphi} = H_{\varphi} \circ \bar{\mathscr{F}}$. As
$\bar{\mathscr{F}}^{- 1} =\mathscr{F}^{- 1} \circ \varsigma$, we deduce that
$\tilde{\mathfrak{H}}_{\varphi}$ and $H_{\varphi}$ have the same rank and the
same type of decomposition of the symbol $\varphi$ holds when
$\tilde{\mathfrak{H}}_{\varphi}$ is of finite rank.

\begin{remark}
  To compute the decomposition of $\varphi \in \mathcal{S}'$ (resp. $\varphi
  \in \mathcal{E}' (\Omega)$) as a polynomial exponential function, we first
  compute the Taylor coefficients of $\sigma_{\mathscr{F} (\varphi)}
  =\mathfrak{H}_{\varphi} (1)$, that is, the values $\sigma_{\alpha} =
  (-\tmmathbf{i})^{| \alpha |} \mathscr{F} (\tmmathbf{x}^{\alpha} \varphi)
  (\tmmathbf{0})$ for some $\alpha \in \tmmathbf{a} \subset \mathbbm{N}^n$ and
  apply the decomposition algorithm \ref{algo:decomposition}
  to the (truncated) sequence $(\sigma_{\alpha})_{\alpha \in \tmmathbf{a}}$.
\end{remark}

\subsection{Reconstruction from Fourier coefficients}

We consider here the problem of reconstruction of functions or distributions
from Fourier coefficients. Let $T = (T_1, \ldots, T_n) \in \mathbbm{R}_+^n$
and $\Omega = \prod_{i = 1}^n \left[ - \frac{2 \pi T_i}{2},
\frac{2 \pi T_i}{2} \right] \subset \mathbbm{R}^n$.
We take:
\begin{itemize}
  \item $\mathcal{F}= L^2 (\Omega)$;
  
  \item $S_i : h (\tmmathbf{x}) \in L^2 (\Omega) \mapsto e^{2 \pi
  \frac{x_i}{T_i}} h (\tmmathbf{x}) \in L^2 (\Omega)$ is the multiplication by
  $e^{2 \pi \frac{x_i}{T_i}}$;
  
  \item $\Delta : h (\tmmathbf{x}) \in \mathcal{O}'_C \mapsto \int h
  (\tmmathbf{x}) d\tmmathbf{x} \in \mathbbm{C}$.
\end{itemize}
For $f \in \mathcal{E}'(\Omega)$ with a support in $\Omega$ and $\gamma = (\gamma_1,
\ldots, \gamma_n) \in \mathbbm{Z}^n$, the $\gamma$-th Fourier coefficient of
$f$ is
\[ \sigma_{\gamma} = \frac{1}{\prod_{j = 1}^n T_j} \mathscr{F} (f) \left( 2
   \pi \frac{\gamma_1}{T_1}, \ldots, 2 \pi \frac{\gamma_n}{T_n} \right) =
   \frac{1}{\prod_{j = 1}^n T_j}  \int f (\tmmathbf{x}) e^{- 2 \pi
   \tmmathbf{i} \sum_{j = 1}^n \frac{\gamma_j x_j}{T_j}} d\tmmathbf{x} \]
Let $\sigma = (\sigma_{\gamma})_{\gamma \in \mathbbm{Z}^n}$ be the sequence
of the Fourier coefficients. The discrete convolution operator associated to
$\sigma$ is $\Phi_{\sigma} : (\rho_{\beta})_{\beta \in \mathbbm{Z}^n} \in
\ell^2 (\mathbbm{Z}^n) \mapsto \left( \sum_{\beta} \sigma_{\alpha - \beta}
\rho_{\beta} \right)_{\alpha \in \mathbbm{Z}^n} \in \ell^2 (\mathbbm{Z}^n)$. The
discrete cross-correlation operator of $\sigma$ is $\Gamma_{\sigma} :
(\rho_{\beta})_{\beta \in \mathbbm{Z}^n} \in \ell^2 (\mathbbm{Z}^n) \mapsto
\left( \sum_{\beta} \sigma_{\alpha + \beta} \rho_{\beta} \right)_{\alpha \in
\mathbbm{Z}^n} \in \ell^2 (\mathbbm{Z}^n)$. It is obtained from $\Phi_{\sigma}$
by composition by $\mathcal{R}: (\rho_{\beta})_{\beta \in \mathbbm{Z}^n} \in
\ell^2 (\mathbbm{Z}^n) \mapsto (\rho_{- \beta})_{\beta \in \mathbbm{Z}^n} \in
\ell^2 (\mathbbm{Z}^n)$: $\Gamma_{\sigma} = \Phi_{\sigma} \circ \mathcal{R}$.

A decomposition similar to the previous section also holds:

\begin{theorem}
  \label{thm:discreteconv}Let $f \in L^2 (\Omega)$ and let $\sigma =
  (\sigma_{\gamma})_{\gamma \in \mathbbm{Z}^n}$ be its sequence of Fourier
  coefficients. The discrete convolution (resp. cross-correlation) operator
  $\Phi_{\sigma}$ (resp. $\Gamma_{\sigma}$) is of finite rank if and only if
  \[ f = \sum_{i = 1}^{r'}  \sum_{\alpha \in A_i \subset \mathbbm{N}^n}
     \omega_{i, \alpha} \delta_{\xi_i}^{(\alpha)} \]
  where
  \begin{itemize}
    \item $\xi_i = (\xi_{i, 1}, \ldots, \xi_{i, n}) \in \Omega$, $\omega_{i,
    \alpha} \in \mathbbm{C}$, $A_i \subset \mathbbm{N}^n$ is finite,
    
    \item the rank of $\Gamma_{\sigma}$ is the sum of the dimensions
    ${\mu} (\omega_i)$ of the vector spaces spanned by \ $\omega_i
    (\tmmathbf{y}) = \sum_{\alpha \in A_i} \omega_{i, \alpha}
    \tmmathbf{y}^{\alpha}$ and all the derivatives
    $\partial_{\tmmathbf{y}}^{\gamma} (\omega_i)$, $\gamma \in \mathbbm{N}^n$.
  \end{itemize}
\end{theorem}

\begin{proof}
  Let $S : f \in L^2 (\Omega) \mapsto (\sigma_{\gamma})_{\gamma \in
  \mathbbm{Z}^n} \in \ell^2 (\mathbbm{Z}^n)$ be the discrete Fourier transform
  where $\sigma_{\gamma} = \frac{1}{\prod_{j = 1}^n T_j} \mathscr{F} (f)
  \left( 2 \pi \frac{\gamma_1}{T_1}, \ldots, 2 \pi \frac{\gamma_n}{T_n}
  \right)$. Its inverse is $S^{- 1} : \sigma = (\sigma_{\gamma})_{\gamma \in
  \mathbbm{Z}^n} \in \ell^2 (\mathbbm{Z}^n) \mapsto \sum_{\alpha \in
  \mathbbm{Z}^n} \sigma_{\gamma} \mathbbm{1}_{\Omega} e^{2 \pi \tmmathbf{i}
  \sum_{j = 1}^n \frac{\gamma_j x_j}{T_j}} \in L^2 (\Omega)$. As the discrete
  Fourier transform exchanges the convolution and the product, using Relation
  (\ref{eq:H-F}), we have $\forall \sigma, \rho \in \ell^2 (\mathbbm{Z}^n)$,
  \[ 
    \Phi_{\sigma} (\rho) = S (S^{- 1} (\sigma) S^{- 1} (\rho)) = S (f g)= S \circ
     \mathscr{F} \circ \mathscr{F}^{- 1} (f g) = S \circ \mathscr{F} \circ
     H_{\mathscr{F}^{- 1} (f)} (g)
     \]
  where $f = S^{- 1} (\sigma), g = S^{- 1} (\rho) \in L^2 (\Omega)$ and
  $H_{\mathscr{F}^{- 1} (f)} : g \in L^2 (\Omega) \mapsto \mathscr{F}^{- 1}
  (f g) \in \mathcal{{PW}} (\Omega)$. We deduce that
  \[ \Phi_{\sigma} = S \circ \mathscr{F} \circ H_{\mathscr{F}^{- 1} \circ S^{-
     1} (\sigma)} \circ S^{- 1}. \]
  As $S$ is an isometry between $\ell^2 (\mathbbm{Z}^n)$ and $L^2 (\Omega)$ and
  $\mathscr{F}$ is an isomorphism between $L^2 (\Omega)$ and
  $\mathcal{{PW}} (\Omega)$, $\Phi_{\sigma} = S \circ \mathscr{F} \circ
  H_{\mathscr{F}^{- 1} \circ S^{- 1} (\sigma)} \circ S^{- 1}$ and
  $H_{\mathscr{F}^{- 1} \circ S^{- 1} (\sigma)}$ have the same rank.
  
  As $\mathbbm{C} [\tmmathbf{x}] \subset \mathcal{\tmop{PW}} (\Omega)$, we
  deduce from Theorem \ref{thm:gorenstein} that
  \[ \mathscr{F}^{- 1} \circ S^{- 1} (\sigma) = \sum_{i = 1}^{r'}
     \tilde{\omega}_i (\tmmathbf{y}) \tmmathbf{e}_{\tilde{\xi}_i}
     (\tmmathbf{y}) \]
  where $\tilde{\xi}_i = (\tilde{\xi}_{i, 1}, \ldots, \tilde{\xi}_{i, n}) \in
  \mathbbm{C}^n$, $\tilde{\omega}_i (\tmmathbf{y}) = \sum_{\alpha \in A_i}
  \tilde{\omega}_{i, \alpha} \tmmathbf{y}^{\alpha} \in \mathbbm{C}
  [\tmmathbf{z}]$. Using a result similar to Proposition \ref{prop:rank1} for
  the elements $\psi \in L^2 (\Omega)$, we deduce that the rank $r$ of
  $\Phi_{\sigma}$ is $r = \sum_{i = 1}^{r'} {\mu} (\tilde{\omega}_i)$.
  Consequently, \
  \[ f = S^{- 1} (\sigma) =\mathscr{F} \left( \sum_{i = 1}^{r'} \sum_{\alpha
     \in A_i} \tilde{\omega}_{i, \alpha} \tmmathbf{y}^{\alpha}
     \tmmathbf{e}_{\tilde{\xi}_i} (\tmmathbf{y}) \right) = (2 \pi)^n \sum_{i =
     1}^{r'}  \sum_{\alpha \in A_i \subset \mathbbm{N}^n} \tmmathbf{i}^{|
     \alpha |} \tilde{\omega}_{i, \alpha} \delta_{\tmmathbf{i}
     \tilde{\xi}_i}^{(\alpha)}. \]
  As the support of $f$ is in $\Omega$, we have $\xi_i =\tmmathbf{i}
  \tilde{\xi}_i \in \Omega$. We deduce the decomposition of $f$ with $\omega_i
  = (2 \pi)^n \sum_{i = 1}^{r'}  \sum_{\alpha \in A_i \subset \mathbbm{N}^n}
  \tmmathbf{i}^{| \alpha |} \tilde{\omega}_{i, \alpha} \tmmathbf{y}^{\alpha}$.
  
  The dimension ${\mu} (\tilde{\omega}_i)$ of the vector space spanned by
  $\tilde{\omega}_i (\tmmathbf{y}) = \sum_{\alpha \in A_i} \omega_{i, \alpha}
  \tmmathbf{y}^{\alpha}$ and all its derivatives is the same as the dimension
  ${\mu} (\omega_i)$ of the space spanned by $\omega_i (\tmmathbf{y}) = (2
  \pi)^n \sum_{\alpha \in A_i} \omega_{i, \alpha} \tmmathbf{i}^{| \alpha |}
  \tmmathbf{y}^{\alpha}$ and all its derivatives, since $\omega_i
  (\tmmathbf{y}) = (2 \pi)^n  \tilde{\omega}_i (\tmmathbf{i}\tmmathbf{y})$.
  Therefore, $\tmop{rank} \Phi_{\sigma} = r = \sum_{i = 1}^{r'} {\mu}
  (\tilde{\omega}_i) = \sum_{i = 1}^{r'} {\mu} (\omega_i)$. This concludes
  the proof of the theorem.
\end{proof}

\begin{remark}
  To compute the decomposition of $f \in L^2 (\Omega)$ as a weighted sum of
  Dirac measures \ and derivates, we apply the decomposition algorithm
 \ref{algo:decomposition} to the (truncated)
  sequence of Fourier coefficients $(\sigma_{\alpha})_{\alpha \in
  \tmmathbf{a}}$ for some subset $\tmmathbf{a} \subset \mathbbm{N}^n$. The
  polynomial-exponential decomposition $\varphi = \sum_{i = 1}^{r'}
  \sum_{\alpha \in A_i} \tilde{\omega}_{i, \alpha} \tmmathbf{y}^{\alpha}
  \tmmathbf{e}_{\tilde{\xi}_i} (\tmmathbf{y})$, from which we deduce the
  decomposition $f = (2 \pi)^n \sum_{i = 1}^{r'}  \sum_{\alpha \in A_i \subset
  \mathbbm{N}^n} \tmmathbf{i}^{| \alpha |} \tilde{\omega}_{i, \alpha}
  \delta_{\tmmathbf{i} \tilde{\xi}_i}^{(\alpha)}$.
\end{remark}

\subsection{Reconstruction from values}\label{sec:values}

In this problem, we are interested in reconstructing a function in $C^{\infty}
(\mathbbm{R}^n)$ from sampled values. We take
\begin{itemize}
  \item $\mathcal{F}= C^{\infty} (\mathbbm{R}^n)$,
  
  \item $S_j : h (x) \mapsto h (x_{1}, \ldots, x_{j - 1}, x_j + 1, x_{j + 1},
  \ldots, x_n)$ the shift operator of $x_j$ by 1,
  
  \item $\Delta : h (x) \mapsto \Delta [h] = h (0)$ the evaluation at
  $0$.
\end{itemize}
The generating series of $h$ is
\[ \sigma_h (\tmmathbf{y}) = \sum_{\alpha \in \mathbbm{N}^n} h (\alpha_1,
   \ldots, \alpha_n)  \frac{\tmmathbf{y}^{\alpha}}{\alpha !} = \sum_{\alpha
   \in \mathbbm{N}^n} h (\alpha)  \frac{\tmmathbf{y}^{\alpha}}{\alpha !}. \]
The operators $S_j$ are commuting and we have $S_j (e^{{\bdf} \cdummy
\tmmathbf{x}}) = \xi_j e^{{\bdf} \cdummy {\bdx}}$ where
$\tmmathbf{f}= (f_1, \ldots, f_n) \in \mathbbm{C}^n$ and $\xi_j = e^{f_j}$.
The generating series associated to $e^{{\bdf} \cdummy \tmmathbf{x}}$
is $\tmmathbf{e}_{\xi} (\tmmathbf{y})$ where $\xi = (\xi_1, \ldots, \xi_n) =
(e^{f_1}, \ldots, e^{f_n})$.

Similarly for any $\alpha = (\alpha_1, \ldots, \alpha_n) \in \mathbbm{N}$,
$S_j (\tmmathbf{x}^{\alpha} e^{{\bdf} \cdummy \tmmathbf{x}}) = \xi_j 
\sum_{i = 0}^{\alpha_j}  \binom{\alpha_j}{i} x_j^i \prod_{i \neq j}
x_i^{\alpha_i} e^{{\bdf} \cdummy \tmmathbf{x}}$, which shows that the
function $\tmmathbf{x}^{\alpha} e^{{\bdf} \cdummy \tmmathbf{x}}$ is a
generalized eigenfunction of the operators $S_j$. Its generating series is
\begin{equation}
  \sigma_{\tmmathbf{x}^{\alpha} e^{{\bdf} \cdummy \tmmathbf{x}}}
  (\tmmathbf{y}) = \sum_{\beta \in \mathbbm{N}^n} \beta^{\alpha} \xi^{\beta} 
  \frac{\tmmathbf{y}^{\beta}}{\beta !}. \label{eq:genbeta}
\end{equation}
Let $b_{\alpha} (\tmmathbf{y}) = \binom{y_1}{\alpha_1} \cdots
\binom{y_n}{\alpha_n}$ be the Macaulay binomial polynomial with
$\binom{y_i}{\alpha_i} = \frac{1}{\alpha_i !} y_i  (y_i - 1) \cdots (y_i -
\alpha_i + 1)$, which roots are $0, \ldots, \alpha_i - 1$. It satisfies the
following relations:
\[ \sum_{\beta \in \mathbbm{N}^n} b_{\alpha} (\beta) \xi^{\beta} 
   \frac{\tmmathbf{y}^{\beta}}{\beta !} = \sum_{\beta_{\gg} \alpha}
   b_{\alpha} (\beta) \xi^{\beta}  \frac{\tmmathbf{y}^{\beta}}{\beta !} =
   \sum_{\beta \gg \alpha} \frac{1}{\alpha !} \xi^{\beta} 
   \frac{\tmmathbf{y}^{\beta}}{(\beta - \alpha) !} = \frac{1}{\alpha !}
   \xi^{\alpha} \tmmathbf{y}^{\alpha} \tmmathbf{e}_{\xi} (\tmmathbf{y}). \]
As $\tmmathbf{y}^{\alpha} = \sum_{\alpha' \ll \alpha} m_{\alpha', \alpha}
b_{\alpha'} (\tmmathbf{y})$ for some coefficients $m_{\alpha', \alpha} \in
\mathbbm{Q}$ such that $m_{\alpha, \alpha} = 1$, we have
\begin{equation}
  \sigma_{\tmmathbf{x}^{\alpha} e^{{\bdf} \cdummy \tmmathbf{x}}}
  (\tmmathbf{y}) = \left( \sum_{\alpha' \ll \alpha} m_{\alpha', \alpha}
  \xi^{\alpha'} \frac{\tmmathbf{y}^{\alpha'}}{\alpha' !} \right)
  \tmmathbf{e}_{\xi} (\tmmathbf{y}) = \omega_{\alpha} (\tmmathbf{y})
  \tmmathbf{e}_{\xi} (\tmmathbf{y}). \label{eq:genxa}
\end{equation}
The monomials of $\omega_{\alpha} (\tmmathbf{y})$ are among the monomials
$\tmmathbf{y}^{\alpha'} = y_1^{\alpha'_1} \cdots y_n^{\alpha_n'}$ such that $0
\leqslant \alpha'_i \leqslant \alpha_i$, which divide $\tmmathbf{y}^{\alpha}$.
As the coefficient of $\tmmathbf{y}^{\alpha}$ in $\omega_{\alpha}
(\tmmathbf{y})$ is $1$, we deduce that $(\omega_{\alpha})_{\alpha \in
\mathbbm{N}^n}$ is a basis of $\mathbbm{C} [\tmmathbf{y}]$ and the
completeness property is satisfied. \

Let $h = (h (\alpha))_{\alpha \in \mathbbm{N}^n}$. The Hankel operator $H_h$
is such that $\forall p = \sum_{\beta} p_{\beta} \tmmathbf{x}^{\beta} \in
\mathbbm{C} [\tmmathbf{x}]$,
\[ H_h (p) = \sum_{\alpha \in \mathbbm{N}^n} \left( \sum_{\beta} h
   (\alpha + \beta) p_{\beta} \right)  \frac{\tmmathbf{y}^{\alpha}}{\alpha !}
\]
Identifying the series $\sigma (\tmmathbf{y}) = \sum_{\alpha \in
\mathbbm{N}^n} \sigma_{\alpha}  \frac{\tmmathbf{y}^{\alpha}}{\alpha !} \in
\mathbbm{C} [[\tmmathbf{y}]]$ with the multi-index sequence
$(\sigma_{\alpha})_{\alpha \in \mathbbm{N}^n}$ and a polynomial $p =
\sum_{\alpha \in A} p_{\alpha} \tmmathbf{x}^{\alpha}$ with the sequence
$(p_{\alpha})_{\alpha \in \mathbbm{N}^n}$ $\fsupp (\mathbbm{N}^n)$ of
finite support, the operator $H_h$ corresponds to the discrete
cross-correlation operator by the sequence $h$. This operator can
be extended to sequences $h, p$ are in $\ell^2 (\mathbbm{N}^n)$.

\begin{theorem}
  \label{thm:decvalues}Let $h \in C^{\infty} (\mathbbm{R}^n)$. The discrete
  cross-correlation operator $\Gamma_h : p \in \ell^2 (\mathbbm{N}^n) \mapsto h
  \star p = \left( \sum_{\beta} h (\alpha + \beta) p_{\beta} \right)_{\alpha
  \in \mathbbm{N}^n} \in \ell^2 (\mathbbm{N}^n)$ is of finite rank, if and only
  if,
  \[ h (\tmmathbf{x}) = \sum_{i = 1}^{r'} g_i (\tmmathbf{x}) e^{f_i \cdummy
     \tmmathbf{x}} + r (\tmmathbf{x}) \]
  where
  \begin{itemize}
    \item $f_i = (f_{i, 1}, \ldots, f_{i, n}) \in \mathbbm{C}^n$, $g_i
    (\tmmathbf{x}) \in \mathbbm{C} [\tmmathbf{x}]$,
    
    \item $r (\tmmathbf{x}) \in C^{\infty} (\mathbbm{R}^n)$ such that $r
    (\alpha) = 0$, $\forall \alpha \in \mathbbm{N}^n$, \
    
    \item The rank of $\Gamma_h$ is the sum of the dimension ${\mu}
    (g_i)$ of the vector space spanned by $g_i (\tmmathbf{x})$ and all its
    derivatives $\partial_{\tmmathbf{x}}^{\alpha} g_i$, $\alpha \in
    \mathbbm{N}^n$.
  \end{itemize}
\end{theorem}

\begin{proof}
  Since $H_h$ is of finite rank, Theorem \ref{thm:gorenstein} implies
  that
  \[ \sigma_h = \sum_{\alpha \in \mathbbm{N}^n} h (\alpha) 
     \frac{\tmmathbf{y}^{\alpha}}{\alpha !} = \sum_{i = 1}^{r'} \omega_i
     (\tmmathbf{y}) \tmmathbf{e}_{\xi_i} (\tmmathbf{y}) \]
  where $\xi_i \in \mathbbm{C}^n$, $\omega_i (\tmmathbf{x}) \in \mathbbm{C}
  [\tmmathbf{x}]$ and rank $H_{\sigma_h} = \sum_{i = 1}^{r'} {\mu}
  (\omega_i)$. Let ${\bdf}_i = (f_{i, 1}, \ldots, f_{i, n}) \in
  \mathbbm{C}^n$ such that $\xi_i = (e^{f_{i, 1}}, \ldots, e^{f_{i, n}})$ and
  $g_{i, \alpha} \in \mathbbm{C}$ for $\alpha \in A_i \subset \mathbbm{N}^n$
  such that
  \[ \omega_i (\tmmathbf{y}) = \sum_{\alpha \in A_i} g_{i, \alpha}
     \omega_{\alpha} (\tmmathbf{y}). \]
  By the relation (\ref{eq:genxa}), the generating series of $r(\bdx)=h - \sum_{i =
  1}^r \sum_{\alpha \in A_i} g_{i, \alpha} \tmmathbf{x}^{\alpha} e^{{\bdf}_i \cdummy \tmmathbf{x}} $ is $0$, which implies that $r$ is a
  function in $C^{\infty} (\mathbbm{R}^n)$ such that $r (\alpha) = 0$,
  $\forall \alpha \in \mathbbm{N}^n$.
  
  It remains to prove that the inverse systems spanned by $\omega_i
  (\tmmathbf{y}) = \sum_{\alpha \in A_i} g_{i, \alpha} \omega_{\alpha}
  (\tmmathbf{y})$ and by $g_i (\tmmathbf{x}) = \sum_{\alpha \in A_i} g_{i,
  \alpha} \tmmathbf{x}^{\alpha}$ have the same dimension. The polynomials
  $\omega_{\alpha}$ are of the form
  \[ \omega_{\alpha} (\tmmathbf{y}) =\tmmathbf{y}^{\alpha} + \sum_{\alpha'
     \neq \alpha, \alpha' \ll \alpha} \omega_{\alpha, \alpha'}
     \tmmathbf{y}^{\alpha'}, \]
  with $\omega_{\alpha, \alpha'} \in \mathbbm{Q}$. Let $\rho$ denotes the
  linear map of $\mathbbm{C} [\tmmathbf{y}]$ such that $\rho
  (\tmmathbf{y}^{\alpha}) = \omega_{\alpha} (\tmmathbf{y})
  -\tmmathbf{y}^{\alpha}$. We choose a monomial ordering $\succ$, which is a
  total ordering on the monomials compatible with the multiplication. Then,
  the initial $\tmop{in} (\omega_{\alpha})$ of $\omega_{\alpha}$, that is the
  maximal monomial of the support of $\omega_{\alpha}$, is
  $\tmmathbf{y}^{\alpha}$ since $\tmmathbf{y}^{\alpha} \succ \tmop{in} (\rho
  (\tmmathbf{y}^{\alpha}))$. 
 As the support of $\omega_{\alpha}$ is in $\{\alpha', \alpha' \ll \alpha\}$, 
 the support of $\partial^{\beta}\omega_{\alpha}$ ($\beta\in \mathbbm{N}^{n}$)
 is  $\{\alpha', \alpha' \ll \alpha-\beta\}$ and  the initial of 
 $\tmmathbf{\partial}^{\beta}\omega_{\alpha}$ is $\tmmathbf{\partial}^{\beta}(\bdx^{\alpha})$.
 By linearity, for any $g \in \mathbbm{C}
  [\tmmathbf{y}]$, we have $\tmop{in} (g) \succ \tmop{in} (\rho (g))$. We
  deduce that
  \[ \omega_i (\tmmathbf{y}) = \sum_{\alpha \in A_i} g_{i, \alpha}
     \omega_{\alpha} (\tmmathbf{y}) = \sum_{\alpha \in A_i} g_{i, \alpha} 
     (\tmmathbf{y}^{\alpha} + \rho (\tmmathbf{y}^{\alpha})) = g_i
     (\tmmathbf{y}) + \rho (g_i) 
\]
and the initial $\tmop{in} (\tmmathbf{\partial}^{\beta}\omega_i)$ is also the initial of  $\tmmathbf{\partial}^{\beta}g_{i}$ ($\beta\in \mathbbm{N}^{n}$).
  Therefore the initial of the vector space spanned by $\omega_i  (\tmmathbf{y}) = g_i (\tmmathbf{y}) + \rho (g_i)$ and all its derivatives
  coincides with the vector space spanned by the initial of $\omega_i (\tmmathbf{y}) = g_i (\tmmathbf{y})$ and all its derivatives. Therefore, the two
  vector spaces have the same dimension. This concludes the proof.
\end{proof}

\begin{remark}
  Instead of a shift by $1$ and the generating series of $h$ computed on the
  unitary grid $\mathbbm{N}^n$, one can consider the shift
  $S_j (h) = h \left( x_{1}, \ldots, x_{j - 1}, x_j + \frac{1}{T_i}, x_{j + 1}, \ldots, x_n \right)$ for $T_j \in \mathbbm{R}_+$ and the generating
  series of the sequence $\left( h \left( \frac{\alpha_1}{T_1}, \ldots,
  \frac{\alpha_n}{T_n} \right) \right)_{\alpha \in \mathbbm{N}^n}$. The
  previous results apply directly, replacing the function $h$ by $h_T : (x_1,
  \ldots, x_n) \mapsto h \left( \frac{x_1}{T_1}, \ldots, \frac{x_n}{T_n}
  \right)$ where $T = (T_1, \ldots, T_n)$.
\end{remark}

\begin{remark}
  Using Lemma \ref{lem:indep}, we check that the map $h \in
  \PolExp \mapsto \sigma_h \in \mathbbm{C} [[\tmmathbf{y}]]$
  is injective and the regularity condition is satisfied on
  $\PolExp$. Thus, in Theorem \ref{thm:decvalues} if $h \in
  \PolExp$ then we must have $r (\tmmathbf{x}) = 0$. 
\end{remark}

\begin{remark}
  By applying Algorithm \ref{algo:decomposition} to the sequence of
  evaluations of a function $h \in
  \PolExp$ on the (first) points of a
  regular grid in $\mathbbm{R}^n$, we obtain a
  method to decompose functions  in $\in
  \PolExp$ as a sum of products of polynomials by
  exponentials.
\end{remark}

\subsection{Sparse interpolation}

For $\beta = (\beta_1, \ldots, \beta_n) \in \mathbbm{N}^n$ and $\tmmathbf{x}
\in \mathbbm{C}^n$, we denote $\log^{\beta} \tmmathbf{x}= \prod_{i = 1}^n
(\log (x_i))^{\beta_i}$ where $\log (x)$ is the principal value of the complex
logarithm $\mathbbm{C}\setminus\{0\}$. Let 
$$ 
\PolyLog (x_1, \ldots, x_n) = \left\{
\sum_{\alpha, \beta} p_{\alpha, \beta} \tmmathbf{x}^{\alpha} \log^{\beta}
(\tmmathbf{x}), p_{\alpha, \beta} \in \mathbbm{C} \right\}
$$
be the set of functions, which are the sum of products of polynomials in $\tmmathbf{x}$ and
polynomials in $\log (\tmmathbf{x})$.

For $h = \sum_{\alpha, \beta} h_{\alpha, \beta} \tmmathbf{x}^{\alpha}
\log^{\beta} (\tmmathbf{x}) \in \PolyLog (\tmmathbf{x})$, we
denote by $\varepsilon (h)$ the set of exponents $\alpha \in \mathbbm{N}^n$
such that $h_{\alpha, \beta} \neq 0$.

The sparse interpolation problem consists in computing the decomposition of a
function $p$ of $\PolyLog (\tmmathbf{x})$ as a sum of terms of
the form $p_{\alpha, \beta} \tmmathbf{x}^{\alpha} \log^{\beta} (\tmmathbf{x})$
from the values of $p$. We apply the construction introduced in Section
\ref{sec:dec} with
\begin{itemize}
  \item $\mathcal{F}= \PolyLog (\tmmathbf{x})$,
  
  \item $S_j : h (x_1, \ldots, x_n) \mapsto h (x_{1}, \ldots, x_{j - 1},
  \lambda_j x_j, x_{j + 1}, \ldots, x_n)$ the scaling operator of $x_j$ by
  $\lambda_j \in \mathbbm{C}$,
  
  \item $\Delta : h (x_1, \ldots, x_n) \mapsto \Delta [h] = h (1,
  \ldots, 1)$ the evaluation at $\tmmathbf{1}= (1, \ldots, 1)$.
\end{itemize}
We easily check that
\begin{itemize}
  \item the operators $S_j$ are commuting,
  
  \item for $\alpha = (\alpha_1, \ldots, \alpha_n) \in \mathbbm{N}^n$, the
  monomial $\tmmathbf{x}^{\alpha}$ is an eigenfunction of $S_j$: $S_j
  (\tmmathbf{x}^{\alpha}) = \lambda_j^{\alpha_j}
  \tmmathbf{x}^{\alpha}$.
  
  \item for $\alpha = (\alpha_1, \ldots, \alpha_n) \in \mathbbm{N}^n$, $\beta
  = (\beta_1, \ldots, \beta_n) \in \mathbbm{N}^n$, $\tmmathbf{x}^{\alpha}
  \log^{\beta} (\tmmathbf{x})$ is a generalized eigenfunction of $S_j$:
  \[ S_j (\tmmathbf{x}^{\alpha} \log^{\beta} ( \tmmathbf{x} )) = \sum_{0
     \leqslant \beta' \leqslant \beta_j} \lambda_j^{\alpha_j}
     \binom{\beta_j}{\beta'} \log^{\beta_j - \beta'} \lambda_j \log^{\beta'}
     (x_j) \tmmathbf{x}^{\alpha}  \prod_{k \neq j} \log^{\beta_k} (x_k). \]
\end{itemize}
More generally, for $\gamma \in \mathbbm{N}^n$, we have
\begin{eqnarray*}
  S^{\gamma} (\tmmathbf{x}^{\alpha} \log^{\beta} (\tmmathbf{x})) & = & \left(
  \prod_{i = 1}^n (\lambda_i^{\gamma_i} x_i)^{\alpha_i} \right) \left(
  \prod_{i = 1}^n (\gamma_i \log (\lambda_i) + \log (x_i))^{\beta_i} \right)\\
  & = & \xi^{\gamma} \tmmathbf{x}^{\alpha} \left( \sum_{\beta' \ll \beta}
  \binom{\beta}{\beta'} \gamma^{\beta'} \log^{\beta'} (\tmmathbf{\lambda})
  \log^{\beta - \beta'} (\tmmathbf{x}) \right)
\end{eqnarray*}
where $\xi = (\lambda_1^{\alpha_1}, \ldots, \lambda_n^{\alpha_n})$. We
deduce that,
\begin{equation}
  \Delta [S^{\gamma} (\tmmathbf{x}^{\alpha} \log^{\beta} (\tmmathbf{x}))] =
  \xi^{\gamma} \gamma^{\beta} \log^{\beta} (\tmmathbf{\lambda}).
  \label{eq:polylogcoeff}
\end{equation}
\begin{theorem}
  Let $h \in \PolyLog (\tmmathbf{x})$. For $\lambda_1, \ldots,
  \lambda_n \in \mathbbm{C}$, the generating series $\sigma_h = \sum_{\gamma
  \in \mathbbm{N}^n} h (\lambda_1^{\gamma_1}, \ldots, \lambda_n^{\gamma_n})
  \frac{\tmmathbf{y}^{\gamma}}{\gamma !}$ of $h$ is of the form
  \[ \sigma_h (\tmmathbf{y}) = \sum_{i = 1}^{r'} \omega_i (\tmmathbf{y})
     \tmmathbf{e}_{\xi_i} (\tmmathbf{y}) \]
  with
  \begin{itemize}
    \item $\varepsilon (h) = \{ \alpha_1, \ldots, \alpha_{r'} \},$
    
    \item $\xi_i = (\lambda_1^{\alpha_{i, 1}}, \ldots, \lambda_n^{\alpha_{i,
    n}}) \in \mathbbm{C}^n$,
    
    \item $\omega_i (\tmmathbf{y}) = \sum_{\beta \in B_i} \omega_{i, \beta}
    \tmmathbf{y}^{\beta} \in \mathbbm{C} [\tmmathbf{y}]$.
  \end{itemize}
  If moreover $\lambda_i \neq 1$ and the points $\xi_i =
  (\lambda_1^{\alpha_{i, 1}}, \ldots, \lambda_n^{\alpha_{i, n}})$, $\alpha_i
  \in \varepsilon (h)$ are distinct, then $h = \sum_{i = 1}^{r'} \sum_{\beta
  \in B_i} \omega_{i, \beta} \tmmathbf{x}^{\alpha_i} \log^{\beta}
  (\tmmathbf{x})$.
\end{theorem}

\begin{proof}
  Let $\alpha, \beta \in \mathbbm{N}^n$. As $\bdx^{\alpha}$ is an eigenfunction of the operators $S_{j}$, its generating series associated to
  $\tmmathbf{x}^{\alpha}$ is $\tmmathbf{e}_{\xi} (\tmmathbf{y})$ where $\xi =
  (\lambda_1^{\alpha_1}, \ldots, \lambda_n^{\alpha_n} )$. From the relations
  (\ref{eq:genbeta}) and (\ref{eq:genxa}), we deduce that the generating
  series of $\tmmathbf{x}^{\alpha} \log^{\beta} (\tmmathbf{x})$ is
  \[ \sigma_{\tmmathbf{x}^{\alpha} \log^{\beta} (\tmmathbf{x})} =
     \log^{\beta} (\tmmathbf{\lambda}) \sum_{\gamma \in \mathbbm{N}^n}
     \gamma^{\beta} \xi^{\gamma}  \frac{\tmmathbf{y}^{\gamma}}{\gamma !} =
     \log^{\beta} (\tmmathbf{\lambda}) \omega_{\beta} (\tmmathbf{y})
     \tmmathbf{e}_{\xi} (\tmmathbf{y}) \]
  where $\omega_{\beta} (\tmmathbf{y})$ is the polynomial obtained from the
  expansion of $\tmmathbf{y}^{\beta}$ in terms of the Macaulay binomial
  polynomials $b_{\alpha} (\tmmathbf{y})$. As in Section \ref{sec:values},
  this shows that the completeness property is satisfied.
  
  If $h = \sum_{i = 1}^{r'} \sum_{\beta \in B_i} h_{i, \beta}
  \tmmathbf{x}^{\alpha_i} \log^{\beta} (\tmmathbf{x})$, $\lambda_i \neq 1$ and
  the points $\xi_i = (\lambda_1^{\alpha_{i, 1}}, \ldots,
  \lambda_n^{\alpha_{i, n}})$ are distinct, then
  \[ \sigma_h = \sum_{i = 1}^{r'} \left( \sum_{\beta \in B_i} h_{i, \beta}
     \log^{\beta} (\tmmathbf{\lambda}) \omega_{\beta} (\tmmathbf{y}) \right)
     \tmmathbf{e}_{\xi_i} (\tmmathbf{y}) = \sum_{i = 1}^{r'} \omega_i
     (\tmmathbf{y}) \tmmathbf{e}_{\xi_i} (\tmmathbf{y}) \]
  with $\xi_i = (\lambda_1^{\alpha_{i, 1}}, \ldots, \lambda_n^{\alpha_{i,
  n}})$ and $\omega_i (\tmmathbf{y}) \in \mathbbm{C} [\tmmathbf{x}]$. By Lemma
  \ref{lem:indep} and the linear independency of the polynomials
  $\omega_{\beta}$, we deduce that the coefficients $h_{i, \beta}$ are
  uniquely determined from the coefficients of the decomposition of $\omega_i
  (\tmmathbf{y})$ in terms of the Macaulay binomial polynomials
  $\omega_{\beta}$, since $\log^{\beta} (\tmmathbf{\lambda}) \neq 0$.
\end{proof}

This result leads to a new method to decompose an element $h \in
\PolyLog(\tmmathbf{x})$ with an exponent set $\varepsilon (h)
\subset A \subset \mathbbm{N}^n$. By choosing $\lambda_1, \ldots, \lambda_n
\in \mathbbm{C} \setminus \{ 1 \}$ such that the points
$(\lambda_1^{\alpha}, \ldots, \lambda_n^{\alpha})$ for $\alpha \in A$ are
distinct and by computing the decomposition of the generating series as a
polynomial-exponential series $\sum_{i = 1}^{r'} \omega_i (\tmmathbf{y})
\tmmathbf{e}_{\xi_i} (\tmmathbf{y})$ (Algorithm \ref{algo:decomposition}), we deduce the exponents
$\alpha_i = (\log_{\lambda_1} (\xi_{i, 1}), \ldots, \log_{\lambda_n} (\xi_{i,
n}))$ and the coefficients $h_{i, \beta}$ in the decomposition $h = \sum_{i =
1}^{r'} \sum_{\beta \in B_i} h_{i, \beta} \tmmathbf{x}^{\alpha_i} \log^{\beta}
(\tmmathbf{x})$ from the weight polynomials $\omega_i (\tmmathbf{y})$.

This method generalizes the sparse interpolation methods of
{\cite{BenOrTiwari:1988}}, {\cite{Zippel:1979}}, {\cite{GiLaWe09}}, where a
single operator $S : h (x_1, \ldots, x_n) \mapsto h (\lambda_1 x_{1}, \ldots,
\lambda_n x_n)$ is used for some $\lambda_1, \ldots, \lambda_n \in
\mathbbm{C}$ and where only polynomial functions are considered. The monomials
$\tmmathbf{x}^{\alpha}$ ($\alpha \in \mathbbm{N}^n$) are eigenfunctions of $S$
for the eigenvalue $\tmmathbf{\lambda}^{\alpha} = \prod_{i = 1}^n
\lambda_i^{\alpha_i}$. For $h = \sum_{i = 1}^r \omega_i
\tmmathbf{x}^{\alpha_i}$, the corresponding univariate generating series
$\sigma_h$ defines an Hankel operator, which kernel is generated by the
polynomial $p (x) = \prod_{i = 1}^r (x -\tmmathbf{\lambda}^{\alpha_i})$ when
$\tmmathbf{\lambda}^{\alpha_1}, \ldots, \tmmathbf{\lambda}^{\alpha_r}$ are
distinct. If $\lambda_1, \ldots, \lambda_n \in \mathbbm{C}$ are chosen
adequately (for instance distinct prime integers {\cite{BenOrTiwari:1988}},
{\cite{Zippel:1979}} or roots of unity of different orders {\cite{GiLaWe09}}),
the roots of $p$ yield the exponents of $h \in \mathbbm{C}
[\tmmathbf{x}]$.

The multivariate approach allows to use moments $h (\lambda_1^{\alpha_1},
\ldots, \lambda_n^{\alpha_n})$ with $\alpha = (\alpha_1, \ldots, \alpha_n) \in
\mathbbm{N}^n$ of degree $| \alpha | = \alpha_1 + \cdots + \alpha_n$ less than
the degree $2 r - 1$ needed in the previous sparse interpolation methods. Sums
of products of polynomials and logarithm functions can also be recovered by
this method, the logarithm terms corresponding to multiple roots.

{\small

\begin{thebibliography}{10}

\bibitem{encyclopedia_mathematics_2016}
Encyclopedia of {{Mathematics}}, April 2016.

\bibitem{akhiezer_questions_1962}
Naum~I. Ahiezer and Mark~G. Kre{\u \i}n.
\newblock {\em Some Questions in the Theory of Moments,}.
\newblock {American Mathematical Society}, Providence, 1962.

\bibitem{andersson_general_2015}
Fredrik Andersson and Marcus Carlsson.
\newblock On {{General Domain Truncated Correlation}} and {{Convolution
  Operators}} with {{Finite Rank}}.
\newblock {\em Integral Equations and Operator Theory}, 82(3):339--370, 2015.

\bibitem{andersson_kronecker_2015}
Fredrik Andersson and Marcus Carlsson.
\newblock On the structure of positive semi-definite finite rank general domain
  {{Hankel}} and {{Toeplitz}} operators in several variables.
\newblock {\em Complex Analysis and Operator Theory}, to appear, 2016.

\bibitem{And10}
Fredrik Andersson, Marcus Carlsson, and Maarten~V. de~Hoop.
\newblock Nonlinear approximation of functions in two dimensions by sums of
  wave packets.
\newblock {\em Appl. Comput. Harmon. Anal.}, 29(2):198--213, 2010.

\bibitem{baker_pade_1996}
George~A. Baker and Peter Graves-Morris.
\newblock {\em {Pad{\'e} Approximants}}.
\newblock {Cambridge University Press}, Cambridge England ; New York, 2nd
  edition, 1996.

\bibitem{barachart_realisation_1985}
Laurent Barachart.
\newblock {Sur la r\'ealisation de Nerode des syst\`emes multi-indiciels}.
\newblock {\em {C. R. Acad. Sc. Paris}}, 301:715--718, 1984.

\bibitem{batenkov_accuracy_2013}
Dmitry Batenkov and Yosef Yomdin.
\newblock On the accuracy of solving confluent {{Prony}} systems.
\newblock {\em SIAM Journal on Applied Mathematics}, 73(1):134--154, 2013.

\bibitem{beckermann_numerical_2007}
Bernhard Beckermann, Gene~H. Golub, and George Labahn.
\newblock On the numerical condition of a generalized {{Hankel}} eigenvalue
  problem.
\newblock {\em Numerische Mathematik}, 106(1):41--68, 2007.

\bibitem{beckermann_uniform_1994}
Bernhard Beckermann and George Labahn.
\newblock A uniform approach for the fast computation of matrix-type
  {{Pad{\'e}}} approximants.
\newblock {\em SIAM Journal on Matrix Analysis and Applications},
  15(3):804--823, 1994.

\bibitem{BenOrTiwari:1988}
Michael Ben-Or and Prasson Tiwari.
\newblock A deterministic algorithm for sparse multivariate polynomial
  interpolation.
\newblock In {\em Proceedings of the twentieth annual ACM symposium on Theory
  of computing}, STOC '88, pages 301--309, New York, NY, USA, 1988. ACM.

\bibitem{berlekamp_nonbinary_1968}
Elwyn~R. Berlekamp.
\newblock Nonbinary {{BCH}} decoding.
\newblock {\em IEEE Transactions on Information Theory}, 14(2):242--242, 1968.

\bibitem{bernardi_general_2013}
Alessandra Bernardi, J{\'e}r{\^o}me Brachat, Pierre Comon, and Bernard
  Mourrain.
\newblock General tensor decomposition, moment matrices and applications.
\newblock {\em Journal of Symbolic Computation}, 52:51--71, 2013.

\bibitem{berthomieu_linear_2015}
J{\'e}r{\'e}my Berthomieu, Brice Boyer, and Jean-Charles Faug{\`e}re.
\newblock Linear {{Algebra}} for {{Computing Gr{\"o}bner Bases}} of {{Linear
  Recursive Multidimensional Sequences}}.
\newblock In {\em International {{Symposium}} on {{Symolic}} and {{Algebraic
  Compution}}}, pages 61--68. {ACM Press}, 2015.

\bibitem{BeyMon05}
Gregory Beylkin and Lucas Monz{\'o}n.
\newblock On approximation of functions by exponential sums.
\newblock {\em Appl. Comput. Harmon. Anal.}, 19(1):17--48, 2005.

\bibitem{BCMT:2009:laa}
J{\'e}rome Brachat, Pierre Comon, Bernard Mourrain, and Elias Tsigaridas.
\newblock Symmetric tensor decomposition.
\newblock {\em {L}inear {A}lgebra and {A}pplications}, 433:1851--1872, 2010.

\bibitem{CandesRombergTao06}
Emmanuel~J. Cand{\`e}s, Justin Romberg, and Terence Tao.
\newblock Stable signal recovery from incomplete and inaccurate measurements.
\newblock {\em Communications on Pure and Applied Mathematics},
  59(8):1207--1223, 2006.

\bibitem{caratheodory_uber_1911}
Constantin Carath{\'e}odory and Lip{\`o}t Fej{\'e}r.
\newblock {{\"U}ber den Zusammenhang der Extremen von Harmonischen Funktionen
  mit Ihren Koeffizienten und {\"U}ber den Picard-Landauschen Satz}.
\newblock In {\em {Rendiconti del Circolo Matematico di Paler mo (1884-1940)}},
  volume~32, pages 218--239. 1911.

\bibitem{cox_ideals_2015}
David~A. Cox, John Little, and Donal O'Shea.
\newblock {\em Ideals, {{Varieties}}, and {{Algorithms}}}.
\newblock Undergraduate Texts in Mathematics. {Springer}, 1992.

\bibitem{curto_solution_1996}
Raul~E. Curto and Lawrence~A. Fialkow.
\newblock {\em Solution of the {{Truncated Complex Moment Problem}} for {{Flat
  Data}}}.
\newblock {Amer Mathematical Society}, Providence, R.I, January 1996.

\bibitem{cuyt_how_1999}
Annie Cuyt.
\newblock How well can the concept of {{Pad{\'e}}} approximant be generalized
  to the multivariate case?
\newblock {\em Journal of Computational and Applied Mathematics},
  105(1-2):25--50, 1999.

\bibitem{Prony1795}
Baron Gaspard~Riche de~Prony.
\newblock Essai exp{\'e}rimental et analytique: sur les lois de la
  dilatabilit{\'e} de fluides {\'e}lastique et sur celles de la force expansive
  de la vapeur de l'alcool, {\`a} diff{\'e}rentes temp{\'e}ratures.
\newblock {\em J. {\'E}cole Polytechnique}, 1:24--76, 1795.

\bibitem{eisunbud_commutative_1994}
David Eisunbud.
\newblock {\em Commutative Algebra: With a View toward Algebraic Geometry},
  volume 150 of {\em Graduate texts in mathematics}.
\newblock {Springer-Verlag}, 1994.

\bibitem{ElkMou07}
Mohamed Elkadi and Bernard Mourrain.
\newblock {\em Introduction \`a la r\'esolution des syst\`emes polynomiaux},
  volume~59 of {\em Math\'ematiques \& Applications}.
\newblock Springer, Berlin, 2007.

\bibitem{emsalem_geometrie_1978}
Jacques Emsalem.
\newblock G{\'e}om{\'e}trie des points {\'e}pais.
\newblock {\em Bulletin de la S.M.F.}, 106:399--416, 1978.

\bibitem{fischer_uber_1911}
Ernst Fischer.
\newblock {\"U}ber das {{Carath{\'e}odory}}'sche {{Problem}}, {{Potenzreihen}}
  mit positivem reellen {{Teil}} betreffend.
\newblock {\em Rendiconti del Circolo Matematico di Palermo (1884-1940)},
  32(1):240--256, 1911.

\bibitem{fitzpatrick_finding_1990}
Patrick Fitzpatrick and Graham~H. Norton.
\newblock Finding a basis for the characteristic ideal of an n-dimensional
  linear recurring sequence.
\newblock {\em IEEE Transactions on Information Theory}, 36(6):1480--1487,
  1990.

\bibitem{fliess_series_1970}
Michel Fliess.
\newblock {S{\'e}ries reconnaissables, rationnelles et alg{\'e}briques}.
\newblock {\em Bulletin des Sciences Math{\'e}matiques. Deuxi{\`e}me
  S{\'e}rie}, 94:231--239, 1970.

\bibitem{GiLaWe09}
Mark Giesbrecht, George Labahn, and Wen-shin Lee.
\newblock Symbolic-numeric sparse interpolation of multivariate polynomials.
\newblock {\em J. Symb. Comput.}, 44(8):943--959, August 2009.

\bibitem{golub_separable_2003}
Gene Golub and Victor Pereyra.
\newblock Separable nonlinear least squares: The variable projection method and
  its applications.
\newblock {\em Inverse Problems}, 19(2):R1--R26, 2003.

\bibitem{olshevsky_2d-extension_2010}
N.~E. Golyandina and K.~D. Usevich.
\newblock {{2D}}-{{Extension}} of {{Singular Spectrum Analysis}}: {{Algorithm}}
  and \ {{Elements}} of {{Theory}}.
\newblock In {\em Matrix {{Methods}}: {{Theory}}, {{Algorithms}} and
  {{Applications}}}, pages 449--473. {World Scientific}, April 2010.

\bibitem{graillat_new_2009}
Stef Graillat and Philippe Tr{\'e}buchet.
\newblock A new algorithm for computing certified numerical approximations of
  the roots of a zero-dimensional system.
\newblock In {\em Proceedings of the 2009 International Symposium on
  {{Symbolic}} and Algebraic Computation}, pages 167--174. {ACM}, 2009.

\bibitem{grobner_uber_1937}
Wolfgang Gr{\"o}bner.
\newblock {\"U}ber das {{Macaulaysche}} inverse {{System}} und dessen
  {{Bedeutung}} f{\"u}r die {{Theorie}} der linearen
  {{Differentialgleichungen}} mit konstanten {{Koeffizienten}}.
\newblock In {\em Abhandlungen Aus Dem {{Mathematischen Seminar}} Der
  {{Universit{\"a}t Hamburg}}}, volume~12, pages 127--132. {Springer}, 1937.

\bibitem{gu_finite_1999}
Caixing Gu.
\newblock Finite rank {{Hankel}} operators on the polydisk.
\newblock {\em Linear Algebra and its Applications}, 288:269--281, February
  1999.

\bibitem{hakopian_partial_2004}
Hakop~A. Hakopian and Mariam~G. Tonoyan.
\newblock Partial differential analogs of ordinary differential equations and
  systems.
\newblock {\em New York J. Math}, 10:89--116, 2004.

\bibitem{hormander_introduction_1990}
Lars Hormander.
\newblock {\em An {{Introduction}} to {{Complex Analysis}} in {{Several
  Variables}}}, volume~7.
\newblock {North Holland}, Amsterdam; New York; N.Y., U.S.A., 3rd edition,
  1990.

\bibitem{iarrobino_power_1999}
Anthony Iarrobino and Vassil Kanev.
\newblock {\em Power {{Sums}}, {{Gorenstein Algebras}}, and {{Determinantal
  Loci}}}.
\newblock Lecture Notes in Mathematics. {Springer}, 1999.

\bibitem{kronecker_zur_1880}
Leopold Kronecker.
\newblock Zur {{Theorie}} der {{Elimination Einer Variabeln}} aus {{Zwei
  Algebraischen Gleichungen}}.
\newblock pages 535--600., December 1880.

\bibitem{kunis_multivariate_2016}
Stefan Kunis, Thomas Peter, Tim R{\"o}mer, and Ulrich {von der Ohe}.
\newblock A multivariate generalization of {{Prony}}'s method.
\newblock {\em Linear Algebra and its Applications}, 490:31--47, February 2016.

\bibitem{lasserre_moment_2013}
Jean-Bernard Lasserre, Monique Laurent, Bernard Mourrain, Philipp Rostalski,
  and Philippe Tr{\'e}buchet.
\newblock Moment matrices, border bases and real radical computation.
\newblock {\em Journal of Symbolic Computation}, 51:63--85, 2013.

\bibitem{laurent_sums_2009}
Monique Laurent.
\newblock Sums of squares, moment matrices and optimization over polynomials.
\newblock In {\em Emerging Applications of Algebraic Geometry}, pages 157--270.
  {Springer}, 2009.

\bibitem{laurent_generalized_2009}
Monique Laurent and Bernard Mourrain.
\newblock A generalized flat extension theorem for moment matrices.
\newblock {\em Archiv der Mathematik}, 93(1):87--98, 2009.

\bibitem{LiCo:2014}
Lek-Heng Lim and Pierre Comon.
\newblock Blind multilinear identification.
\newblock {\em IEEE Transactions on Information Theory}, 60(2):1260--1280,
  2014.

\bibitem{macaulay_algebraic_1916}
Francis~S. Macaulay.
\newblock {\em The {{Algebraic Theory}} of {{Modular Systems}}}.
\newblock {Cambridge University Press}, 1916.

\bibitem{macwilliams_theory_1977}
Jessie~F. MacWilliams and Neil J.~A. Sloane.
\newblock {\em The {{Theory}} of {{Error}}-{{Correcting Codes}}, {{Volume}}
  16}.
\newblock {North Holland Publishing Co.}, 1977.

\bibitem{malgrange_existence_1956}
Bernard Malgrange.
\newblock Existence et approximation des solutions des {\'e}quations aux
  d{\'e}riv{\'e}es partielles et des {\'e}quations de convolution.
\newblock {\em Annales de l'institut Fourier}, 6:271--355, 1956.

\bibitem{massey_shift-register_1969}
James Massey.
\newblock Shift-register synthesis and {{BCH}} decoding.
\newblock {\em IEEE transactions on Information Theory}, 15(1):122--127, 1969.

\bibitem{mourrain_isolated_1996}
Bernard Mourrain.
\newblock Isolated points, duality and residues.
\newblock {\em J. of Pure and Applied Algebra}, 117\&118:469--493, 1996.

\bibitem{m-99-nf}
Bernard Mourrain.
\newblock A new criterion for normal form algorithms.
\newblock In M.~Fossorier, H.~Imai, Shu Lin, and A.~Poli, editors, {\em Proc.
  AAECC}, volume 1719 of {\em LNCS}, pages 430--443. Springer, Berlin, 1999.

\bibitem{mp-jcomplexity-2000}
Bernard Mourrain and Victor~Y. Pan.
\newblock {Multivariate Polynomials, Duality, and Structured Matrices}.
\newblock {\em Journal of Complexity}, 16(1):110--180, 2000.

\bibitem{MoTr:2005:issac}
Bernard Mourrain and Philippe Tr{\'e}buchet.
\newblock Generalized normal forms and polynomials system solving.
\newblock In M.~Kauers, editor, {\em Proc. of the International Symposium on
  Symbolic and Algebraic Computation (ISSAC'05)}, pages 253--260, 2005.

\bibitem{oberst_constructive_2001}
Ulrich Oberst and Franz Pauer.
\newblock The {{Constructive Solution}} of {{Linear Systems}} of {{Partial
  Difference}} and {{Differential Equations}} with {{Constant Coefficients}}.
\newblock {\em Multidimensional Systems and Signal Processing},
  12(3-4):253--308, 2001.

\bibitem{pedersen_basis_1999}
Paul~S. Pedersen.
\newblock Basis for {{Power Series Solutions}} to {{Systems}} of {{Linear}},
  {{Constant Coefficient Partial Differential Equations}}.
\newblock {\em Advances in Mathematics}, 141(1):155--166, 1999.

\bibitem{peller_excursion_1998}
Vladimir~V. Peller.
\newblock An excursion into the theory of {{Hankel}} operators.
\newblock {\em Holomorphic spaces (Berkeley, CA, 1995), Math. Sci. Res. Inst.
  Publ}, 33:65--120, 1998.

\bibitem{pereyra_exponential_2010}
V.~Pereyra and G.~Scherer, editors.
\newblock {\em Exponential {{Data Fitting}} and Its {{Applications}}}.
\newblock {Bentham Science Publisher}, 2012.

\bibitem{peter_generalized_2013}
Thomas Peter and Gerlind Plonka.
\newblock A generalized {{Prony}} method for reconstruction of sparse sums of
  eigenfunctions of linear operators.
\newblock {\em Inverse Problems}, 29(2):025001, 2013.

\bibitem{plonka_prony_2014}
Gerlind Plonka and Manfred Tasche.
\newblock Prony methods for recovery of structured functions.
\newblock {\em GAMM-Mitteilungen}, 37(2):239--258, 2014.

\bibitem{PotTas10}
Daniel Potts and Manfred Tasche.
\newblock Parameter estimation for exponential sums by approximate prony
  method.
\newblock {\em Signal Processing}, 90(5):1631--1642, 2010.

\bibitem{potts_parameter_2013}
Daniel Potts and Manfred Tasche.
\newblock Parameter estimation for multivariate exponential sums.
\newblock {\em Electronic Transactions on Numerical Analysis}, 40:204--224,
  2013.

\bibitem{power_finite_1982}
S.~C. Power.
\newblock Finite rank multivariable {{Hankel}} forms.
\newblock {\em Linear Algebra and its Applications}, 48:237--244, 1982.

\bibitem{riquier_les_1910}
Charles Riquier.
\newblock {\em Les syst{\`e}mes d'{\'e}quations aux d{\'e}riv{\'e}es
  partielles}, volume XXVII.
\newblock {Gauthier-Villars}, 1910.

\bibitem{rochberg_toeplitz_1987}
Richard Rochberg.
\newblock Toeplitz and {{Hankel}} operators on the {{Paley}}-{{Wiener}} space.
\newblock {\em Integral Equations and Operator Theory}, 10(2):187--235, 1987.

\bibitem{RoyKailath90}
Richard Roy and Thomas Kailath.
\newblock Signal processing part {II}.
\newblock chapter {ESPRIT-estimation of signal parameters via rotational
  invariance techniques}, pages 369--411. Springer-Verlag New York, 1990.

\bibitem{sakata_finding_1988}
Shojiro Sakata.
\newblock Finding a minimal set of linear recurring relations capable of
  generating a given finite two-dimensional array.
\newblock {\em Journal of Symbolic Computation}, 5(3):321--337, 1988.

\bibitem{sauer_pronys_2016-1}
Tomas Sauer.
\newblock Prony's method in several variables.
\newblock {\em Numerische Mathematik}, 136(2):411--438, 2017.

\bibitem{schwartz_theorie_1966}
Laurent Schwartz.
\newblock {\em {Th{\'e}orie des distributions}}.
\newblock {Editions Hermann}, Paris, 1966.

\bibitem{SwindlehurstKailath92}
A.~Lee Swindlehurst and Thomas Kailath.
\newblock {A Performance Analysis of Subspace-Based Methods in the Presence of
  Model Errors -- Part {I}: The {MUSIC} Algorithm}.
\newblock {\em IEEE Trans. on Signal Processing}, 40:1758--1774, 1992.

\bibitem{sylvester_essay_1851}
James~Joseph Sylvester.
\newblock {\em Essay on Canonical Form}.
\newblock The collected mathematical papers of J. J. Sylvester, Vol. I, Paper
  34, Cambridge University Press. 1909 (XV und 688). {G. Bell}, London, 1851.

\bibitem{gathen_modern_2013}
Joachim von~zur Gathen and J{\"u}rgen Gerhard.
\newblock {\em Modern {{Computer Algebra}}}.
\newblock {Cambridge University Press}, 3rd edition, 2013.

\bibitem{yang_vandermonde_2016}
Zai Yang, Lihua Xie, and Petre Stoica.
\newblock Vandermonde {{Decomposition}} of {{Multilevel Toeplitz Matrices With
  Application}} to {{Multidimensional Super}}-{{Resolution}}.
\newblock {\em IEEE Transactions on Information Theory}, 62(6):3685--3701, June
  2016.

\bibitem{Zippel:1979}
Richard Zippel.
\newblock Probabilistic algorithms for sparse polynomials.
\newblock In {\em Proceedings of the International Symposiumon on Symbolic and
  Algebraic Computation}, EUROSAM '79, pages 216--226, London, UK, 1979.
  Springer-Verlag.

\end{thebibliography}

}
\end{document}